\documentclass[oneside,11pt]{article}

\usepackage[utf8]{inputenc}
\usepackage[T1]{fontenc}

\usepackage{amsthm}
\usepackage{amsmath}
\usepackage{amsfonts}
\usepackage{amssymb}

\usepackage{graphicx} 
\usepackage{graphbox}
\usepackage{float}
\usepackage{booktabs}
\usepackage{multirow}
\usepackage{rotating}
\usepackage{array}
\usepackage{xcolor}
\usepackage{subcaption}
\usepackage[ruled]{algorithm2e}

\newcolumntype{C}[1]{>{\centering\arraybackslash}p{#1}}
\newcolumntype{R}[1]{>{\raggedleft\let\newline\\\arraybackslash\hspace{0pt}}m{#1}}
\newcolumntype{L}[1]{>{\raggedright\let\newline\\\arraybackslash\hspace{0pt}}m{#1}}

\usepackage{breakcites}

\usepackage{nowidow}

\usepackage{enumitem}

\usepackage[top=1.5cm,bottom=2cm,right=2.5cm,left=2.5cm]{geometry}
\usepackage{titling}

\usepackage{natbib}

\usepackage{ifplatform}

\usepackage{textcomp}
\usepackage{bbm}

\usepackage{comment}

\ifwindows
\fi

\allowdisplaybreaks

\newcommand{\lc}{\left[}
\newcommand{\rc}{\right]}
\newcommand{\lb}{\left\{}
\newcommand{\rb}{\right\}}

\newcommand{\R}{\mathbb{R}}

\newcommand{\N}{\mathbb{N}}
\newcommand{\rd}{\mathrm{d}}

\newcommand{\bx}{\boldsymbol{x}}
\newcommand{\be}{\boldsymbol{e}}

\newcommand{\by}{\boldsymbol{y}}
\newcommand{\bX}{\boldsymbol{X}}
\newcommand{\bY}{\boldsymbol{Y}}

\newcommand{\bmu}{\boldsymbol{\mu}}
\newcommand{\bu}{\boldsymbol{u}}

\newcommand{\bs}{\boldsymbol{s}}
\newcommand{\bz}{\boldsymbol{z}}

\newcommand{\zero}{\boldsymbol{0}}

\newcommand{\Ical}{\mathcal{I}}

\newcommand{\bSigma}{\boldsymbol\Sigma}

\newcommand{\bI}{\boldsymbol{I}}

\newcommand{\bW}{\boldsymbol{W}}

\newcommand{\Spp}{\mathcal{S}^{p-1}}

\newcommand{\lrp}[1]{\left(#1\right)}

\newcommand{\lrb}[1]{\left\{#1\right\}}
\newcommand{\Prob}[1]{\mathbb{P}\lb #1\rb}
\newcommand{\E}[1]{\mathbb{E}\lc #1\rc}

\newcommand{\Es}[2]{\mathbb{E}_{#2}\lc #1\rc}
\newcommand{\Vs}[2]{\mathbb{V}\mathrm{ar}_{#2}\lc #1\rc}

\newcommand{\Ebig}[1]{\mathbb{E}\big[ #1\big]}
\newcommand{\Ebigg}[1]{\mathbb{E}\bigg[ #1\bigg]}

\newcommand{\abs}[1]{\left| #1\right|}

\newcommand{\vect}[1]{\mathrm{vec}\left(#1\right)}

\newcommand{\bt}{\boldsymbol{t}}

\DeclareFontFamily{OT1}{pzc}{}
\DeclareFontShape{OT1}{pzc}{m}{it}{<-> s * [1.10] pzcmi7t}{}
\DeclareMathAlphabet{\mathpzc}{OT1}{pzc}{m}{it}

\newtheorem{theorem}{Theorem}[section]

\newtheorem{remark}{Remark}[section]
\newtheorem{proposition}{Proposition}[section]
\newtheorem{lemma}{Lemma}[section]

\newtheorem{example}{Example}[section]

\renewenvironment{proof}[1]{\textit{Proof#1.}}{\qed\\} 

\usepackage[pdftex,final,bookmarksnumbered,bookmarksopen=false,breaklinks,colorlinks]{hyperref}
\usepackage{dsfont}
\usepackage{siunitx}
\hypersetup{
	final,
	bookmarksnumbered=true,
	bookmarksopen=true,
	bookmarksopenlevel=0,
	unicode=false,
	pdftoolbar=true,
	pdfmenubar=true,
	pdffitwindow=false,
	pdftitle={On Stein's test of uniformity on the hypersphere},
	pdfdisplaydoctitle=true,
	pdftoolbar=true,
	pdfmenubar=true,
	pdflang={English},
	pdfauthor={Paul Axmann, Bruno Ebner, Eduardo Garcia-Portugues},
	pdfsubject={arXiv paper},
	pdfcreator={Paul Axmann},
	pdfproducer={Paul Axmann},
	pdfkeywords={Directional data}{Laplace-Beltrami operator}{Sobolev tests}{Uniformity}{Stein characterization},
	pdfnewwindow=true,
	breaklinks=true,
	hidelinks,
	linkcolor=black,
	citecolor=black,
	filecolor=magenta,
	urlcolor=cyan
}

\newif\ifmain
\maintrue
\newif\ifsupplement
\supplementtrue

\newif\iffigstabs
\figstabstrue

\begin{document}

\ifmain

\title{On Stein's test of uniformity on the hypersphere}
\setlength{\droptitle}{-1cm}
\predate{}%
\postdate{}%
\date{}

\author{Paul Axmann$^{1,3}$, Bruno Ebner$^{1}$, and Eduardo Garc\'{\i}a-Portugu\'es$^{2}$}
\footnotetext[1]{Institute of Stochastics, Karlsruhe Institute of Technology (Germany).}
\footnotetext[2]{Department of Statistics, Carlos III University of Madrid (Spain).}
\footnotetext[3]{Corresponding author. e-mail: \href{mailto:paul.axmann@kit.edu}{paul.axmann@kit.edu}.}
\maketitle

\begin{abstract}
 We propose a new test of uniformity on the hypersphere based on a Stein characterization associated with the Laplace--Beltrami operator. We identify a sufficient class of test functions for this characterization, linked to the moment generating function. Exploiting the operator's eigenfunctions to obtain a harmonic decomposition in terms of Gegenbauer polynomials, we show that the proposed procedure belongs to the class of Sobolev tests. We derive closed-form series representations for the asymptotic distribution of the test statistic under the null hypothesis and under fixed alternatives. To enhance power against a range of alternatives, we introduce a tuning parameter into the characterization and study its impact on rejection probabilities. We discuss data-driven strategies for selecting this parameter to maximize rejection rates for a given alternative and compare the resulting performance with that of related parametric tests. Additional numerical experiments compare the proposed test with competing Sobolev-class procedures, highlighting settings in which it offers clear advantages.
\end{abstract}
\begin{flushleft}
\end{flushleft}

\section{Introduction}
\label{sec:intro}

Stein operators offer a powerful tool for characterizing probability distributions and can be naturally applied to construct goodness-of-fit tests. The use of distributional characterizations to design goodness-of-fit procedures dates back to Yu. V. Linnik in the early 1950s \citep{L:1953LFI,L:1953LFII,N:2017}. Recent works presenting new uni- and multivariate (Euclidean) procedures exploit Stein characterizations and build $L^2$-type test statistics that quantify the magnitude of the expectation of a Stein operator evaluated over a characterizing class of functions; see \citet{Anastasiou2023}. In this paper, we adopt Stein's framework for testing uniformity on the unit (hyper)sphere $\Spp:= \{\bx \in \R^p : \|\bx\| = 1\}$, $p \geq 2$, where the relevant Stein operator is the Laplace--Beltrami operator, i.e., the spherical component of the Euclidean Laplacian.

When dealing with random points supported on $\Spp$ (i.e., directions), testing for uniformity is one of the most fundamental inferential problems, as uniformity corresponds to the absence of structure. Among other applied fields, this testing problem has found applications in astronomy (e.g., distributions of craters in Rhea; \cite{Garcia-Portugues2023}) and biology (e.g., nursing times of polar bears; \cite{Fernandez-de-Marcos2023b}). Fundamental techniques and results for directional statistics are presented in the monographs by \cite{Mardia1999a} and \cite{Ley2017a}, and recent developments are reviewed in \citet{Pewsey2021}. Stein-characterization-based approaches to testing uniformity on $\Spp$ have only recently begun to be explored \citep{xu2020a}, and their relationship to classical uniformity tests remains underdeveloped.

Formally, for an independent and identically distributed (iid) sample $\bX_1,\dots,\bX_n\sim \mathrm{P}$ on $\Spp$, $n\in\N$, the hypothesis
\begin{align*}
    \mathcal{H}_0\colon \mathrm{P}=\mathrm{Unif}(\Spp) \text{ vs. }\mathcal{H}_1\colon \mathrm{P}\neq \mathrm{Unif}(\Spp)
\end{align*}
is tested. This classical problem has been widely studied, with some of the most relevant tests being the \citet{Rayleigh1919} test based on the first moments, the \cite{Bingham1974} test based on second moments, and the \cite{Gine1975} $F_n$ test, based on an expansion in spherical harmonics; see \citet{Garcia-Portugues2020a} for a review of classical tests. More recent proposals include projection-based classes \citep{Garcia-Portugues2023, Borodavka2026} and tests based on the Poisson kernel \citep{Fernandez-de-Marcos2023b, Ding2025}. Closely related to our setting, \citet{Fernandez-de-Marcos2023b} also introduce a softmax test based on the von Mises--Fisher kernel, which involves a tuning parameter. Many of these tests belong to the class of Sobolev tests introduced in \cite{Beran1968} and \citet{Gine1975}. Within this Sobolev framework, \citet{Cutting2017} and \citet{ebner2025} derive the asymptotic null distribution as the dimension diverges to infinity, while \citet{Garcia-Portugues:Sobolev} establishes detection thresholds for rotationally symmetric alternatives. Another relevant approach is the directional kernel Stein discrepancy (dKSD) test of \cite{xu2020a}, employing Stein operators on $\Spp$ or more general extensions to Riemannian manifolds \citep{barp2022,qu2025,xu2021}.

Our characterization approach relies on Stein's method \citep{Stein1972,Chen2010}, and the proposed testing procedure follows the construction outlined in \citet[Section~5.2]{Anastasiou2023}. The main idea is to apply a suitable Stein operator $\mathcal{A}$ to a characterizing parametric function class $\mathcal{F}$, so that $\mathbb{E}[\mathcal{A}f(\bX)]=0$ holds for all $f\in\mathcal{F}$, if and only if the distribution of $\bX$ satisfies the hypothesis $\mathcal{H}_0$. This characterization naturally leads to a test statistic by replacing the expectation with its empirical counterpart. As the empirical mean consistently estimates the expectation, the statistic converges to zero under $\mathcal{H}_0$, while ``large'' deviations from zero imply rejection of $\mathcal{H}_0$ in favor of $\mathcal{H}_1$. In our setting, with a uniform target distribution, a characterization induced by the Laplace--Beltrami operator $\Delta_{\Spp}$ is practically useful: for all smooth functions $f$, we have $\mathbb{E}[\Delta_{\Spp}f(\bX)]=0$ exactly for uniformly distributed random unit vectors~$\bX$.

This operator is a special case of the second-order Stein operator $\mathcal{A}$ in \cite{fischer2024} and \cite{barp2022} for general target distributions on $\Spp$ with density $q$, which can also be found in local coordinates in \citet{xu2021}. There, the spherical Stein operator
\begin{align}
    \mathcal{A}f:=\Delta_{\Spp}f+\langle\nabla \log(q),\nabla_{\Spp}f\rangle_{\R^p},\label{eq:operator}
\end{align}
is derived through Green's first identity. In the uniform case, the density $q$ is constant, implying $\langle\nabla \log(q),\nabla_{\Spp}f\rangle_{\R^p}=0$ for all $f$, so the Stein operator simplifies to $\Delta_{\Spp}$. Up to a multiplicative constant, this Stein operator coincides with the infinitesimal generator of spherical Brownian motion \citep[Chapter 3]{H:2002}, whose stationary distribution is the uniform; it is therefore connected to the so-called generator approach to find Stein operators \citep{B:1988,B:1990}.

To construct the test statistic, we further need a parametric class of test functions that is rich enough to characterize the distribution through the Stein identity induced by the operator. Here, we choose the class $\{e^{\lambda\bt^\top\bx}:\bt\in \Spp\}$, for $\lambda>0$, connecting the test to the moment generating function (mgf) $M_{\bX} (\bt)=\Ebig{e^{\bt^\top \bX}}$, $\bt\in \R^p$. Plugging in this class of test functions leads to the following characterization of the uniform law that we prove in Appendix \ref{sec:proofs}.

\begin{proposition}\label{prop:char}
Let $p\geq 2$ and $\lambda>0$. Let $\bX$ be a random vector on $\Spp$. Then
\begin{align}
    \Ebig{\Delta_{\Spp}e^{\lambda\bt^\top\bX}}=\Delta_{\Spp}M_{\bX}(\lambda \bt)=0,\, \bt\in\Spp,\quad   \text{ if and only if }    \bX\sim \mathrm{Unif}(\Spp).
    \label{eq:Stein_char}
\end{align}
\end{proposition}

Let $\nu_{p-1}$ denote the uniform probability measure on $\Spp$, and let $L^2(\Spp)$ be the Hilbert space of square-integrable functions on $\Spp$ with scalar product $\langle h,k\rangle_{L^2(\Spp)}=\int_{\Spp}h(\bx)k(\bx)\,\rd\nu_{p-1}(\bx)$, $h,k\in L^2(\Spp)$. Based on Proposition \ref{prop:char}, we define the population discrepancy
\begin{align*}
    T(\lambda):=\Big\|\Ebig{\Delta_{\Spp}e^{\lambda\bt^\top \bX}}\Big\|^2_{L^2(\Spp)},
\end{align*}
which vanishes if and only if $\bX\sim\mathrm{Unif}(\Spp)$. Since the uniform distribution on $\Spp$ is characterized by rotational invariance with respect to all rotations about the origin, we use the unweighted $L^2$ norm over $\Spp$, which leads to a rotation-invariant test statistic. Now, given $n\in\N$ iid copies $\bX_1,\ldots,\bX_n$ of $\bX$, we approximate the expectation $\Ebig{\Delta_{\Spp}e^{\lambda\bt^\top \bX}}$ by the empirical mean to propose the test statistic
\begin{align}
    T_n(\lambda):=&\; n\bigg\|\frac{1}{n}\sum_{j=1}^n\Delta_{\Spp}e^{\lambda\bt^\top \bX_j}\bigg\|^2_{L^2(\Spp)} 
    =\frac{1}{n}\sum_{i,j=1}^n\int_{\Spp}\Delta_{\Spp}e^{\lambda\bt^\top \bX_i}\Delta_{\Spp}e^{\lambda\bt^\top \bX_j}\,\rd\nu_{p-1}(\bt).\label{eq:Tn(lambda)}
\end{align}

The remainder of the paper is organized as follows. Starting from the construction \eqref{eq:Tn(lambda)}, we derive in Section \ref{sec:stein} a closed-form representation of the proposed test statistic. Our main tool is a Gegenbauer (spherical harmonic) decomposition of the test functions, exploiting orthogonality and their connection to the Laplace--Beltrami operator. This representation is computationally convenient, allows us to establish the characterization \eqref{eq:Stein_char}, and enables us to develop an asymptotic theory both at the level of the underlying process and for $T_n(\lambda)$, carried out in Section \ref{sec:asymptotic}. Here, we treat the null case $\mathcal{H}_0$ as well as fixed alternatives and derive explicit series representations of the limit distributions, including simplified expressions for rotationally symmetric alternatives. To relate the procedure to other tests, in Section \ref{sec:connections}, we study limit regimes of the tuning parameter and establish connections to other tests of uniformity. In particular, we derive a direct link between our Stein test and arbitrary Sobolev tests, and compare our method to a kernel Stein discrepancy test. In the simulations of Section \ref{sec:sim}, we first illustrate the effect of the tuning parameter on the test statistic. We then study how $\lambda$ can be selected, considering both an oracle criterion based on the standardized mean shift and a data-driven selection method justified by the functional convergence of the process $\lambda \mapsto T_n(\lambda)$ under $\mathcal{H}_0$. Across a range of alternative distributions, we demonstrate the substantial impact of tuning on power, empirically validate the proposed selection strategies, and compare the empirical power of the test with that of other tests of uniformity. We close the paper with a discussion (Section \ref{sec:disc}). Proofs are relegated to the appendix.


\section{Spherical harmonic decomposition of the test statistic}
\label{sec:stein}

To obtain an explicit decomposition of the test statistic, we reduce integrals of zonal functions to one-dimensional integrals. For $p\ge 2$, define
\begin{align*}
    L^{2,p}:=L^2\big([-1,1], (1-u^2)^{(p-3)/2}\,\rd u\big),
    \quad
    \langle f,g\rangle_{L^{2,p}}:=\int_{-1}^1 f(u)g(u)(1-u^2)^{(p-3)/2}\,\rd u ,
\end{align*}
for $f,g\in L^{2,p}$. In dimension $p\ge 3$, the Gegenbauer polynomials $\{C_k^{(p-2)/2}\}_{k=0}^\infty$ \citep[Chapter 18]{NIST:DLMF} form an orthogonal basis of $L^{2,p}$, where $k$ denotes the degree of the polynomial, while in the case $p=2$, the Chebyshev polynomials $C_k^0(u):=\cos\big(k\arccos(u)\big)$ for all $k\in\N_0$ form an orthogonal basis of $L^{2,2}$. To unify notation, we denote the Chebyshev polynomials by $\{C_k^{0}\}_{k=0}^\infty$, as they are a limiting case of the Gegenbauer polynomials, making them a natural choice to extend the Gegenbauer construction to the circular case:
\begin{align*}
    \lim_{\nu\to0^+}\frac{1}{\nu}C_k^\nu(u)=\frac{2}{k}C^0_k(u), \quad \text{for }k\ge 1.
\end{align*}

Further, let $\omega_{m}=2\pi^{(m+1)/2}/\Gamma\big((m+1)/2\big)$ denote the Lebesgue surface measure of the unit sphere $\mathcal{S}^{m}$ for all $m\in\N_{0}$. Then, for $\bt,\bx \in\Spp$ and any zonal function $f_{\bt}(\bx)=f(\bx^\top \bt)$, the change of variables
\begin{align*}
    \int_{\Spp}f(\bx^\top\bt)\,\rd\nu_{p-1}(\bx)=\frac{\omega_{p-2}}{\omega_{p-1}}\int_{-1}^1f(u)(1-u^2)^{(p-3)/2}\,\rd u
\end{align*}
connects the spaces $L^2(\Spp)$ and $L^{2,p}$, by reducing integrals of zonal functions on $\Spp$ to integrals on $[-1,1]$. By rotational invariance of $\nu_{p-1}$, this integral is independent of $\bt$. Applying an orthogonal rotation matrix $\boldsymbol{O}$, with $\boldsymbol{O}\bt=\be_1$, the change of variables $\by=\boldsymbol{O}\bx$ yields $\int_{\Spp} f(\bx^\top \bt)\,\rd\nu_{p-1}(\bx)=\int_{\Spp} f(\by^\top \be_1)\,\rd\nu_{p-1}(\by)$.

We consider an orthogonal expansion of the function $g_\lambda(u):=e^{\lambda u}$ in $L^{2,p}$ with $u\in[-1,1]$ and $\lambda>0$ using Gegenbauer or Chebyshev polynomials,
\begin{align}
    m_{k,p}(\lambda):=\frac{\langle g_\lambda,C_k^{(p-2)/2}\rangle_{L^{2,p}}}{\| C_k^{(p-2)/2}\|^2_{L^{2,p}}},\quad
    e^{\lambda u}=\sum_{k=0}^\infty m_{k,p}(\lambda)C_k^{(p-2)/2}(u),\quad u\in[-1,1].\label{eq:expansion_mkp}
\end{align}

\begin{remark}\label{rem:unif.conv.}
Since $\bx\mapsto e^{\lambda \bt^\top \bx}$ is infinitely differentiable on $\Spp$, the expansion \\$e^{\lambda \bt^\top \bx}=\sum_{k=0}^\infty m_{k,p}(\lambda )C_k^{(p-2)/2}(\bt^\top \bx)$ converges uniformly on $[-1,1]$ \citep[Theorem 2]{Kalf1995}. This justifies term-wise application of the Laplace--Beltrami operator in the following derivations.
\end{remark}

To obtain a closed-form expression for $m_{k,p}(\lambda)$ when $p\ge 3$, we use the identity
\begin{align*}
    \frac{1}{h_{k,p}}\int_{-1}^1 e^{au}C_k^{(p-2)/2}(u)(1-u^2)^{(p-3)/2}\,\rd u
    = \bigg(\frac{2}{a}\bigg)^{(p-2)/2}\Gamma\bigg(\frac{p-2}{2}\bigg) \bigg(k+\frac{p-2}{2}\bigg)\Ical_{(p-2)/2+k}(a),
\end{align*}
which follows from \citet[Formula 7.321]{Zwillinger2014} using $h_{k,p}=\| C_k^{(p-2)/2}\|^2_{L^{2,p}}$ for $a\in\mathbb{C}\setminus\{0\}$. Here, $\Ical_{k}$ denotes the modified Bessel function of the first kind and order $k$. The case $p=2$ is obtained analogously, by applying \citet[18.3.1]{NIST:DLMF} and \citet[10.9.2]{NIST:DLMF}. Setting $a=\lambda>0$ yields
\begin{align}
    m_{k,p}(\lambda) =
    \left\{
    \begin{array}{ll}
        \displaystyle(2-1_{\{k=0\}})\Ical_k(\lambda), & p = 2,\\
        \displaystyle\left(\frac{2}{\lambda}\right)^{(p-2)/2}\Gamma\left(\frac{p-2}{2}\right) \left(k+\frac{p-2}{2}\right)\Ical_{(p-2)/2+k}(\lambda), & p > 2.
    \end{array}
    \right.\label{eq:m_k}
\end{align}

We introduce the constants
\begin{align}
    \gamma_{k,p}:=\left\{
    \begin{array}{ll}
        \displaystyle\frac{1+1_{\{k=0\}}}{2},& p=2,\\
        \displaystyle\frac{p-2}{2k+p-2}, & p>2,
    \end{array}\label{eq:gamma}
    \right.
\end{align}
to unify the notation. With this notation, the Funk--Hecke formula \citep[Theorem 1.2.9]{Dai2013} yields
\begin{align}
    \int_{\Spp}C_k^{(p-2)/2}(\bt^\top \bx)C_k^{(p-2)/2}(\by^\top \bt)\,\rd\nu_{p-1}(\bt)
    =\gamma_{k,p}C_k^{(p-2)/2}(\bx^\top \by)\label{eq:Gegen_FH}.
\end{align}

The series expansion \eqref{eq:expansion_mkp} is particularly convenient because, for any fixed $\bt\in \Spp$, the function $\bx\mapsto C_k^{(p-2)/2}(\bt^\top \bx)$ satisfies $\Delta_{\Spp} C_k^{(p-2)/2}(\bt^\top \bx) =(-k)(k+p-2)C_k^{(p-2)/2}(\bt^\top \bx)$; see Property 2 in \citet[Section B.2]{Dai2013}. Thus, $C_k^{(p-2)/2}(\bt^\top \bx)$ is an eigenfunction of $\Delta_{\Spp}$ with eigenvalue $(-k)(k+p-2)$.

The case $p=2$ simplifies further, since $\mathcal{S}^1$ can be parametrized by a single angular variable as $\bx=(\cos{\theta},\sin{\theta})^\top$ for $\theta\in[0,2\pi)$. In this parametrization, the Laplace--Beltrami operator is given by $\Delta_{\mathcal{S}^1}f(\cos{\theta},\sin{\theta})=\frac{\rd^{2}}{\rd\theta^{2}} f(\cos{\theta},\sin{\theta})$, for functions $f$ on $\mathcal{S}^1$, see \citet[Section~1.6.1]{Dai2013}. Consequently, for each fixed $\bt\in\mathcal{S}^1$ the function $\bx\mapsto C_k^0(\bt^\top\bx)$ satisfies $\Delta_{\mathcal{S}^1}C^0_k(\bt^\top \bx)=-k^2C^0_k(\bt^\top \bx)$, and is an eigenfunction of $\Delta_{\mathcal{S}^1}$ with eigenvalue $-k^2$.

Using the expansion~\eqref{eq:expansion_mkp}, which converges uniformly as seen in Remark \ref{rem:unif.conv.}, and the eigenfunction property above, we obtain, for $p\geq 2$,
\begin{align*}
    \Delta_{\Spp}e^{\lambda \bt^\top \bx}=&\;\Delta_{\Spp}\sum_{k=0}^\infty m_{k,p}(\lambda )C_k^{(p-2)/2}(\bt^\top \bx)\\
    =&\;\sum_{k=0}^\infty \big(m_{k,p}(\lambda)(-k)(k+p-2)\big)C_k^{(p-2)/2}(\bt^\top \bx).
\end{align*}
Since $\bx \mapsto \Delta_{\Spp}e^{\lambda \bt^\top \bx}$ is infinitely differentiable on $\Spp$, this series also converges uniformly \citep[Theorem 2]{Kalf1995}. With these observations, the test statistic takes the following series representation. Its proof, as well as all other proofs of the paper, can be found in Appendix \ref{sec:proofs}.

\begin{lemma}\label{lem:cMdk}
Let $p\geq 2$ and $\lambda>0$. Then $T_n(\lambda)$ has a harmonic decomposition of the form
\begin{align}
    T_n(\lambda)=\frac{1}{n}\sum_{i,j=1}^n\sum_{k=1}^\infty c_{k,p}(\lambda)C_k^{(p-2)/2}(\bX_i^\top \bX_j),\quad c_{k,p}(\lambda) :=\big(m_{k,p}(\lambda)k(k+p-2)\big)^2 \gamma_{k,p},\quad k\in\N,\label{eq:Tn_gegen}
\end{align}
where the coefficients are explicitly given by
\begin{align}
    c_{k,p}(\lambda) =
    \left\{
    \begin{array}{ll}
        \displaystyle2k^4\Ical_k(\lambda)^2, &  p = 2,\\
        \displaystyle 2^{p-3}\lambda^{2-p}(p-2)\left(k+\frac{p-2}{2}\right)\left(\Gamma\left(\frac{p-2}{2}\right) k(k+p-2)\Ical_{(p-2)/2+k}(\lambda)\right)^2,\!\! &  p > 2.
    \end{array}
    \right.\!\!\!\label{eq:cdkreal}
\end{align}
\end{lemma}

Lemma \ref{lem:cMdk} provides a closed-form series representation for $T_n(\lambda)$, and shows that $T_n(\lambda)$ belongs to the class of Sobolev tests in the sense of \citet{Gine1975}. For more details, see Section~\ref{subsec:connect_sobolev}. In practice, it is sufficient to consider the truncated series $T_{n,K}(\lambda)=\frac{1}{n}\sum_{i,j=1}^n\sum_{k=1}^K c_{k,p}(\lambda)C_k^{(p-2)/2}(\bX_i^\top \bX_j)$, $K\in\N$, since the coefficients $c_{k,p}(\lambda)$ decay super-exponentially in $k$.
\begin{proposition}\label{prop:truncation}
For every fixed $\lambda>0$ and all sufficiently large $K\in\N$,
\begin{align*}
     \abs{T_n(\lambda)-T_{n,K}(\lambda)}=O\lrp{n \lrp{\frac{e\lambda}{2K}}^{K}}
\end{align*}  
Consequently, for any sequence $(K_n)$ such that $K_n\log(K_n)-\log(n)\to \infty$ we have $\abs{T_n(\lambda) - T_{n,K_n}(\lambda)} \to 0$ as $n\to \infty$. In particular, $K_n \ge c\log n$ for some $c > 0$ is a sufficient condition.
\end{proposition}
\begin{remark}\label{rem:trunc-slutsky}
Proposition \ref{prop:truncation} shows that, for any sequence $(K_n)$ with $K_n \ge c\log n$, the truncated statistic $T_{n,K_n}(\lambda)$ is asymptotically equivalent to $T_n(\lambda)$. Hence, asymptotic distributions established in Section \ref{sec:asymptotic} for $T_n(\lambda)$ also carry over to $T_{n,K_n}(\lambda)$ by Slutsky's theorem. 
For fixed $K$, Proposition \ref{prop:truncation} provides an approximation bound, whereas asymptotic equivalence requires $K_n\to\infty$ sufficiently fast.
\end{remark}

For the analysis that follows, let $\{Y_{r,k}:r=1,\ldots,d_{k,p} \}$ denote an arbitrary orthonormal basis of the space of spherical harmonics of degree $k\geq 0$ on $\Spp$ with dimension $d_{k,p}=\binom{p+k-3}{p-2}+\binom{p+k-2}{p-2}$. Then $\{Y_{r,k}:k\in\N_0,r=1,\ldots,d_{k,p} \}$ forms an orthonormal basis of $L^2(\Spp)$ \citep[Theorem~2.2.2]{Dai2013}. Details on the explicit construction of a spherical harmonic basis are provided in \citet[Section~3]{Garcia-Portugues:Sobolev} and explicit orthonormal systems up to degree 4 are listed in \citet[Tables 1--2]{Manzotti2001}.


\section{Asymptotic results}
\label{sec:asymptotic}

To analyze the asymptotic behavior of the test statistic, we define the $L^2(\Spp)$-valued random element $W_n:\Spp\to \R$ by
\begin{align*}
    W_n(\bt):=\frac{1}{\sqrt{n}}\sum_{i=1}^n\sum_{k=1}^\infty \big(m_{k,p}(\lambda)(-k)(k+p-2)\big)C_k^{(p-2)/2}(\bt^\top \bX_i),\quad\bt\in \Spp,
\end{align*}
so that $T_n(\lambda)=\|W_n\|^2_{L^2(\Spp)}$. Obviously, $\{W_n(\bt):\bt\in\Spp\}$ is a real-valued random field indexed by $\Spp$.

\subsection{Limits under \texorpdfstring{$\mathcal{H}_0$}{H0}}

We first derive closed-form expressions of the limiting null distribution. Since $T_n(\lambda)$ can be represented as the norm of a sum of iid $L^2$-valued random elements, the central limit theorem in separable Hilbert spaces \citep[Theorem 17.29]{henze2024} and the continuous mapping theorem are used to prove the following result.

\begin{theorem} \label{thm:1_mgf}
Let $p\geq 2$ and let $\bX_1,\ldots,\bX_n$ be iid uniformly distributed random vectors on $\Spp$. Then, as $n\to \infty$, there exists a centered Gaussian random element $\mathcal{W}$ in the Hilbert space $L^2(\Spp)$ such that $W_n\xrightarrow{d}\mathcal{W}$, implying $T_n(\lambda)\xrightarrow{d}\|\mathcal{W}\|^2$. The covariance kernel of $\mathcal{W}$ is
\begin{align}
    K(\bs,\bt)=\sum_{k=1}^\infty c_{k,p}(\lambda)C_k^{(p-2)/2}(\bs^\top \bt),\quad \bs,\bt\in\Spp.\label{eq:K_H0}
\end{align}
\end{theorem}

From this result, we derive the limit distribution of the test statistic.

\begin{theorem} \label{thm:2_MGF}
For $p\geq 2$ and under $\mathcal{H}_0$ we get the asymptotic distribution
\begin{align*}
    T_n(\lambda)\xrightarrow{d}T_\infty(\lambda):=\sum_{k=1}^\infty c_{k,p}(\lambda)\gamma_{k,p} Z_{d_{k,p}}\quad \textit{for }n\to \infty,
\end{align*}
where $Z_{d_{k,p}}\sim \chi^2_{d_{k,p}}$ are independent and $\gamma_{k,p}$ is defined in \eqref{eq:gamma}.
\end{theorem}

From Theorem~\ref{thm:2_MGF} and the moments of chi-squared distributions, we derive the expectation and variance of the limiting random variable as the series $\Es{T_\infty}{\mathcal{H}_0}=\sum_{k=1}^\infty c_{k,p}(\lambda)\gamma_{k,p} d_{k,p}$ and $\Vs{T_\infty}{\mathcal{H}_0}=\sum_{k=1}^\infty 2(c_{k,p}(\lambda)\gamma_{k,p})^2d_{k,p}$.

To compute the variance of $T_n(\lambda)$ under $\mathcal{H}_0$ for a fixed $n\in\N$, we use the variance formula for $U$-statistics, and the fact that we have a centered degenerate kernel and a constant diagonal, to see that the variance takes the form:
\begin{align}
    \Vs{T_n(\lambda)}{\mathcal{H}_0}=&\;(n-1)^2\frac{2}{n(n-1)}\mathbb{E}_{\mathcal{H}_0}\bigg[\bigg(\sum_{k=1}^\infty c_{k,p}(\lambda)C_k^{(p-2)/2}(\bX^\top \bY)\bigg)^2\bigg]\\
    =&\;
    \sum_{k=1}^\infty 2\frac{n-1}{n}\big(c_{k,p}(\lambda)\gamma_{k,p}\big)^2d_{k,p}\label{eq:varH0}.
\end{align}
\subsection{Fixed alternatives}
\label{subsec:fixed_alt}

For any random vector $\bX$ on $\Spp$ with density $q\in L^2(\Spp)$ with respect to~$\nu_{p-1}$, we derive the almost sure limit of $T_n(\lambda)/n$ as well as the limit distribution of the centered test statistic, using the decomposition 
\begin{align}
    q(\bx)=\sum_{k=0}^\infty\sum_{r=1}^{d_{k,p}}\beta_{r,k}Y_{r,k}(\bx)\quad \text{in }L^2(\Spp),\quad \beta_{r,k}=\int_{\Spp}q(\bx)Y_{r,k}(\bx)\,\rd\nu_{p-1}(\bx).\label{eq:betark}
\end{align}

\begin{lemma}\label{lem:mgf_alt}
For an absolutely continuous random vector $\bX$ on $\Spp$ with density $q\in L^2(\Spp)$, let $\bt\in\R^p\setminus\{\zero\}$ and $\bs\in\Spp$. Then, in $L^2(\Spp)$,
\begin{align}
    M_{\bX}(\lambda\bt)=&\;\sum_{k=0}^\infty\sum_{r=1}^{d_{k,p}}\beta_{r,k}m_{k,p}(\lambda\|\bt\|)\gamma_{k,p}Y_{r,k}\left(\frac{\bt}{\|\bt\|}\right),\nonumber\\
    z(\bs):=&\;\Delta_{\Spp}M_{\bX}(\lambda\bs)=\sum_{k=1}^\infty\sum_{r=1}^{d_{k,p}}\beta_{r,k}m_{k,p}(\lambda)\gamma_{k,p}(-k)(k+p-2)Y_{r,k}(\bs).\label{eq:z_SH}
\end{align}
\end{lemma}

Now, by establishing the almost sure convergence $W_n/\sqrt{n}\to z$ in $L^2(\Spp)$ and applying representation \eqref{eq:z_SH}, we obtain the almost sure limit of $T_n(\lambda)/n$ for $n\to \infty$.

\begin{theorem}\label{thm:tau:mgf}
Let $p\geq 2$ and let $\bX_1,\ldots, \bX_n$ be iid copies of an absolutely continuous random vector $\bX$ on $\Spp$ with density $q\in L^2(\Spp)$. Then,
\begin{align*}
    \frac{T_n(\lambda)}{n}\xrightarrow{a.s.}\tau=&\;\|z\|_{L^2(\Spp)}^2=
    \sum_{k=1}^\infty\sum_{r=1}^{d_{k,p}}\beta_{r,k}^2c_{k,p}(\lambda)\gamma_{k,p},\quad \text{as }n\to\infty.
\end{align*}
\end{theorem}

\begin{remark}\label{rem:charac}
Theorem \ref{thm:tau:mgf} implies consistency against all absolutely continuous non-uniform distributions. By the characterization in \eqref{eq:Stein_char}, $z\equiv 0$ if and only if $\bX$ is uniformly distributed on $\Spp$ and thus $\tau>0$ for all alternative distributions. This consistency is also observed in the Gegenbauer decomposition of Theorem~\ref{thm:tau:mgf}, since $c_{k,p}(\lambda)\gamma_{k,p}>0$ implies that, for all densities $q$, $z\equiv 0$ if and only if $\beta_{r,k}=0$ for all $k\ge1$ and $r=1,\ldots,d_{k,p}$, which again implies uniformity. These observations connect to Sobolev test theory \citep[Theorem~4.4]{Gine1975} since the coefficients $c_{k,p}(\lambda)$ are positive for all $k\in\N$.
\end{remark}

With the same arguments as in Theorem~\ref{thm:tau:mgf}, we derive the expectation for fixed $n$ as a series of spherical harmonics.
\begin{remark}\label{rem:Ex:H1}
As a consequence of the proof of Theorem~\ref{thm:tau:mgf}, we obtain
\begin{align*}
    \E{T_n(\lambda)}=(n-1)\sum_{k=1}^\infty\sum_{r=1}^{d_{k,p}}\beta_{r,k}^2c_{k,p}(\lambda)\gamma_{k,p}+\sum_{k=1}^\infty c_{k,p}(\lambda)C_k^{(p-2)/2}(1).
\end{align*}
\end{remark}
Focusing further on the underlying random field, we find the limit Gaussian field (after recentering by the expectation) in analogy to Theorem~\ref{thm:1_mgf}. We introduce the notation $\Delta_{\Spp,\bt}$ to denote the Laplace--Beltrami operator on $\Spp$ acting with respect to the variable $\bt\in\Spp$.

\begin{theorem} \label{thm:3mgf}
Let $p\geq 2$ and let $\bX_1,\ldots, \bX_n$ be iid copies of an absolutely continuous random vector $\bX$ on $\Spp$ with density $q\in L^2(\Spp)$. Then, there exists a real-valued centered Gaussian random element $\mathcal{W}'$ in $L^2(\Spp)$ for which
\begin{align*}
    \left(W_n-\sqrt{n}z\right)\xrightarrow{d}\mathcal{W}'
\end{align*}
holds for $n\to \infty$, and where $\mathcal{W}'$ has the covariance kernel
\begin{align*}
    K'(\bs,\bt)=&\;\Delta_{\Spp,\bs}\Delta_{\Spp,\bt}M_{\bX}\big(\lambda(\bs+\bt)\big)-\Delta_{\Spp}M_{\bX}(\lambda\bs)\Delta_{\Spp}M_{\bX}(\lambda\bt)\\
    =&\;\sum_{k_1=1}^\infty\sum_{k_2=1}^\infty \big(m_{k_1,p}(\lambda)(-k_1)(k_1+p-2)\big)
  \big(m_{k_2,p}(\lambda)(-{k_2})({k_2}+p-2)\big)\xi_{k_1,k_2}(\bs,\bt)
  -z(\bs)z(\bt).
\end{align*}
Here, we write $\xi_{k_1,k_2}(\bs,\bt)=\Ebig{C_{k_1}^{(p-2)/2}(\bs^\top \bX)C_{k_2}^{(p-2)/2}(\bt^\top \bX)}$ for all $\bs,\bt\in\Spp$.
\end{theorem}

For applications, it is convenient to consider a finite-dimensional projection of the random field $\mathcal{W}'$ to get a covariance matrix corresponding to the kernel at a fixed set of vectors on $\Spp$.

\begin{remark}\label{rem:cov_gen}
Let $m\in \N$ and fix $\bt_1,\ldots,\bt_m\in\Spp$. For $k\in\N$, define the vectors of Gegenbauer polynomials and spherical harmonics as
\begin{align*}
    \mathbf{C}_k(\bx):=\big(C_{k}^{(p-2)/2}(\bt_1^\top \bx),\ldots,C_{k}^{(p-2)/2}(\bt_m^\top \bx)\big)^\top \text{ and } \bY_{r,k}:=\big(Y_{r,k}(\bt_1),\ldots,Y_{r,k}(\bt_m)\big)^\top.
\end{align*}
This notation allows us to write the covariance matrix $\mathbf{K}_m$ of the Gaussian limit of the random vector $\bW_n-\sqrt{n}\bz:=\big(W_n(\bt_1)-\sqrt{n}z(\bt_1),\ldots,W_n(\bt_m)-\sqrt{n}z(\bt_m)\big)^\top$, corresponding to the kernel $K'$ in Theorem~\ref{thm:3mgf}, as
\begin{align*}
    \vect{\mathbf{K}_m}=&\;\Ebigg{\bigg(\sum_{k=1}^\infty\big(m_{k,p}(\lambda)(-k)(k+p-2)\big)\bigg(\mathbf{C}_k(\bX)-\gamma_{k,p}\sum_{r=1}^{d_{k,p}}\beta_{r,k}\bY_{r,k}\bigg)\bigg)^{\otimes2}}\\
    =&\;\Ebigg{\bigg(\sum_{k=1}^\infty\big(m_{k,p}(\lambda)(-k)(k+p-2)\big)\mathbf{C}_k(\bX)\bigg)^{\otimes2}}-\bz^{\otimes 2},
\end{align*}
where $\bz^{\otimes 2}=\bz\otimes \bz=\vect{\bz\bz^\top}$. The entries of $\mathbf{K}_{m}$ are $(\mathbf{K}_{m})_{i,j}=K'(\bt_i,\bt_j)$.
\end{remark}

Although this representation is more practical, it cannot be expressed in closed form, as the expectation $\xi_{k_1,k_2}$ has to be evaluated at vectors $\bs,\bt$ with $\bs\neq \bt$. Restricting to the case $\bs=\bt$, the expectation can be expressed using the linearization formula \eqref{eq:linearization}, leading to a closed expression for the variance function of the random field.

\begin{remark}\label{rem:var_func}
Evaluating the variance function of $\mathcal{W}'$ in a direction $\bs\in\Spp$ with the linearization formula \eqref{eq:linearization} yields the closed expression
\begin{align*}
    \xi_{k_1,k_2}(\bs,\bs)
    =&\;\sum_{\ell=0}^{\min(k_1,k_2)}L_{k_1,k_2}^{(p)}(\ell)\gamma_{k_1+k_2-2\ell,p} \sum_{r_3=1}^{d_{k_1+k_2-2\ell,p}}\beta_{r_3,k_1+k_2-2\ell}
    Y_{r_3,k_1+k_2-2\ell}(\bs).
\end{align*}
\end{remark}

Using the limit distribution of the random field $W_n$ in Theorem~\ref{thm:3mgf}, we derive the limit distribution of the centered test statistic.

\begin{theorem} \label{thm:var:mgf}
Let $p\geq 2$ and let $\bX_1,\ldots, \bX_n$ be iid copies of a random vector $\bX$ on $\Spp$ with density $q\in L^2(\Spp)$. Then
\begin{align*}
    \sqrt{n}\left(\frac{T_n(\lambda)}{n}-\tau\right)\xrightarrow[]{d}\mathcal{N}(0,\sigma^2),
\end{align*}
with
\begin{align*}
    \sigma^2=&\;4\int_{\Spp}\int_{\Spp}K'(\bs,\bt){z(\bs)}{z(\bt)}\,\rd\nu_{p-1}(\bs)\,\rd\nu_{p-1}(\bt)\\
    =&\;4\Ebigg{\bigg(\sum_{k=1}^\infty\sum_{r=1}^{d_{k,p}} \gamma_{k,p}c_{k,p}(\lambda) \beta_{r,k} \big(Y_{r,k}(\bX)-\beta_{r,k}\big)\bigg)^2}.
\end{align*}
\end{theorem}

\subsection{Rotationally symmetric alternatives}
\label{sec:rot-inv}

We specialize the general alternative theory to the important class of rotationally symmetric alternatives about a fixed direction $\bmu\in\Spp$. The key advantage of rotational symmetry is that it allows for simplifications of the spherical harmonic decomposition. For a zonal density $q$, there is an angular function $g:[-1,1]\to\R$ so that we find the Gegenbauer decomposition,
\begin{align*}
    q(\bx)=g(\bmu^\top \bx)=\sum_{k=0}^\infty \beta_k C_k^{(p-2)/2}(\bmu^\top \bx),\quad
    \beta_{k}=\frac{1}{h_{k,p}}\int_{-1}^1 g(u)C_k^{(p-2)/2}(u)(1-u^2)^{(p-3)/2}\,\rd u.
\end{align*}
As an example, we explicitly derive the coefficients $\beta_k$ in closed form for the von Mises--Fisher (vMF) distribution.

\begin{example}\label{ex:betavMF}
Let $\kappa>0$ and $\bmu\in\Spp$, and denote by $f_{\mathrm{vMF}}(\cdot;\bmu,\kappa)$ the density of the von Mises--Fisher distribution $\mathrm{vMF}(\bmu,\kappa)$ with respect to $\nu_{p-1}$, so
\begin{align*}
    f_\mathrm{vMF}(\bx;\bmu,\kappa) =\frac{\kappa^{(p-2)/2}\omega_{p-1}}{(2\pi)^{p/2}\Ical_{(p-2)/2}(\kappa)}e^{\kappa\bmu^\top\bx}, \quad \text{for all }\bx\in\Spp.
\end{align*}
Combining \eqref{eq:m_k} and \eqref{eq:expansion_mkp}, we obtain the decomposition
\begin{align}
    f_\mathrm{vMF}(\bx;\bmu,\kappa) =&\;\frac{\kappa^{(p-2)/2}\omega_{p-1}}{(2\pi)^{p/2}\Ical_{(p-2)/2}(\kappa)}\sum_{k=0}^\infty m_{k,p}(\kappa)C_k^{(p-2)/2}(\bmu^\top\bx),\;\;\\
    \beta_{k}=&\;\frac{\kappa^{(p-2)/2}\omega_{p-1}}{(2\pi)^{p/2}\Ical_{(p-2)/2}(\kappa)}m_{k,p}(\kappa),\, k\in\N_0.\label{eq:beta_k-vMF}
\end{align}
\end{example}

The results in Lemma~\ref{lem:mgf_alt} and Theorem~\ref{thm:tau:mgf} simplify by exploiting the Gegenbauer decomposition.

\begin{remark}\label{rem:rot_sym_z_mgf}
Under the assumption of rotational symmetry, we find with the same arguments used in the proof of Lemma~\ref{lem:mgf_alt} that
\begin{align}
    z(\bs)=\sum_{k=1}^\infty\beta_{k}m_{k,p}(\lambda)\gamma_{k,p}(-k)(k+p-2)C_k^{(p-2)/2}(\bmu^\top\bs),\quad \bs\in\Spp.\label{eq:Z1}
\end{align}
The limit $\tau$ as defined in Theorem~\ref{thm:tau:mgf} simplifies to
\begin{align*}
    \frac{T_n(\lambda)}{n}\xrightarrow{a.s.}\tau
    =\sum_{k=1}^\infty(\beta_{k}\gamma_{k,p})^2c_{k,p}(\lambda)C_k^{(p-2)/2}(1).
\end{align*}
Here, the factor $\gamma_{k,p}C_k^{(p-2)/2}(1)$ arises from taking the integral with respect to $\nu_{p-1}$ of $C_k^{(p-2)/2}(\bmu^\top\bt)^2$, via the Funk--Hecke formula and exploiting the orthogonality of the Gegenbauer polynomials \eqref{eq:Gegen_FH}.
\end{remark}

Further, the results for the random field $W_n$ simplify. A key advantage in these remarks is that the spherical harmonic coefficients $\beta_k$ are determined in explicit form for alternatives such as the $\mathrm{vMF}$ distribution, so the asymptotic distribution is available without numerically approximating the coefficients $\beta_{r,k}$.

\begin{remark}\label{rem:cov_rot_sym}
With the same notation as in Remark~\ref{rem:cov_gen} we write the covariance matrix $\mathbf{K}_m$, corresponding to the kernel $K'$ in Theorem~\ref{thm:3mgf}, as
\begin{align}
\vect{\mathbf{K}_m}=&\;\Ebigg{\bigg(\sum_{k=1}^\infty\big(m_{k,p}(\lambda)(-k)(k+p-2)\big)\big(\mathbf{C}_k(\bX)-\gamma_{k,p}\beta_{k}\mathbf{C}_k(\bmu)\big)\bigg)^{\otimes2}}.\label{eq:vecKm}
\end{align}
More generally, for two fixed vectors $\bs,\bt\in\Spp$, the kernel is expressed as
\begin{align*}
    K'(\bs,\bt)=&\;\sum_{k_1=1}^\infty\sum_{k_2=1}^\infty \big(m_{k_1,p}(\lambda)(-k_1)(k_1+p-2)\big) \big(m_{k_2,p}(\lambda)(-{k_2})({k_2}+p-2)\big)\\
    &\times\left(\xi_{k_1,k_2}(\bs,\bt)-\gamma_{k_1,p}\gamma_{k_2,p}\beta_{k_1}\beta_{k_2}C_{k_1}^{(p-2)/2}(\bmu^\top\bs)C_{k_2}^{(p-2)/2}(\bmu^\top\bt)\right).
\end{align*}
\end{remark}

\begin{remark}\label{rem:var_rot}
In this setting, we simplify the expression of the variance from Theorem~\ref{thm:var:mgf}, similar to Remark~\ref{rem:cov_rot_sym}, and get
\begin{align*}
    \sigma^2=&\;4\Ebigg{\bigg(\sum_{k=1}^\infty\gamma_{k,p}c_{k,p}(\lambda) \beta_{k}
    \left(C^{(p-2)/2}_{k}(\bmu^\top\bX)-\beta_{k}\gamma_{k,p}C^{(p-2)/2}_{k}(1)\right)\bigg)^2}\\
    =&\;4\bigg(\sum_{k_1=1}^\infty\sum_{k_2=1}^\infty \gamma_{k_1,p}c_{k_1,p}(\lambda)\beta_{k_1}\gamma_{k_2,p}c_{k_2,p}(\lambda)\beta_{k_2}
    \sum_{\ell=0}^{\min(k_1,k_2)} \beta_{k_1+k_2-2\ell}L_{k_1,k_2}^{(p)}(\ell)
    \gamma_{k_1+k_2-2\ell,p}C_{k_1+k_2-2\ell}^{(p-2)/2}(1)\\
    &-\bigg(\sum_{k=1}^\infty \lrp{\beta_{k}\gamma_{k,p}}^2c_{k,p}(\lambda)C^{(p-2)/2}_{k}(1)\bigg)^2\bigg).
\end{align*}
Here, the last equality follows from applying the linearization formula \eqref{eq:linearization} to the polynomials\\ $C^{(p-2)/2}_{k}(\bmu^\top\bX)$, yielding a closed-form expression. This expression is derived in the proof of Theorem \ref{thm:var:mgf} in Appendix~\ref{sec:proofs}.
\end{remark}

\subsection{Functional convergence}
\label{subsec:func_conv}

The previous asymptotic results were stated for fixed values of $\lambda$. To justify a procedure that optimizes over $\lambda$ introduced in Section \ref{sec:parameter}, we now consider the statistic as a stochastic process on a compact interval
$[a,b]\subset(0,\infty)$.

\begin{proposition}\label{prop:H0_func}
Let $[a,b]\subset(0,\infty)$ be compact. Under $\mathcal{H}_0$, the process $T_n=(T_n(\lambda))_{\lambda\in[a,b]}$ converges weakly in $(C([a,b]),\|\cdot\|_\infty)$ to the continuous process
$T_\infty=(T_\infty(\lambda))_{\lambda\in[a,b]}$.
\end{proposition}
As an immediate consequence of Proposition \ref{prop:H0_func}, the standardized
process
\begin{align}
    q_n(\lambda):=\big(T_n(\lambda)-\Es{T_n(\lambda)}{\mathcal{H}_0}\big)/\allowbreak\sqrt{\Vs{T_n(\lambda)}{\mathcal{H}_0}},\quad \lambda\in[a,b]\label{eq:q_n_def}
\end{align}

also converges weakly in $(C([a,b]),\|\cdot\|_\infty)$ to the corresponding
limit process $q_\infty$, by the continuous mapping theorem. Since the map
$f\mapsto \sup_{\lambda\in[a,b]} f(\lambda)$ is continuous on $C([a,b])$, it
follows that
\begin{align}
\sup_{\lambda\in[a,b]} q_n(\lambda)\xrightarrow{d}\sup_{\lambda\in[a,b]} q_\infty(\lambda)=\sup_{\lambda\in[a,b]} \frac{\sum_{k=1}^\infty c_{k,p}(\lambda)\gamma_{k,p} \lrp{Z_{d_{k,p}}-d_{k,p}}}{\sqrt{\sum_{k=1}^\infty 2\big(c_{k,p}(\lambda)\gamma_{k,p}\big)^2d_{k,p}}}\quad \text{as}\quad n\to\infty,\label{eq:limit:max_q_n}
\end{align}
where $Z_{d_{k,p}}\sim \chi^2_{d_{k,p}}$ are independent.

\section{Connections to other tests}
\label{sec:connections}

\subsection{Limit behavior of the test for \texorpdfstring{$\lambda\to 0$}{lambda -> 0} and \texorpdfstring{$\lambda\to \infty$}{lambda -> ∞}}

The power of the test based on $T_n(\lambda)$ is sensitive to different choices of $\lambda$, see Figure \ref{fig:power_lambda}. In the following proposition, we analyze the limit behavior of the test statistic for extreme values of $\lambda$ and fixed sample size $n$.

\begin{proposition}\label{prop:limit}
Fix $n\geq 2$ and $\bX_1,\ldots,\bX_n$ iid on $\Spp$. As $\lambda\to 0$ or $\lambda\to\infty$ the rejection rule based on $T_n(\lambda)$ is asymptotically equivalent to
\begin{enumerate}
    \item the \cite{Rayleigh1919} test for $\lambda\to0$, since $\lim_{\lambda\to 0}\lambda^{-2}T_n(\lambda)\propto\frac{1}{n}\sum_{i,j=1}^n \bX_i^\top \bX_j$;
    \item the \cite{Cai2013} test for $\lambda\to\infty$, since, for $D_n(\lambda):=\frac{1}{n}\sum_{j=1}^n\big\|\Delta_{\Spp}e^{\lambda\bt^\top\bX_j}\big\|^2_{L^2(\Spp)}$,
    \begin{align*}
        \lim_{\lambda\to \infty}\lambda^{-1}\log \big(T_n(\lambda)-D_n(\lambda)\big)=\max_{1\leq i<j\leq n} \|\bX_i+\bX_j\|.
    \end{align*}
\end{enumerate}
\end{proposition}

The Rayleigh limit concentrates all weight on the first-order component and is consequently non-omnibus consistent. The maximum-type limit reduces to a single extreme inner product, so it cannot be expressed as a $V$- or $U$-statistic of the form \eqref{eq:Tn_gegen}. In contrast, for any fixed $\lambda>0$, the representation in \eqref{eq:cdkreal} assigns positive weight to all orders $k$.

This limit behavior in $\lambda$ coincides with the behavior observed in the softmax test ($S_n(\kappa)$) introduced in \cite{Fernandez-de-Marcos2023b}. For any fixed $\lambda=\kappa\in(0,\infty)$, the Gegenbauer coefficients of $T_n(\lambda)$ and $S_n(\lambda)$ differ by a factor of $m_{k,p}(\lambda)\big(k(k+p-2)\big)^2\gamma_{k,p}$. Since this factor decays rapidly as $k\to\infty$ for fixed $\lambda$, $T_n(\lambda)$ places the majority of its weight on a smaller range of indices $k$ than the softmax test. 

\begin{remark}
The construction can be extended to imaginary arguments $i\lambda$. The resulting coefficients $c_{k,p}(i\lambda)$ are obtained from the coefficients $c_{k,p}(\lambda)$ by replacing the modified Bessel function of the first kind $\Ical$ with the Bessel function of the first kind $\mathcal{J}$, reflecting the oscillating structure of the characteristic function in contrast to the exponential growth of the mgf.
\end{remark}

\subsection{Connections to Sobolev tests}
\label{subsec:connect_sobolev}

A very rich family of tests of uniformity on $\Spp$ is given by the Sobolev tests. The equivalent harmonic and $L^2$ representations illustrate the connections to our construction.

\begin{remark}\label{rem:Sobolev}
Let $\bX_1,\ldots,\bX_n$ be iid on $\Spp$ and define $\theta_{i,j}=\arccos(\bX_i^\top\bX_j)$. For non-negative sequences $(w_{k,p})_{k\ge 1}$ with $\sum_{k=1}^\infty w_{k,p}d_{k,p}<\infty$, the class of Sobolev test statistics by \citet{Beran1968}, \citet{Gine1975}, and \citet{Prentice1978} has the form
\begin{align*}
    S_{n,p}(\{w_{k,p}\})=\frac{1}{n}\sum_{i,j=1}^n\psi(\theta_{i,j}),\quad \psi(\theta)=\sum_{k=1}^\infty \frac{w_{k,p}}{\gamma_{k,p}}C_k^{(p-2)/2}(\cos{\theta}).
\end{align*}
\end{remark}

With representation \eqref{eq:Tn_gegen}, it becomes clear that $T_n(\lambda)$ is a member of the class of Sobolev test statistics, since for $\cos{\theta_{i,j}}=\bX_i^\top\bX_j$ we see that the Sobolev weights are given by $w^{\lambda}_{k,p}=\gamma_{k,p}c_{k,p}(\lambda)$.

A different representation of Sobolev tests is based on the $L^2$ norm of an angular function $g$ \citep{Beran1968,Gine1975}. Here, the corresponding Sobolev test, which is the asymptotically and locally most powerful rotation-invariant test for testing $\mathcal{H}_0$ against local alternatives of the form $(1-\kappa)+\kappa g(\cdot^\top\bmu)$, as $\kappa\to 0$, is given by
\begin{align}
    S_{n,p}(\{w_{k,p}\})=&\;\frac{1}{n}\bigg\|\sum_{i=1}^n g(\bX_i^\top \cdot)-n\bigg\|^2_{L^2(\Spp)},\\ g(z):=&\;1+\sum_{k=1}^\infty \frac{\sqrt{w_{k,p}}}{\gamma_{k,p}}C_k^{(p-2)/2}(z),\quad z\in[-1,1]. \label{eq:sob_L^2}
\end{align}

We now obtain representations of general Sobolev tests as $L^2$-Stein tests indexed by function classes $\{f_{\bt}:\bt\in\Spp\}$ more general than the exponential class. Consider a Sobolev test statistic with kernel $\psi(\theta)=\sum _{k=1}^\infty b_{k,p} C_{k}^{(p-2)/2}(\cos\theta)$, where $p\geq 2$ and $b_{k,p}\geq 0$. For the function class $\{f_{\bt}:\bt\in\Spp\}$ defined below, the statistic $S_{n,p}(\{w_{k,p}\})$ admits the representation
\begin{align*}
    S_{n,p}(\{w_{k,p}\})=\frac{1}{n}\bigg\|\sum_{i=1}^n\Delta_{\Spp}f_{\bt}(\bX_i)\bigg\|^2_{L^2(\Spp)},\quad f_{\bt}(\bx)=\sum _{k=1}^\infty\frac{\sqrt{b_{k,p}}}{k(k+p-2)\sqrt{\gamma_{k,p}}}C_k^{(p-2)/2}(\bt^\top\bx).
\end{align*}
For the angular function $g$ in \eqref{eq:sob_L^2}, we see $g(\bx^\top\bt)=1-\Delta_{\Spp}f_{\bt}(\bx)$, using the eigenfunction relation of Gegenbauer polynomials for $\Delta_{\Spp}$.

\subsection{Connections to the \texorpdfstring{$\mathrm{dKSD}^2_{(2)}$}{dKSD2(2)} test}
\label{sec:dKSD}

There are several ways to define a test statistic using a Stein operator. To contrast the proposed $L^2$-Stein approach, we consider a directional kernel Stein discrepancy test built from the same operator and kernel and compare the resulting structures. The application of a kernel Stein discrepancy (KSD) in a directional setting has been considered in \cite{xu2020a}. For the iid random vectors $\bX_1,\ldots,\bX_n$ with unknown density $q$ and target density $d$, with Stein operator $\mathcal{A}_d$, the $\mathrm{dKSD}$ $V$-statistic takes the form
\begin{align}
    \mathrm{dKSD}^2_{(2)}=\frac{1}{n^2}\sum_{i,j=1}^nh_d(\bX_i,\bX_j),\quad \mathrm{where}\quad h_d(\bx,\by)=\langle\mathcal{A}_dk(\bx,\cdot),\mathcal{A}_dk(\by,\cdot)\rangle_{\mathcal{H}},\label{eq:dKSD2}
\end{align}
see \citet[Equation (13)]{xu2020a}. While \cite{xu2020a} uses a first-order Stein operator $\mathcal{A}$, here we consider the second-order Stein operator defined in \eqref{eq:operator}, which we denote by the subscript $(2)$ in \eqref{eq:dKSD2}. In \cite{xu2021}, a version of KSD using a second-order Stein operator in local coordinates is discussed in a more general setting on manifolds with empty boundary.

Considering a KSD construction on the sphere for the uniform target distribution, we obtain $\mathcal{A}_d=\Delta_{\Spp}$. To connect this construction to our $L^2$-Stein test, we fix the kernel to be the von Mises--Fisher kernel $ k(\bx,\by)=e^{\lambda \bx ^\top \by}$ with $\lambda>0$, to see
\begin{align*}
    h_d(\bx,\by)=\Delta_{\Spp,\bx}\Delta_{\Spp,\by}k(\bx,\by)=\sum_{k=1}^\infty c_{k,p}^{\mathrm{dKSD}}(\lambda) C_k^{(p-2)/2}(\bx^\top \by).
\end{align*}
Here, we use the reproducing kernel property and the Gegenbauer expansion in \eqref{eq:m_k}, to obtain $c_{k,p}^{\mathrm{dKSD}}(\lambda):=m_{k,p}(\lambda)\big(k(k+p-2)\big)^2$. This representation helps highlight the difference between the two constructions and allows for direct application of our asymptotic results from Section~\ref{sec:asymptotic} to $n\mathrm{dKSD}^2_{(2)}$, after replacing the coefficients $c_{k,p}(\lambda)$ by $c_{k,p}^{\mathrm{dKSD}}(\lambda)$. Hence, we provide a direct method to compute the asymptotic distribution of the test statistic, both under $\mathcal{H}_0$ and under fixed alternatives, incorporate the perspective of the underlying random field to analyze the test, and show that the test belongs to the class of Sobolev tests.

The coefficients of the $L^2$-Stein and $\mathrm{dKSD}^2_{(2)}$ tests differ by a factor of $c_{k,p}(\lambda)/c_{k,p}^{\mathrm{dKSD}}(\lambda)=\gamma_{k,p}m_{k,p}(\lambda)$. To illustrate this difference, we plot standardized versions of the functions $k\mapsto c_{k,p}(\lambda)$ and $k\mapsto c^{\mathrm{dKSD}}_{k,p}(\lambda)$ in Figure~\ref{fig:coef}. The weight of the $L^2$-Stein test is more concentrated on a narrow set of indices compared to the $\mathrm{dKSD}^2_{(2)}$ test, while the concentration parameter $\lambda$ has a similar effect in both approaches, shifting the weight of the tests to Gegenbauer polynomials of higher order as $\lambda$ increases.
\begin{figure}[h!]
    \vspace*{-0.75cm}
    \centering
    \begin{subfigure}{0.4\linewidth}
        \includegraphics[width=\textwidth]{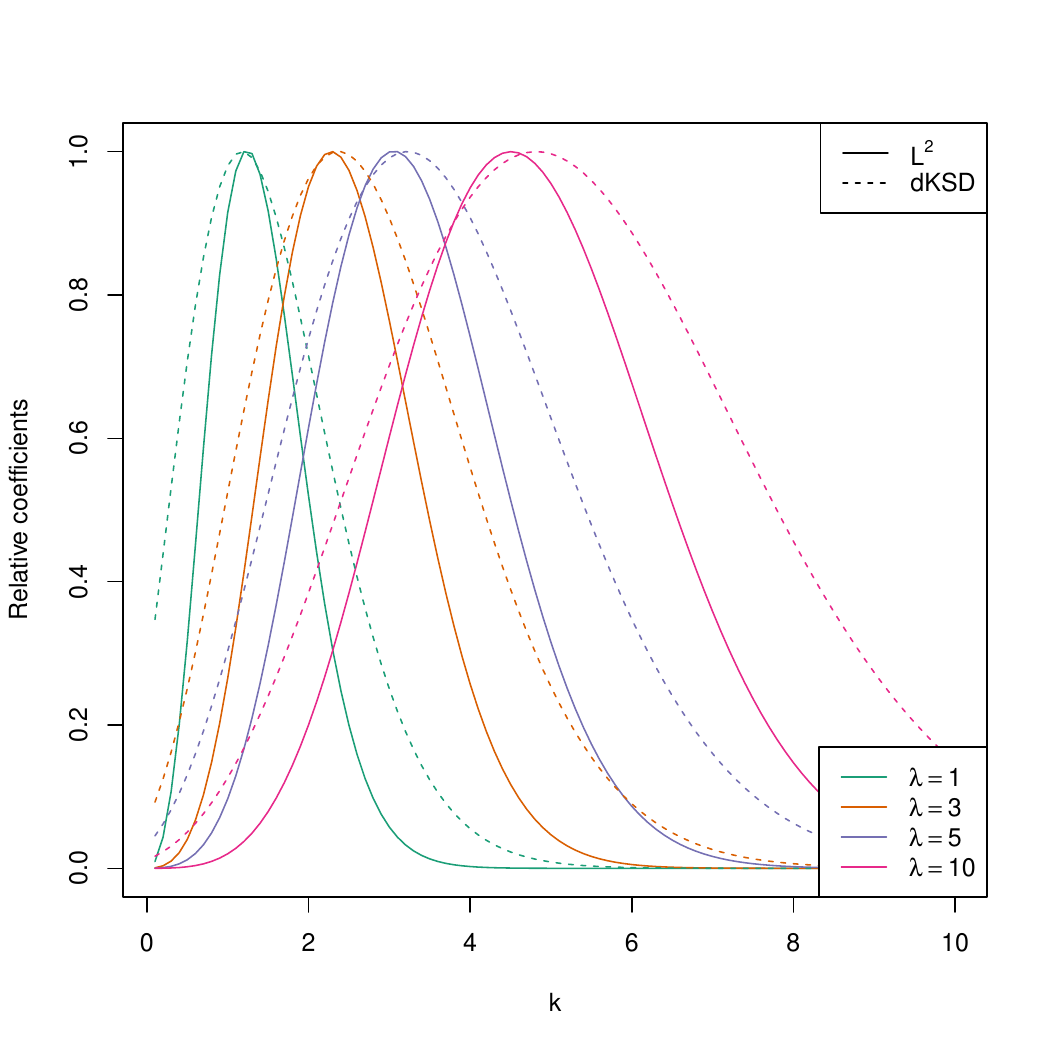}%
        \vspace*{-0.25cm}
        \caption{$p=3$}
    \end{subfigure}%
    \begin{subfigure}{0.4\linewidth}
        \includegraphics[width=\textwidth]{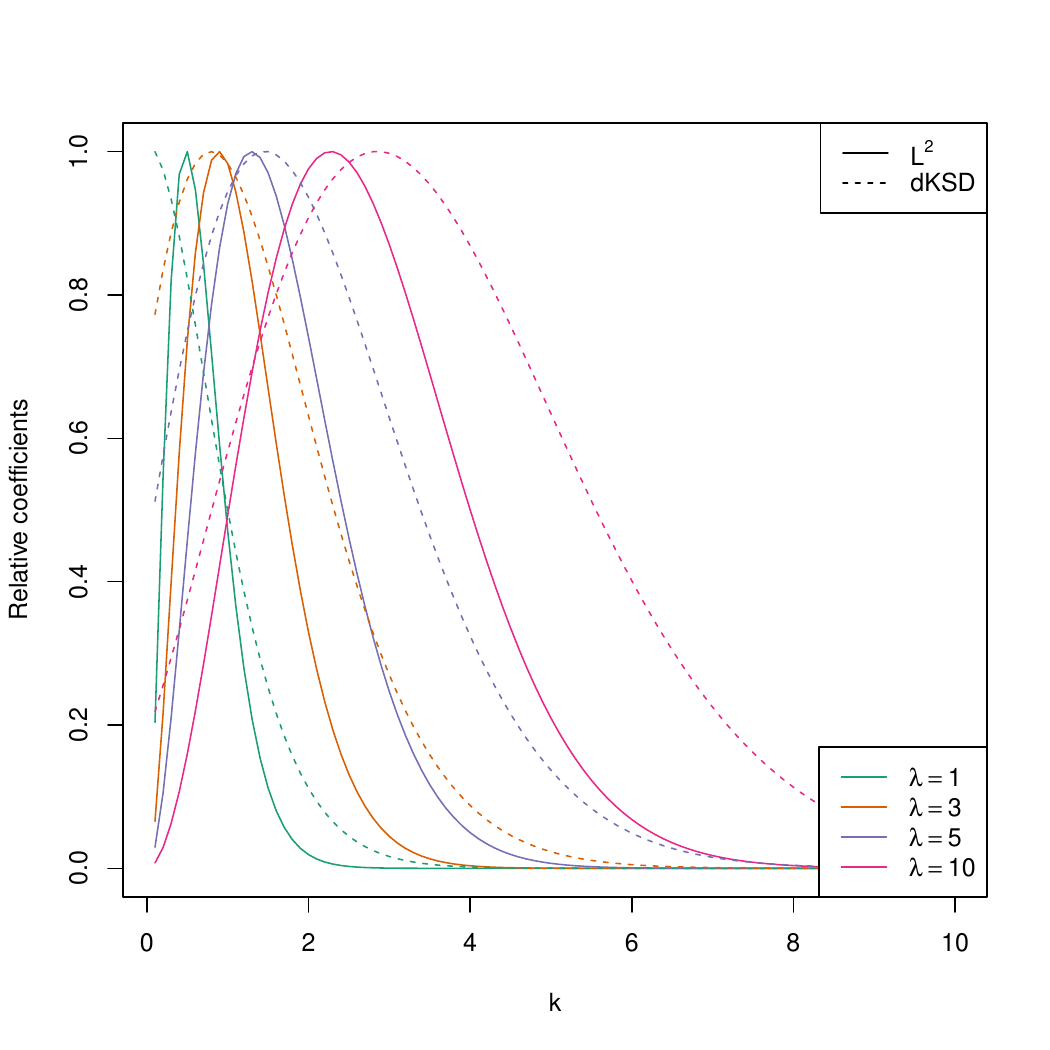}%
        \vspace*{-0.25cm}
        \caption{$p=10$}
    \end{subfigure}
     \caption{\small Relative coefficients $k\mapsto c_{k,p}(\lambda)$ and $k\mapsto c^{\mathrm{dKSD}}_{k,p}(\lambda)$ for dimensions $p=3$ and $p=10$, for the $L^2$-Stein test (solid lines) and the $\mathrm{dKSD}$ test (dashed lines). For each choice of $\lambda$, the coefficients are standardized by their maximum. To illustrate the effect, we plot the continuous mappings in $k$.}
      \label{fig:coef}
\end{figure}

\section{Numerical experiments}
\label{sec:sim}

\subsection{Visualization of covariance under alternatives}

In this section, we visualize the structure of the limiting Gaussian processes obtained in Theorems \ref{thm:1_mgf} and \ref{thm:3mgf}. These are, respectively, $\mathcal{W}$ and $\mathcal{W}'$, the limits of the empirical processes $W_n-\sqrt{n}z$. To do so, we explore the shape of: (\textit{i}) the centering $\bs\mapsto \sqrt{n}z(\bs)$ under a fixed alternative ($z(\bs)\equiv0$ under $\mathcal{H}_0$); (\textit{ii}) the null correlation kernel $\bs\mapsto\rho(\bs,\bt):=K(\bs,\bt)/\sqrt{K(\bs,\bs)K(\bt,\bt)}$; and (\textit{iii}) the fixed-alternative correlation kernel $\bs\mapsto\rho'(\bs,\bt)$. These explorations shed light on which parts of the sphere contribute most to increasing the expectation of the test statistic under a fixed alternative, and on how the correlation structure of the random field $\mathcal{W}$ changes into that of~$\mathcal{W}'$.

To visualize the previous functions, we use the equal-area Hammer projection to map $\mathcal{S}^2$ to an elliptical projection, displaying also selected parallels and meridians. We consider the $\mathrm{vMF}(\bmu,\kappa)$ distribution as a fixed alternative to leverage the expressions \eqref{eq:Z1} and \eqref{eq:vecKm} and compute $K'(\bs,\bt)$ and $z(\bs)$ using the explicit form for the vMF coefficients in \eqref{eq:beta_k-vMF}. We set $\bmu=(0,-1,0)^\top$. For computing $K(\bs,\bt)$, we used \eqref{eq:K_H0}. The series in $z(\bs)$, $K(\bs,\bt)$, and $K'(\bs,\bt)$ were truncated to their first $100$ terms. To compute \eqref{eq:vecKm}, we used Monte Carlo with $M=10,\!000$ replicates.

Figure \ref{fig:z} shows $\bs\mapsto \sqrt{n}|z(\bs)|$, illustrating the effects that $\lambda$ and $\kappa$ have on its structure. The larger $\lambda$, the larger the relative weight of $|z(\bs)|$ near $\bs=\bmu$ (Figure \ref{fig:z:3}), with the relative weight at the antipodal region (see Figure \ref{fig:z:1}) disappearing. This effect parallels the relative concentration effect of larger $\kappa$ (Figures \ref{fig:z:4}--\ref{fig:z:6}). The value of $|z(\bs)|$ at $\bs=\bmu$ and, as a result, the value of $\|z\|_{L^2(\Spp)}^2$, increase monotonically with $\lambda$, as manifested in the increasing upper limits of the legends in Figures \ref{fig:z:1}--\ref{fig:z:3}.
\begin{figure}[h!]
    \centering
    \begin{subfigure}{0.33\textwidth}
        \includegraphics[width=\textwidth]{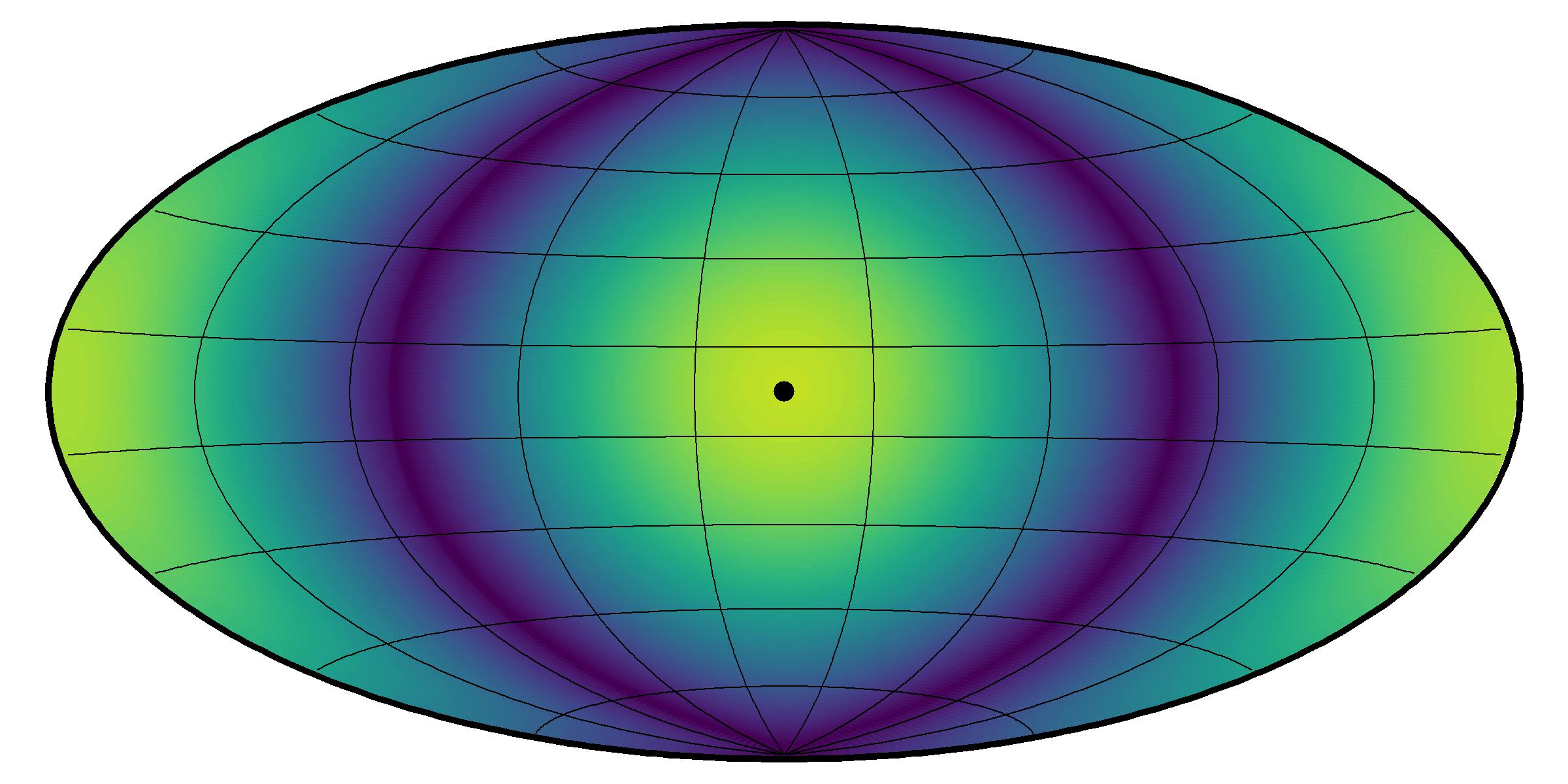}\\%
        \includegraphics[width=\textwidth]{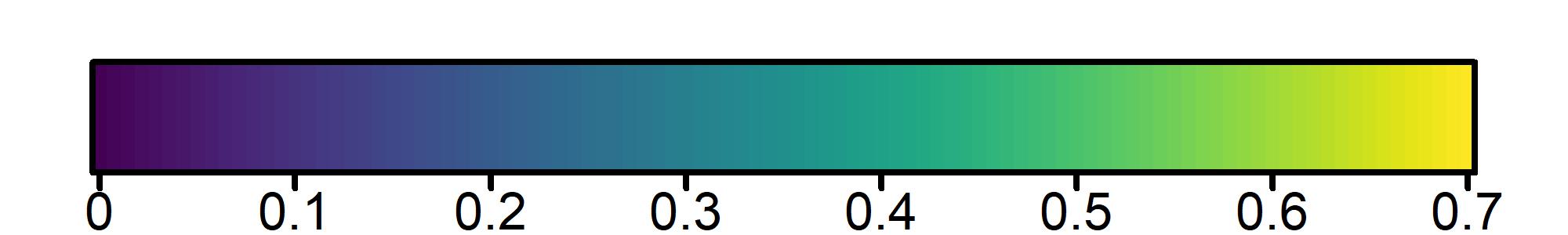}%
        \caption{$\kappa=1$, $\lambda=0.1$\label{fig:z:1}}
    \end{subfigure}%
    \begin{subfigure}{0.33\textwidth}
        \includegraphics[width=\textwidth]{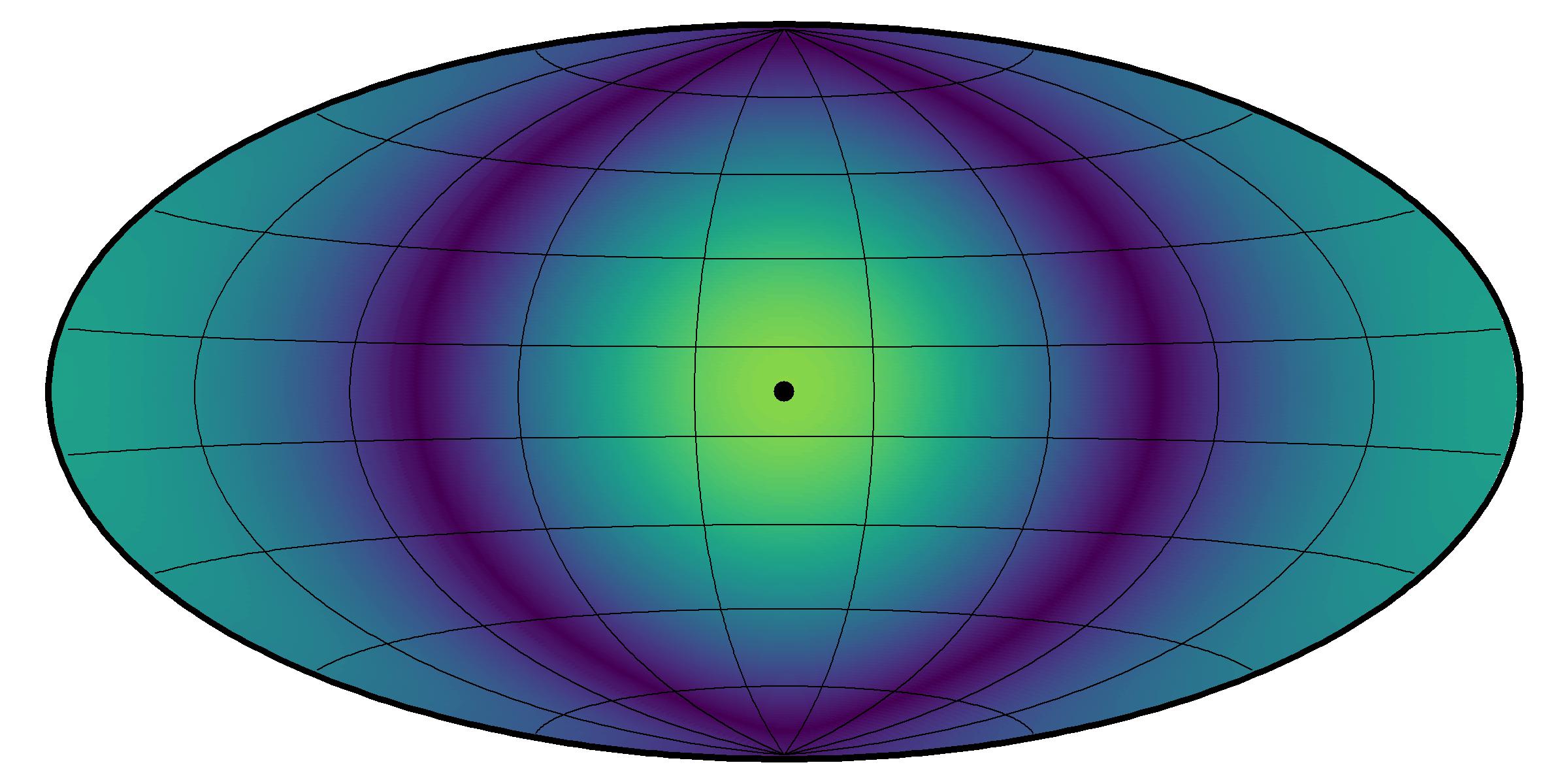}\\%
        \includegraphics[width=\textwidth]{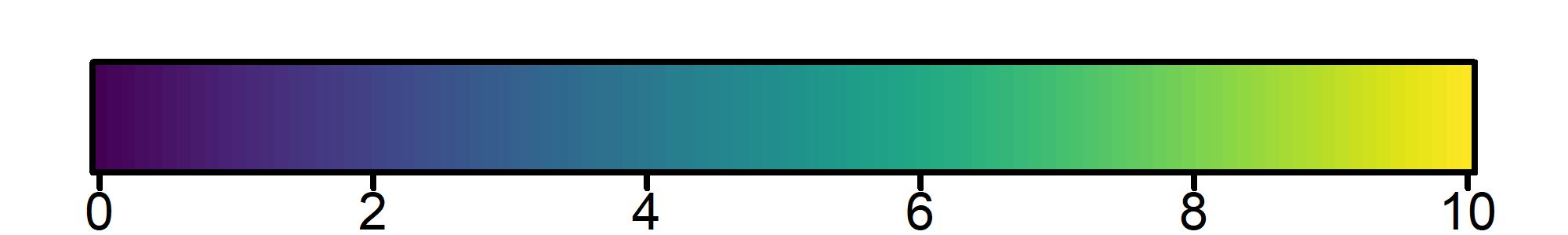}%
        \caption{$\kappa=1$, $\lambda=1$\label{fig:z:2}}
    \end{subfigure}%
    \begin{subfigure}{0.33\textwidth}
        \includegraphics[width=\textwidth]{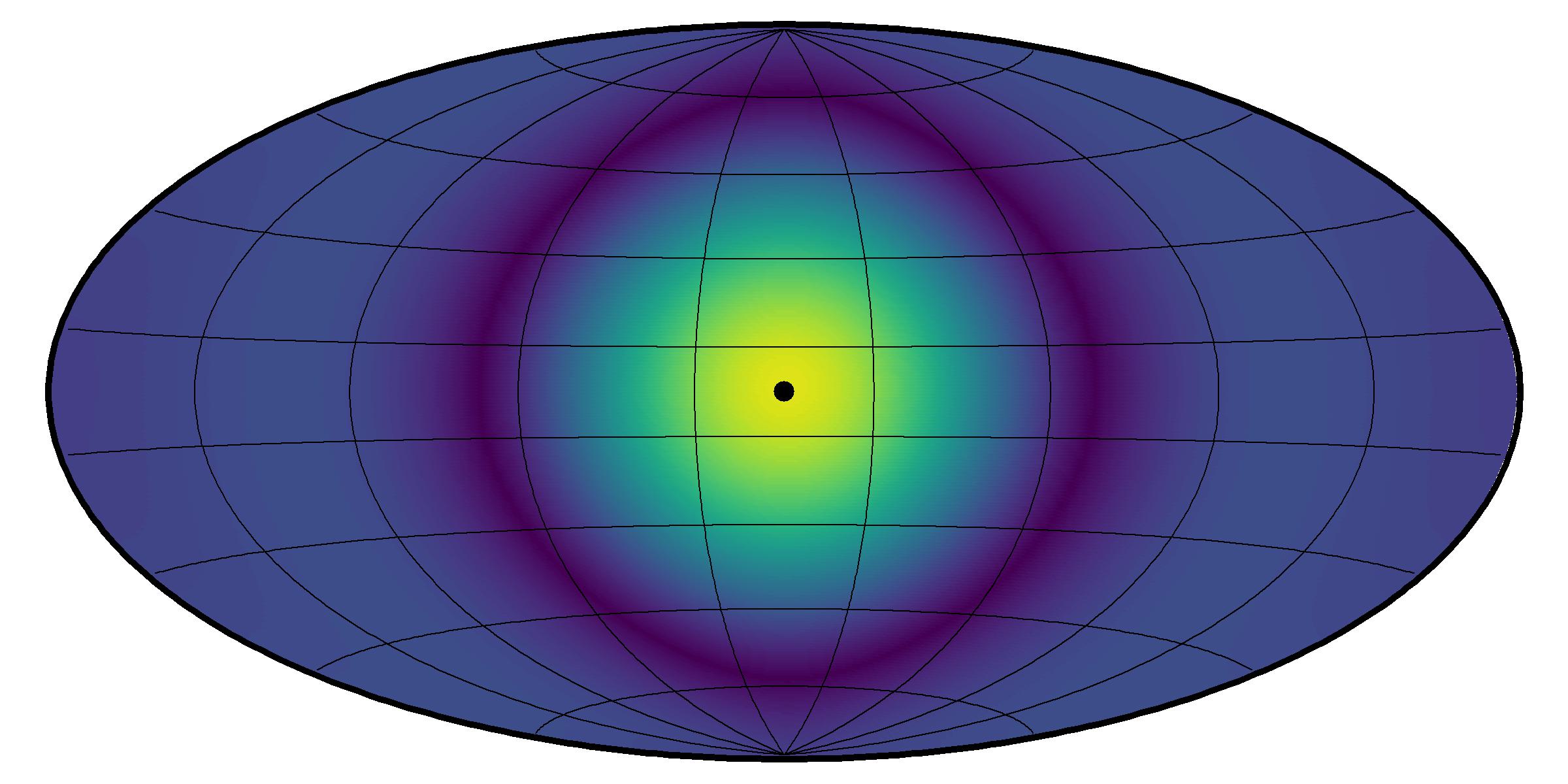}\\%
        \includegraphics[width=\textwidth]{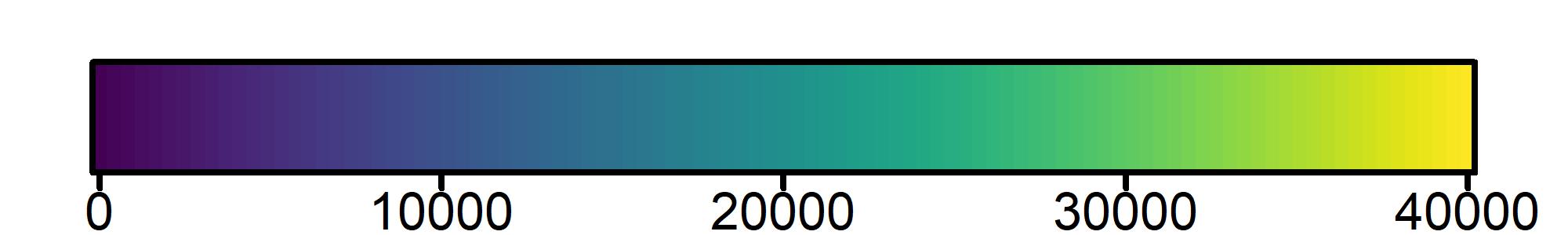}%
        \caption{$\kappa=1$, $\lambda=10$\label{fig:z:3}}
    \end{subfigure}\\%
    \begin{subfigure}{0.33\textwidth}
        \includegraphics[width=\textwidth]{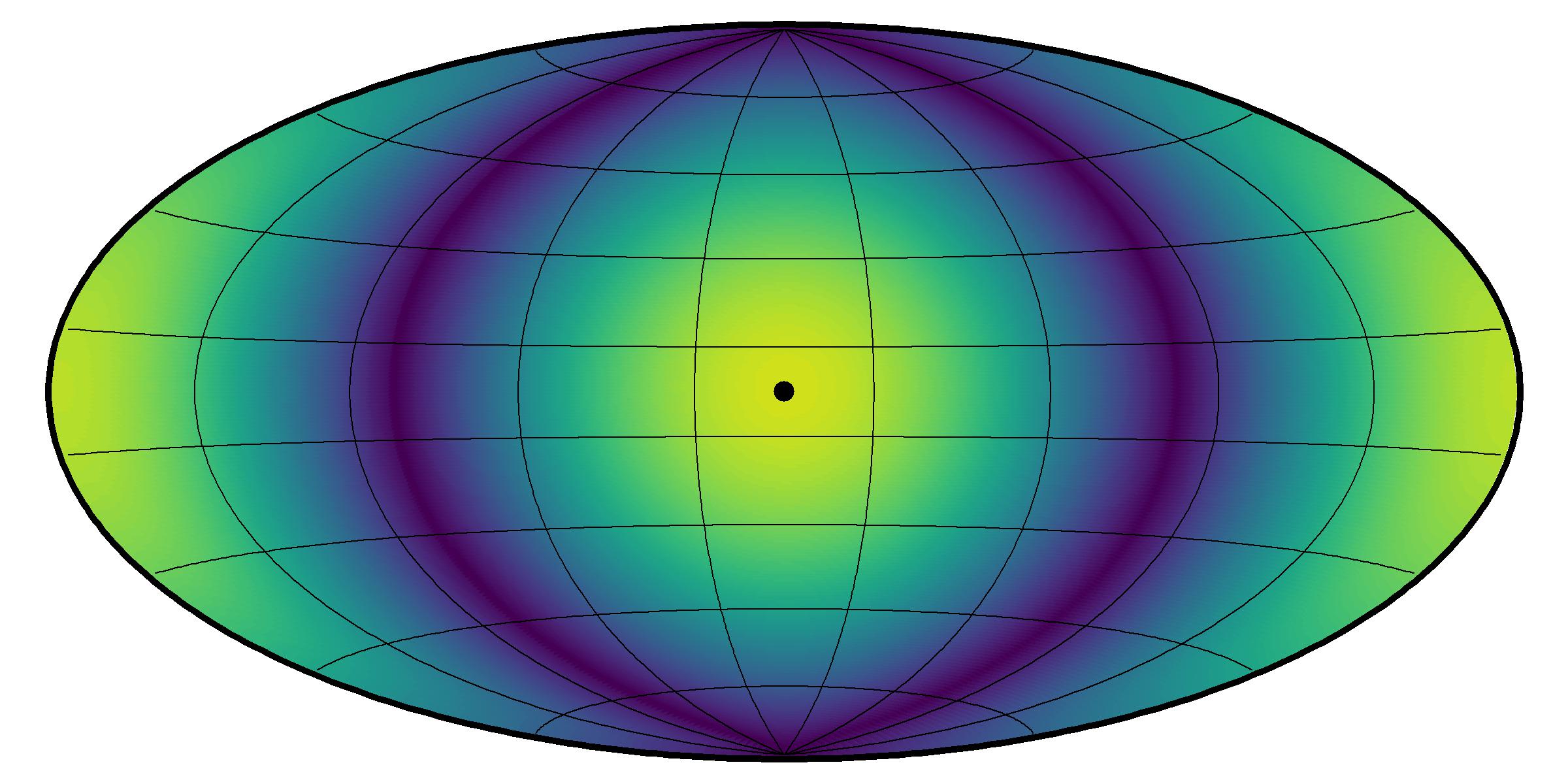}\\%
        \includegraphics[width=\textwidth]{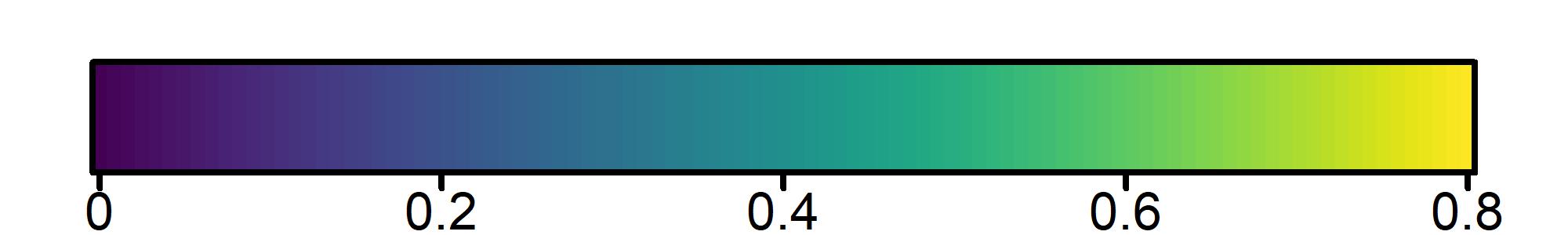}%
        \caption{$\kappa=0.1$, $\lambda=1$\label{fig:z:4}}
    \end{subfigure}%
    \begin{subfigure}{0.33\textwidth}
        \includegraphics[width=\textwidth]{figures/fields2/z_kappa1_lambda1.jpeg}\\%
        \includegraphics[width=\textwidth]{figures/fields2/z_colorbar_kappa1_lambda1.jpeg}%
        \caption{$\kappa=1$, $\lambda=1$\label{fig:z:5}}
    \end{subfigure}%
    \begin{subfigure}{0.33\textwidth}
        \includegraphics[width=\textwidth]{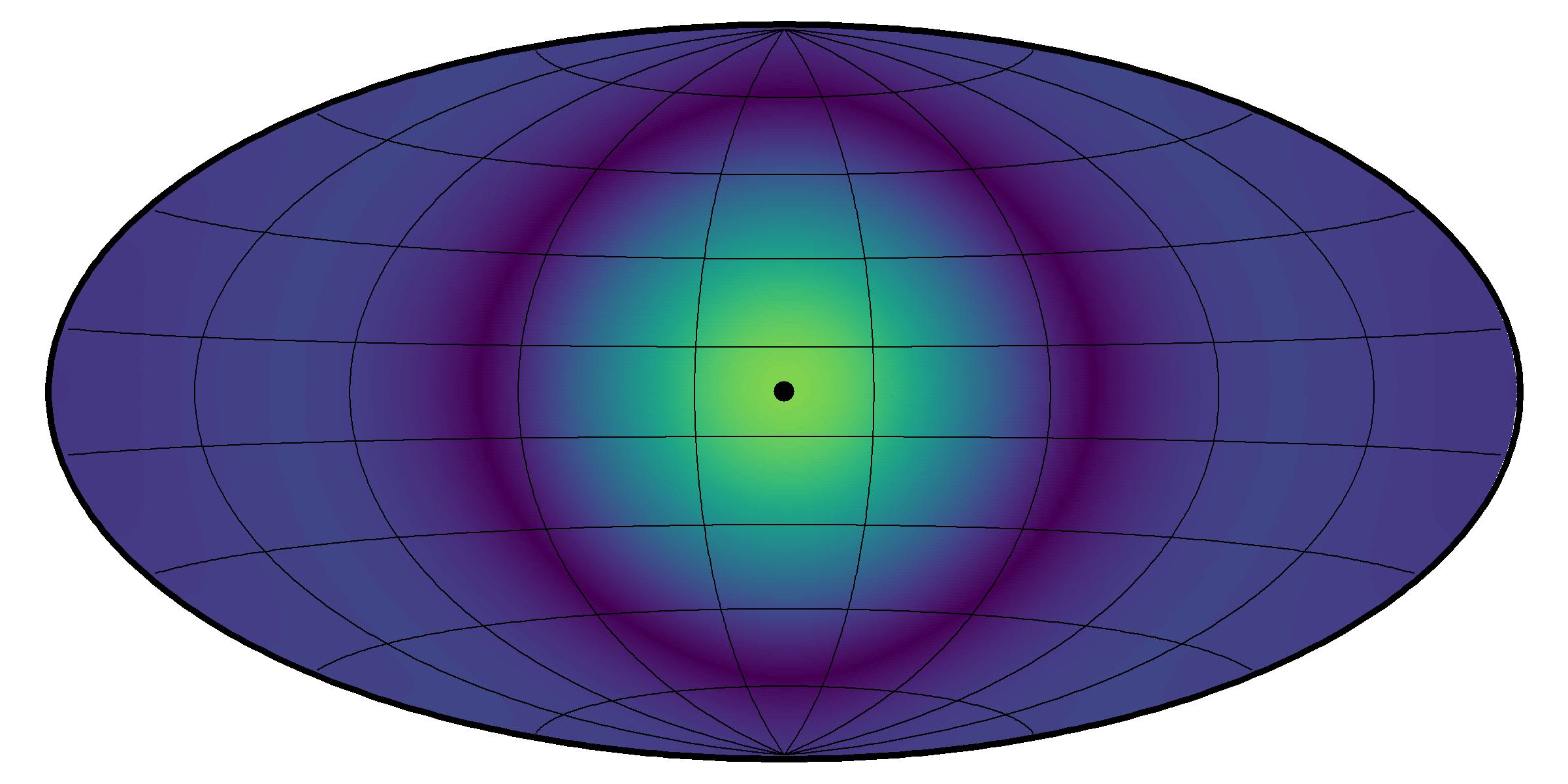}\\%
        \includegraphics[width=\textwidth]{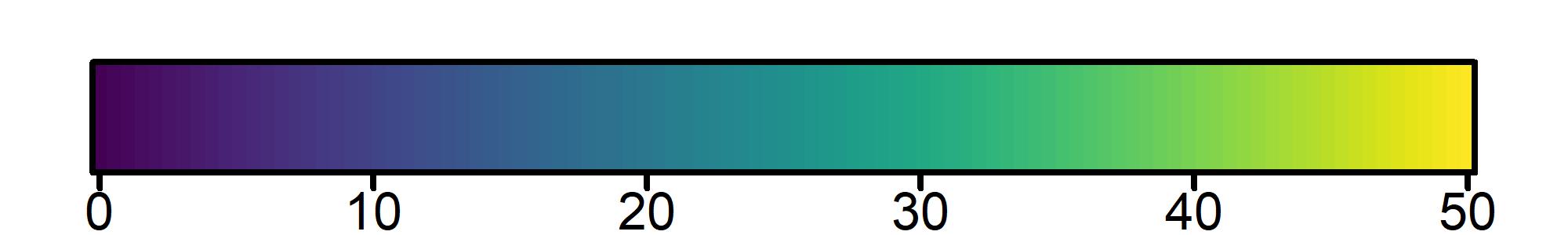}%
        \caption{$\kappa=10$, $\lambda=1$\label{fig:z:6}}
    \end{subfigure}
    \caption{\small Hammer projection representation of $\bs\mapsto \sqrt{n} |z(\bs)|$, for the fixed alternative $\mathrm{vMF}(\bmu,\kappa)$ and $n=100$. The central point is $\bmu=(0,-1,0)^\top$. In the first row, $\kappa=1$ is fixed, while in the second, $\lambda=1$ is.\label{fig:z}}
\end{figure}

The null correlation kernel $\bs\mapsto\rho(\bs,\bt)$ is shown in Figure \ref{fig:cor0} for $\bt=(0,0,1)^\top$. The kernel only depends on $\bs^\top\bt$ (i.e., it is isotropic). Increasing $\lambda$ has the effect of localizing the range of the correlation kernel at $\bs=\bt$. This happens both for positive and negative correlations. Positive correlations are located on the northern hemisphere, for $\lambda$ close to zero (Figure \ref{fig:cor0:1}), and then concentrate at $\bs=\bt$ for large $\lambda$ (Figure \ref{fig:cor0:3}). Negative correlations are located on the southern hemisphere for small $\lambda$, but then are attracted to parallels close to the north pole for large $\lambda$. Indeed, for large $\lambda$, near-zero correlations appear at the south pole and southern hemisphere.

\begin{figure}[h!]
    \centering
    \begin{subfigure}{0.33\linewidth}
        \includegraphics[width=\textwidth]{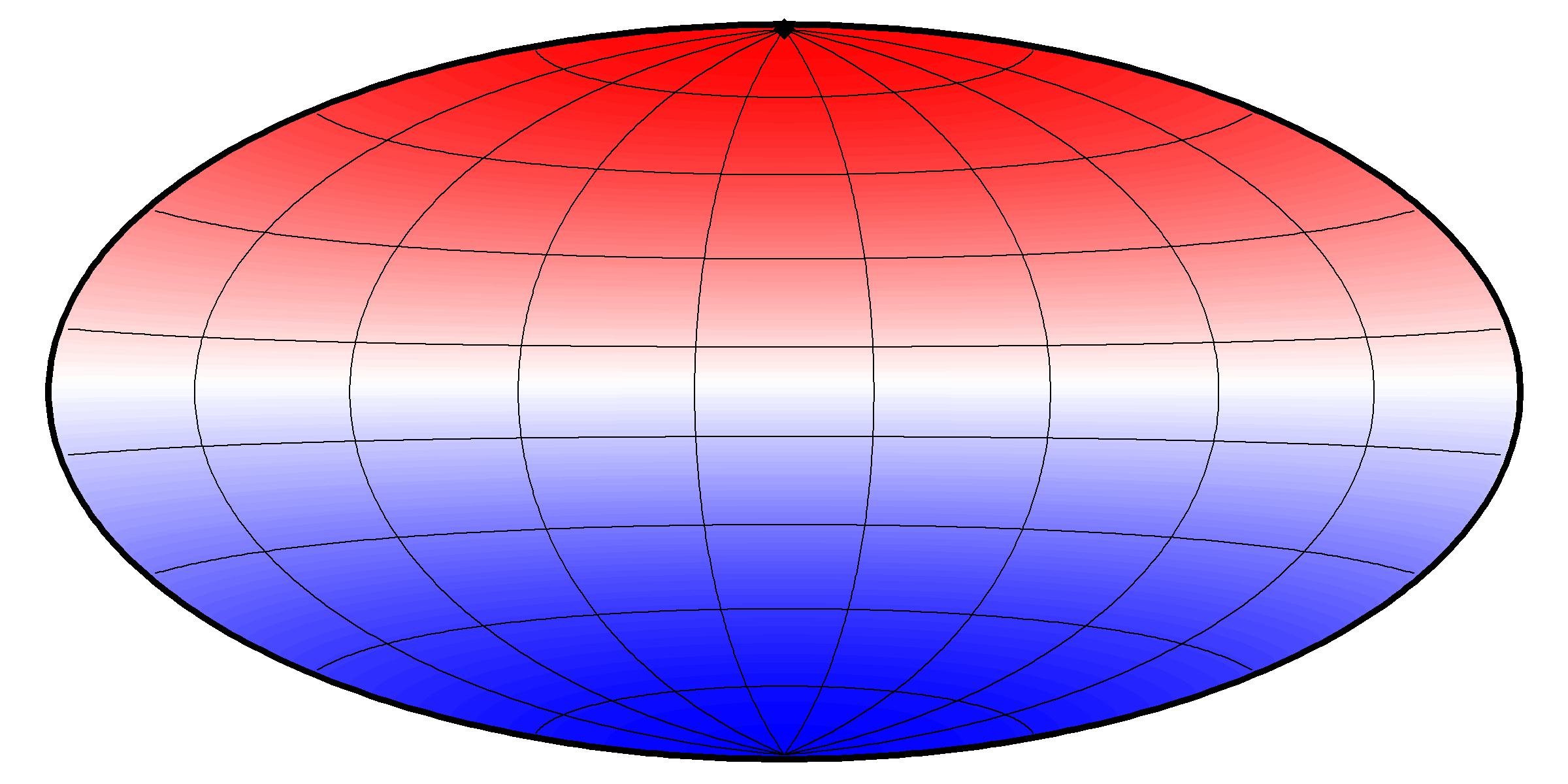}%
        \caption{$\lambda=0.1$\label{fig:cor0:1}}
    \end{subfigure}%
    \begin{subfigure}{0.33\linewidth}
        \includegraphics[width=\textwidth]{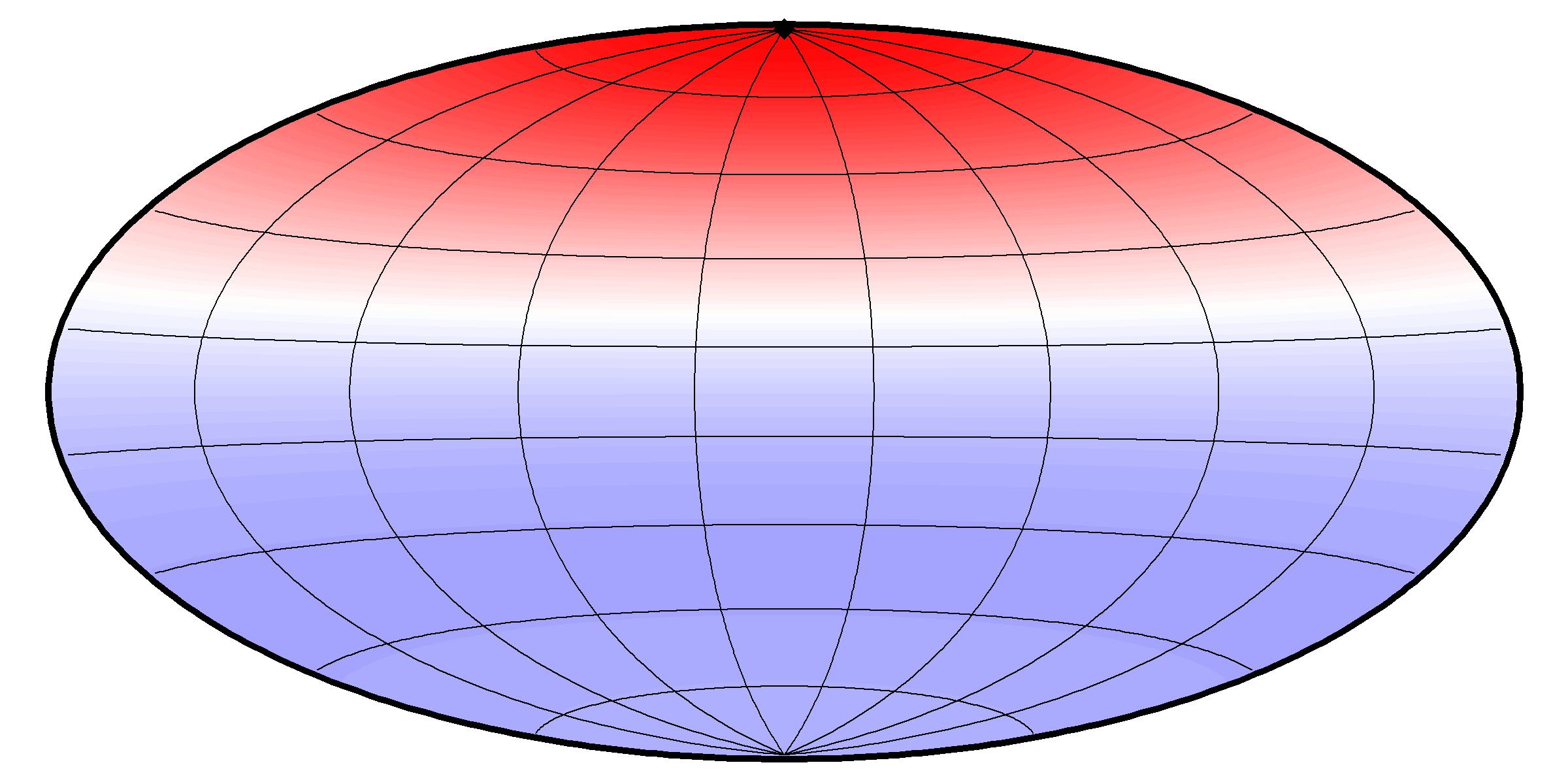}%
        \caption{$\lambda=1$\label{fig:cor0:2}}
    \end{subfigure}%
    \begin{subfigure}{0.33\linewidth}
        \includegraphics[width=\textwidth]{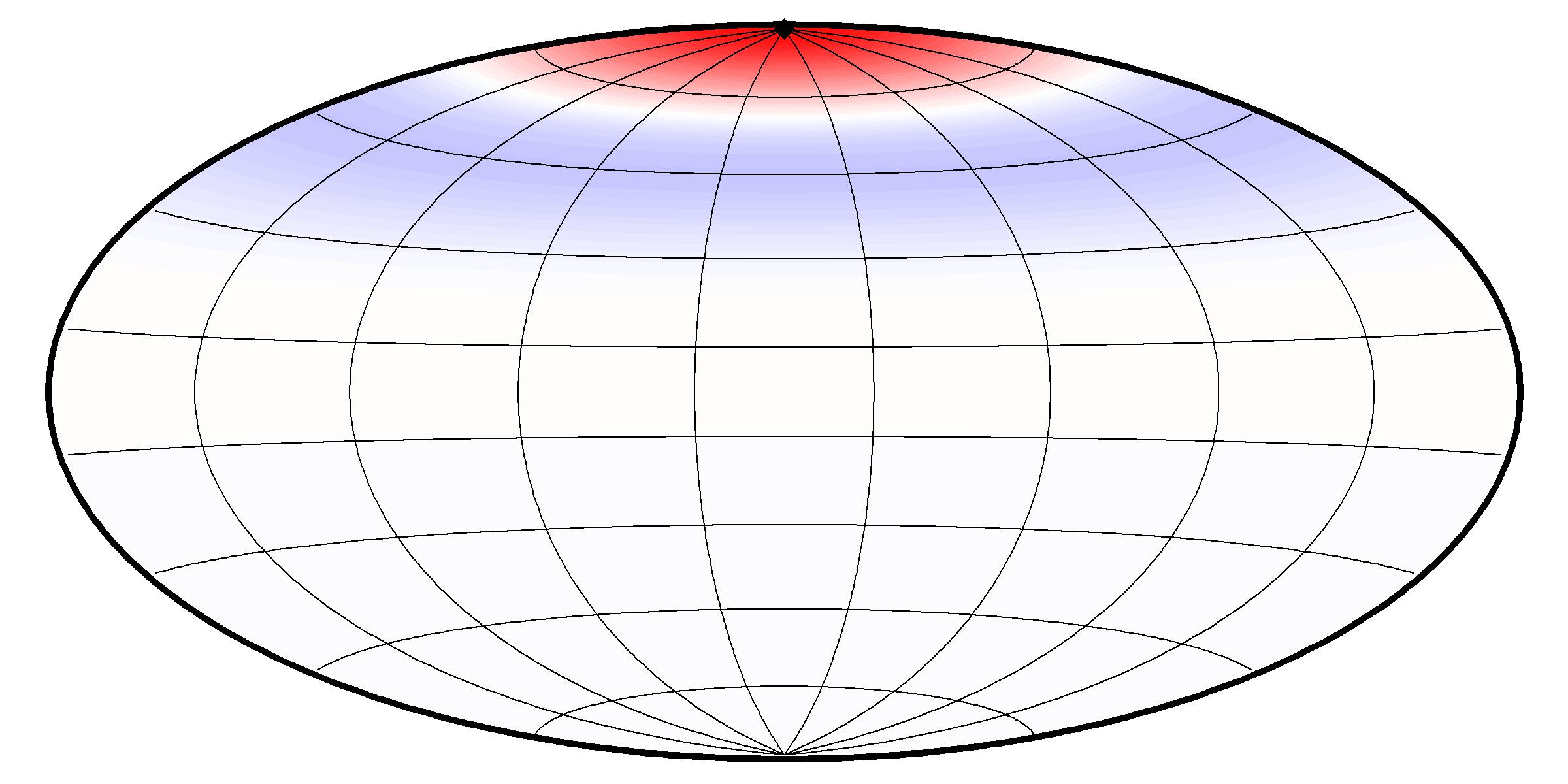}%
        \caption{$\lambda=10$\label{fig:cor0:3}}
    \end{subfigure}\\%
    \begin{subfigure}{0.33\linewidth}
        \includegraphics[width=\textwidth]{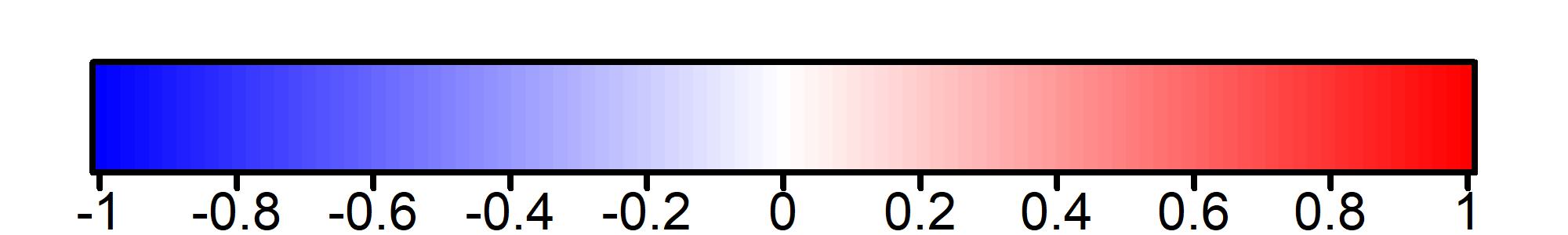}%
    \end{subfigure}%
    \caption{\small Hammer projection representation of the null correlation kernel $\bs\mapsto \rho(\bs,\bt)$, for $\bt=(0,0,1)^\top$ (north pole, diamond). The shape of the kernel is invariant to the choice of~$\bt$.\label{fig:cor0}}
\end{figure}

Finally, Figure \ref{fig:cor1} shows the fixed-alternative correlation kernel $\bs\mapsto\rho'(\bs,\bt)$, now dependent on $(\bs^\top\bt,\bmu^\top\bs,\bmu^\top\bt)$, for $\bmu=(0,-1,0)^\top$ and $\bt=(0,0,1)^\top$. For $\kappa=1$, the non-isotropy is subtle, with the effects of $\bmu$ being very mild, and the correlations resemble those in Figure \ref{fig:cor0}. The non-isotropy becomes evident for $\kappa=10$, where $\bmu$ affects the correlation field with the field value at $\bs$ depending on the angle between $\bs$ and $\bmu$. Strong positive correlations are still maintained at $\bs=\bt$, as expected.

\begin{figure}[t]
    \centering
    \begin{subfigure}{0.33\linewidth}
        \includegraphics[width=\textwidth]{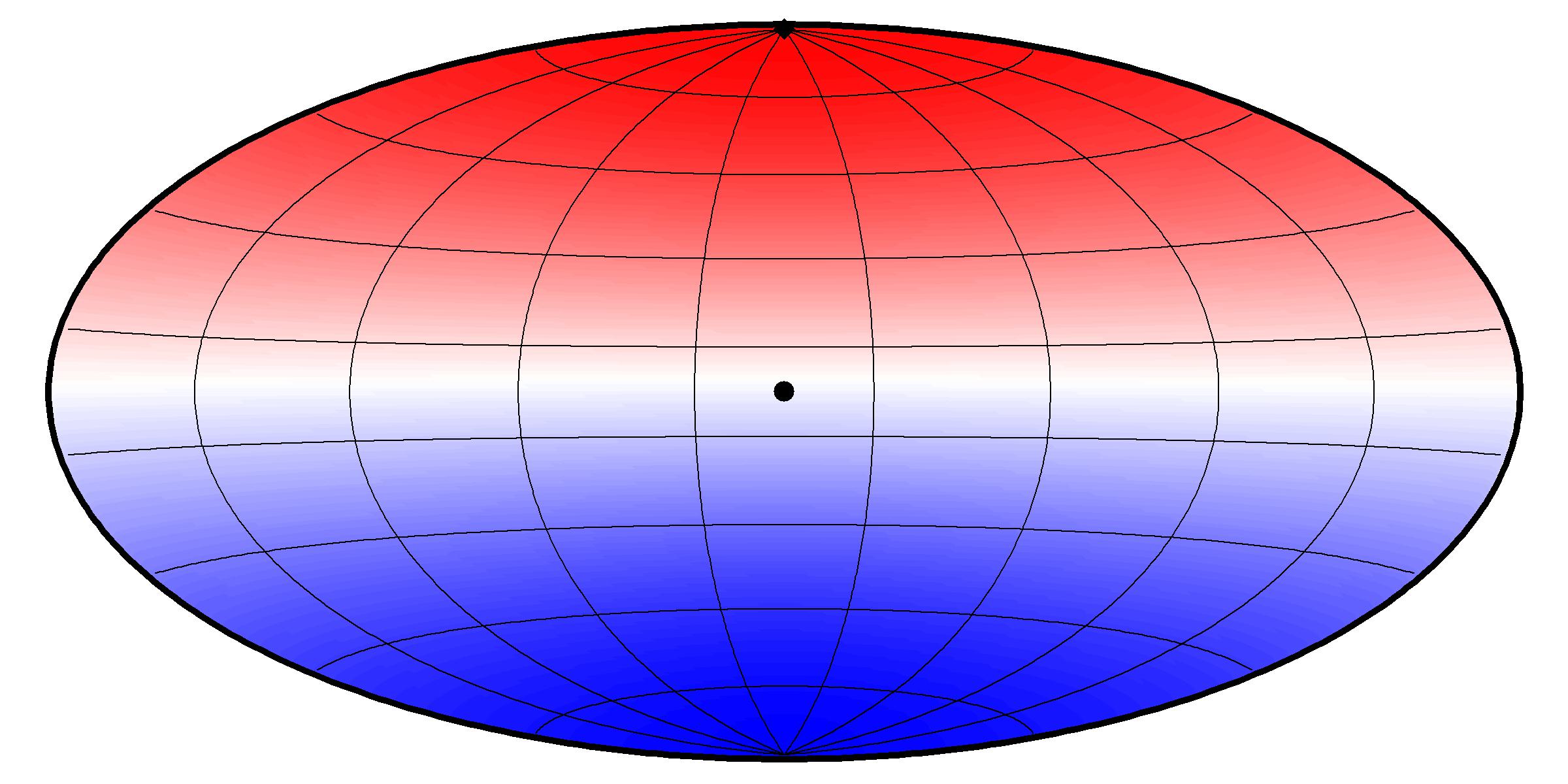}%
        \caption{$\kappa=1$, $\lambda=0.1$\label{fig:cor1:1}}
    \end{subfigure}%
    \begin{subfigure}{0.33\linewidth}
        \includegraphics[width=\textwidth]{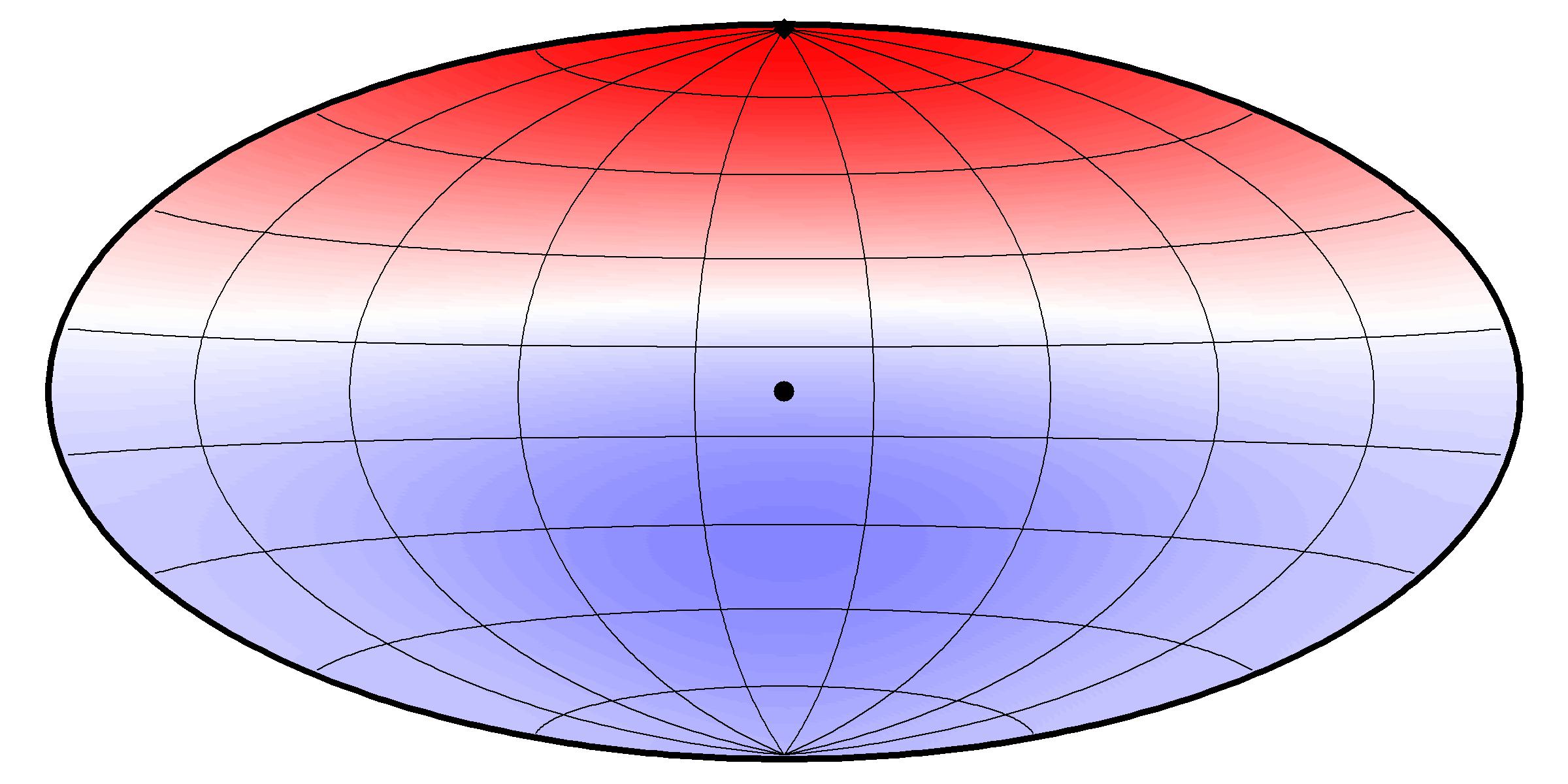}%
        \caption{$\kappa=1$, $\lambda=1$\label{fig:cor1:2}}
    \end{subfigure}%
    \begin{subfigure}{0.33\linewidth}
        \includegraphics[width=\textwidth]{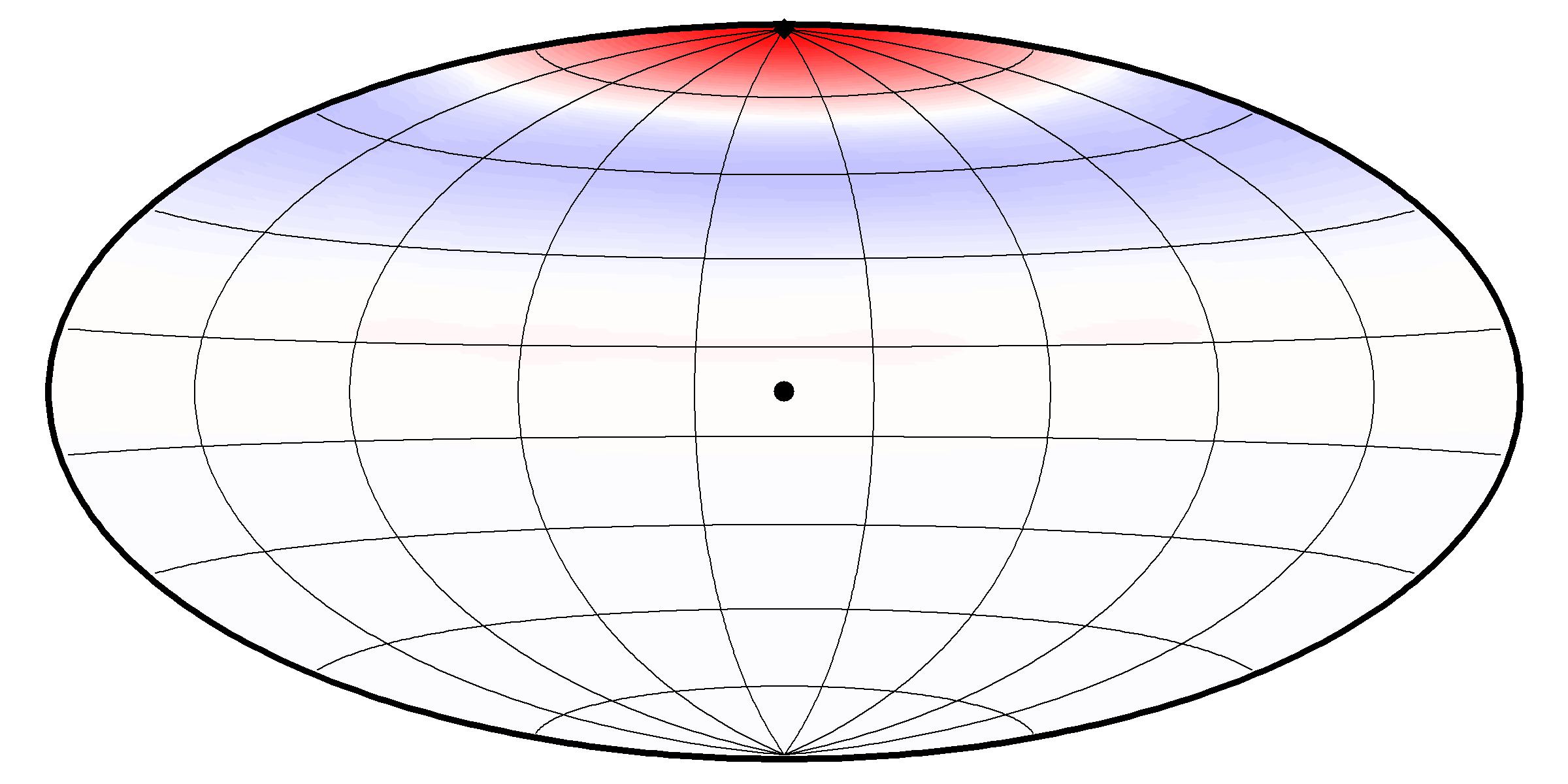}%
        \caption{$\kappa=1$, $\lambda=10$\label{fig:cor1:3}}
    \end{subfigure}\\%
    \begin{subfigure}{0.33\linewidth}
        \includegraphics[width=\textwidth]{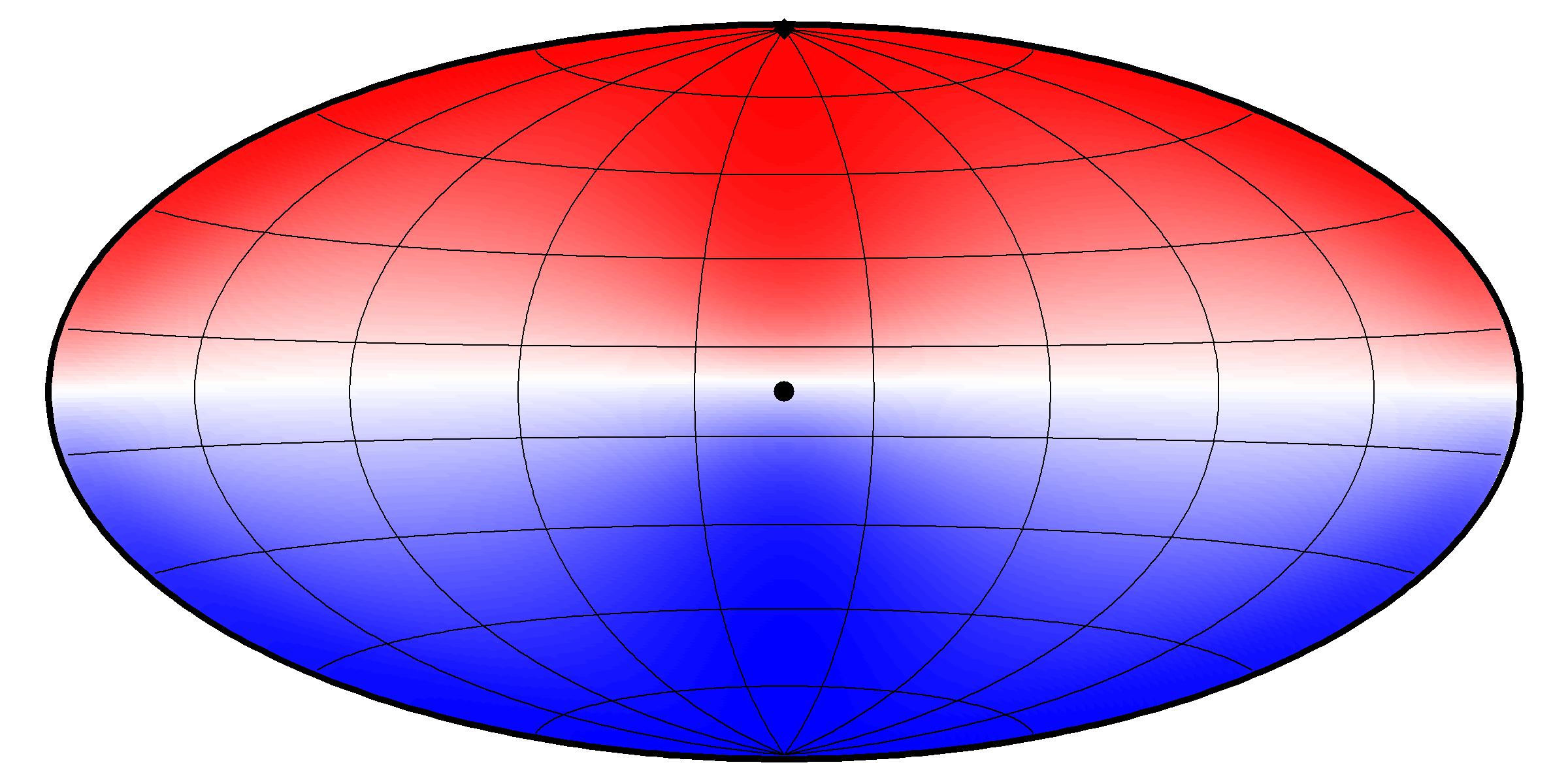}%
        \caption{$\kappa=10$, $\lambda=0.1$\label{fig:cor1:4}}
    \end{subfigure}%
    \begin{subfigure}{0.33\linewidth}
        \includegraphics[width=\textwidth]{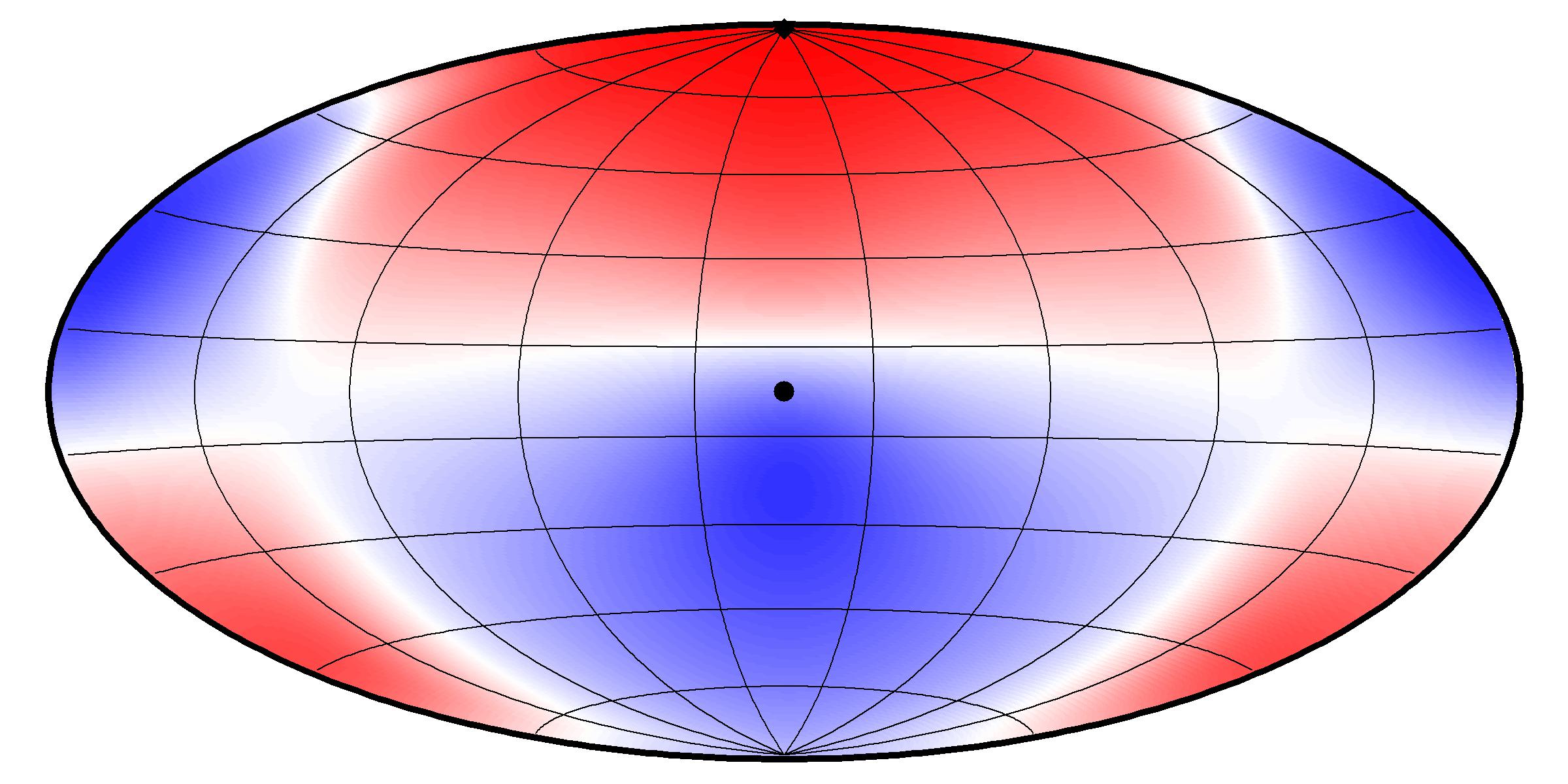}%
        \caption{$\kappa=10$, $\lambda=1$\label{fig:cor1:5}}
    \end{subfigure}%
    \begin{subfigure}{0.33\linewidth}
        \includegraphics[width=\textwidth]{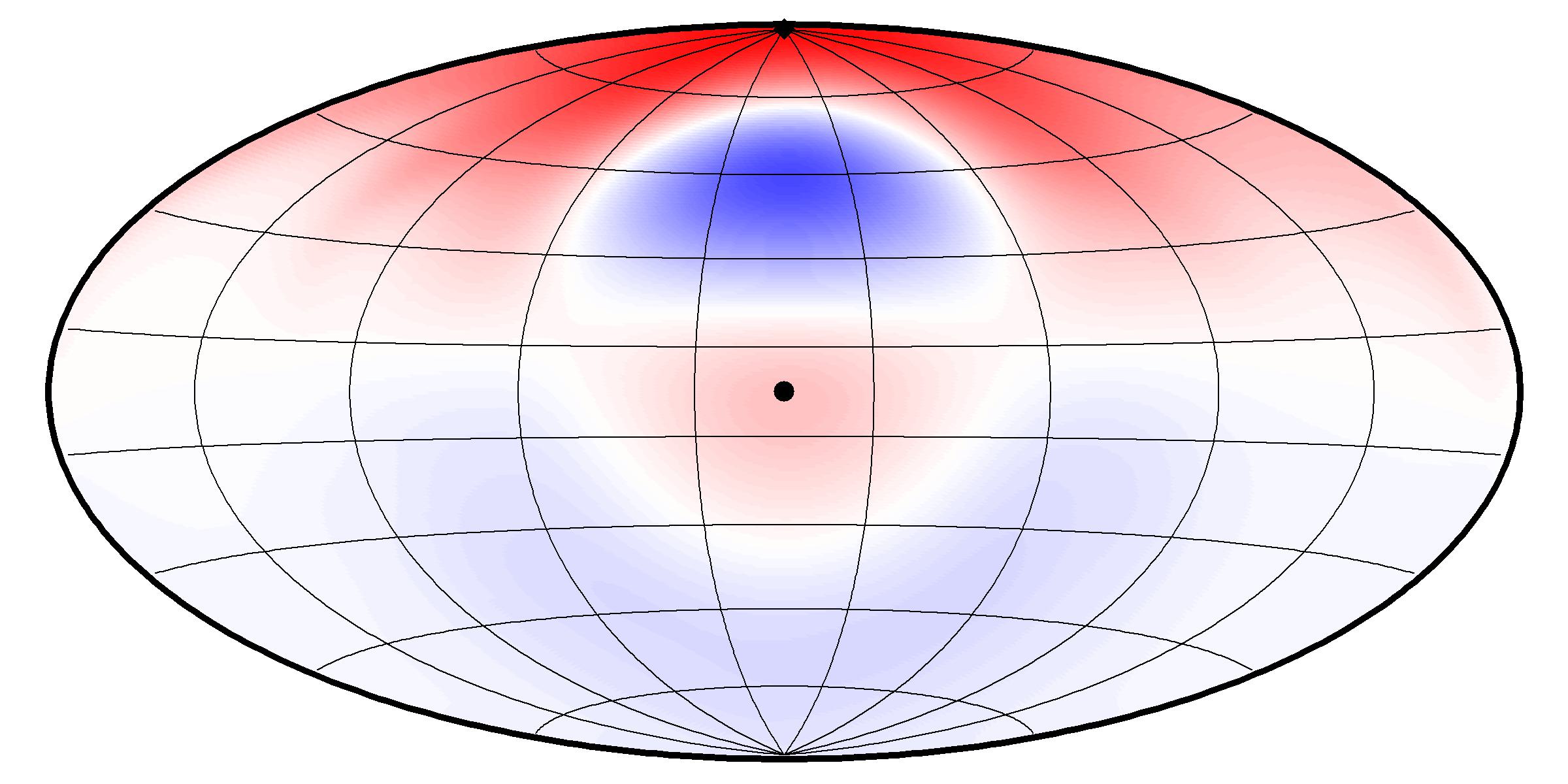}%
        \caption{$\kappa=10$, $\lambda=10$\label{fig:cor1:6}}
    \end{subfigure}\\%
    \begin{subfigure}{0.33\linewidth}
        \includegraphics[width=\textwidth]{figures/fields2/cor_colorbar.jpeg}%
    \end{subfigure}%
    \caption{\small Hammer projection representation of the fixed-alternative correlation kernel $\bs\mapsto \rho'(\bs,\bt)$, for the fixed alternative $\mathrm{vMF}(\bmu,\kappa)$ and $\bt=(0, 0, 1)^\top$ (north pole, diamond). The central point is $\bmu=(0,-1,0)^\top$. \label{fig:cor1}}
\end{figure}

\subsection{Parameter selection}
\label{sec:parameter}

The testing procedure can be adapted to a specific alternative by selecting the tuning parameter $\lambda$ that maximizes a tuning criterion. As an oracle criterion, we consider the standardized mean shift under $\mathcal{H}_1$, compared to $\mathcal{H}_0$:
\begin{align*}
    \tilde\lambda=\arg\max_{\lambda>0}q(\lambda)=\arg\max_{\lambda>0}\frac{\mathbb{E}_{\mathcal{H}_1}[T_n(\lambda)]-\mathbb{E}_{\mathcal{H}_0}[T_n(\lambda)]}{\sqrt{\Vs{T_n(\lambda)}{\mathcal{H}_0}}}.
\end{align*}
Maximizing this expression is a natural criterion for selecting $\lambda$ under a given alternative; see \cite{Gregory1977}. 
Alternatively, using Theorem~\ref{thm:var:mgf} and incorporating the critical value $c_n(\lambda)$, an estimate of the power function (see \citet[Section 3.2]{BEH:2017}) can be obtained as $\mathbb{P}_{\mathcal{H}_1}\big(T_n(\lambda)>c_n(\lambda)\big)\approx 1-\Phi\big({\sqrt{n}}/{\sigma}\big({c_n(\lambda)}/{n}-\tau\big)\big)$. For simplified computation, we use the score function $q(\lambda)$.

Here, \eqref{eq:varH0} provides closed expressions for $\mathbb{E}_{\mathcal{H}_0}[T_n(\lambda)]$ and $\Vs{T_n(\lambda)}{\mathcal{H}_0}$ while we approximate the expectation $\mathbb{E}_{\mathcal{H}_1}[T_n(\lambda)]$ by Monte--Carlo simulation. Since $\lambda$ affects the test statistic only through the coefficients $c_{k,p}(\lambda)$, it is convenient to rearrange the summation to
\begin{align*}
    T_n(\lambda)=\sum_{k=1}^\infty c_{k,p}(\lambda)A_k, \quad A_k=\frac{1}{n}\sum_{i,j=1}^nC_k^{(p-2)/2}(\bX_i^\top \bX_j).
\end{align*}

To approximate the oracle choice, suppose that $\bY_1,\ldots,\bY_N$, $N\in\N$ are iid from the alternative distribution. Since $\mathbb{E}[T_n(\lambda)]=\sum_{k=1}^\infty c_{k,p}(\lambda)\mathbb{E}[A_k]$, it is sufficient to estimate $\mathbb{E}[A_k]=(n-1)\Ebig{C_k^{(p-2)/2}(\bX_1^\top \bX_2)}+C_k^{(p-2)/2}(1)$ by $\bar{A}_k=(n-1)\frac{1}{N(N-1)}\sum_{1\leq i\neq j\leq N}C_k^{(p-2)/2}(\bY_i^\top \bY_j)+C_k^{(p-2)/2}(1)$. This allows direct evaluation of $\bar{T}_n(\lambda)=\sum_{k=1}^\infty c_{k,p}(\lambda)\bar{A}_k$, and the corresponding oracle parameter
\begin{align*}
    \hat{\lambda}=\arg\max_{\lambda>0}\frac{\bar{T}_n(\lambda)-\mathbb{E}_{\mathcal{H}_0}[T_n(\lambda)]}{\sqrt{\Vs{T_n(\lambda)}{\mathcal{H}_0}}}.
\end{align*}

Since an independent sample from the underlying alternative is typically unavailable, the oracle selection cannot be applied in practice. A standard approach to data-driven parameter selection is cross-validation, for which we refer to the procedure described in \citet[Section~4]{Fernandez-de-Marcos2023b}. We denote the 10-fold cross-validation test using the criterion $q_n$ from \eqref{eq:q_n_def} by $T_n(\lambda_{\mathrm{CV}})$.

In an alternative data-driven approach, we consider the maximum of the standardized test statistic over a closed interval $[a,b]\subset (0,\infty)$, 
\begin{align*}
    q_{n,\max}:=\sup_{a\le\lambda\le b} q_n(\lambda)=\sup_{a\le\lambda\le b} \frac{{T_n}(\lambda)-\mathbb{E}_{\mathcal{H}_0}[T_n(\lambda)]}{\sqrt{\Vs{T_n(\lambda)}{\mathcal{H}_0}}}.
\end{align*}
In the simulations, we use Monte--Carlo calibrations to determine critical values both for $q_{n,\max}$ and $T_n(\lambda_{\mathrm{CV}})$. By Proposition \ref{prop:H0_func} $q_{n,\max}$ can alternatively be calibrated via the asymptotic distribution in \eqref{eq:limit:max_q_n}.

In the following section, we implement these approaches by searching the grid $\{i/10: i\in\{1,\ldots,300\}\}$ for $\lambda$ in all three variants. To approximate the oracle coefficients, we use $10,\!000$ independent draws from the alternative density under consideration. For many rotationally symmetric alternatives, the coefficients $\beta_k$ admit closed-form expressions for known distributions; see Example~\ref{ex:betavMF} for the von Mises--Fisher distribution. The corresponding expectation can then be derived using Remark~\ref{rem:Ex:H1}.

\subsection{Comparison to other tests}

We compare the power of the proposed test statistics across a selection of alternative distributions and benchmark them against different Sobolev tests of uniformity. To that end, we consider $T_n(\hat{\lambda})$, $q_{n,\max}$, $T_n(\lambda_{\mathrm{CV}})$ and the test statistic $T_n(\lambda)$ for fixed values $\lambda\in\{1,4\}$. As competing tests we consider the \cite{Gine1975} $F_n$ test ($F_n$), the \cite{Bingham1974} test ($B_n$), the \cite{Rayleigh1919} test ($R_n$), the softmax test ($S_n$) from \cite{Fernandez-de-Marcos2023b}, the Projected Anderson--Darling test ($\mathrm{PAD}$) from \cite{Garcia-Portugues2023}, and the $\mathrm{dKSD}$ test \citep{xu2020a}. The comparison is performed for dimensions $p=2,3,5$ and sample sizes $n=50$ and $n=100$. The test statistic is truncated to its first $100$ terms.

We structure the comparison by first considering unimodal alternatives. Here, we consider a von Mises--Fisher distribution $f_\mathrm{vMF}(\bx;\bmu,\kappa) ~\propto~e^{\kappa\bmu^\top \bx}$ with $\bmu:=\be_1$ and concentration parameter $\kappa=0.5$. We also consider a Cauchy-like distribution
\begin{align*}
    f_{\mathrm{Ca}}(\bx;\bmu,\kappa)=\left(\frac{1-\rho(\kappa)^2}{1-2\bmu^\top\bx\rho(\kappa)+\rho(\kappa)^2}\right)^p \quad\text{with}\quad \rho(\kappa)=\frac{2\kappa+1-\sqrt{4\kappa+1}}{2\kappa},\quad \bx\in\Spp,
\end{align*}
with $\kappa=0.25$. We denote these alternatives by $\mathrm{vMF}(0.5)$ and $\mathrm{Ca}(0.25)$, respectively.

We now consider axial data. On the one hand, we sample from the Watson distribution $f_{\mathrm{W}}(\bx;\bmu,\kappa)\propto e^{\kappa(\bmu^\top \bx)^2}$ with $\kappa=1$ (denoted $\mathrm{W}(1)$). On the other hand, we sample from an unbalanced mixture of two von Mises--Fisher distributions $\mathrm{vMF}(\be_1,2)$ and $\mathrm{vMF}(-\be_1,2)$, at opposite poles,
\begin{align*}
    f_{\mathrm{MvMF}_{2}}(\bx;q)= (1-q) f_\mathrm{vMF}(\bx;\be_1,2)+q f_\mathrm{vMF}(\bx;-\be_1,2),\quad \bx\in\Spp,
\end{align*}
with $q=0.3$ ($\mathrm{MvMF}_{2}(0.3)$).

A small circle distribution $f_{\mathrm{SC}}(\bx;\kappa,\nu)\propto e^{-\kappa(\be_1^\top \bx-\nu)^2}$, concentrated around a modal lower-dimensional subsphere, with $\kappa=0.5$ and $\nu=0.5$, is also considered ($\mathrm{SC}(0.5,0.5)$).

To define alternatives obtained by rotating rotationally symmetric distributions, let $\mathbf{R}_{i,j}(\alpha)$ denote the rotation matrix in the $(i,j)$-plane with rotation angle $\alpha$. We denote by $\mathrm{SCM}(3)$ the equally weighted mixture of $k=3$ copies of $\mathrm{SC}(10,0)$, rotated by an angle $(j/k)2\pi$ for $j\in[k]$ in the $(2,3)$-plane:
\begin{align*}
    f_{\mathrm{SCM}}(\bx;k)=\sum_{j=1}^k\frac{1}{k}f_{\mathrm{SC}}\lrp{\mathbf{R}_{2,3}\lrp{-\frac{j}{k}2\pi}\bx;10,0},\quad \bx\in\Spp.
\end{align*}

To generate a random vector from $\mathrm{projNM}(5)$, we first draw from $\mathcal{N}(4\be_1,\bSigma)$ with the diagonal covariance matrix $\bSigma=\bI_p+9\be_p\be_p^\top$; the resulting vectors are projected onto the unit sphere and rotated. First, we define the density of the projected normal distribution
\begin{align*}
    f_{\mathrm{projN}}(\bx)\propto\int_0^\infty r^{p-1}\exp\left(-\frac{1}{2}(r\bx-4\be_1)^\top \bSigma^{-1}(r\bx-4\be_1)\right)\,\rd r,\quad \bx\in\Spp,
\end{align*}
to then obtain the density of the mixture of rotated projected normal distributions, as
\begin{align*}
    f_{\mathrm{projNM}}(\bx;k)= \sum_{j=1}^k\frac{1}{k}f_{\mathrm{projN}}\lrp{\mathbf{R}_{1,2}\lrp{-\frac{j}{k}2\pi}\bx},\quad \bx\in\Spp.
\end{align*}

Finally, we denote by $\mathrm{MvMF}_{2p}(30)$ a distribution consisting of an equal mixture of $2p$ vMF distributions with equal concentration parameter $\kappa=30$. Here, the mean directions are the canonical unit vectors and their negatives:
\begin{align*}
    f_{\mathrm{MvMF}_{2p}}(\bx;\kappa)= \frac{1}{2p}\sum_{j=1}^p\lrb{f_\mathrm{vMF}(\bx;\be_j,\kappa)+f_\mathrm{vMF}(\bx;-\be_j,\kappa)},\quad \bx\in\Spp.
\end{align*}

The results are summarized in Tables~\ref{table:powers_d2}--\ref{table:powers_d5}. The rejection rates are computed from $M=10,\!000$ Monte Carlo repetitions, with critical values under $\mathcal{H}_0$ at significance level $\alpha=5\%$ approximated with $M$ samples under the null hypothesis. For each alternative, we highlight in bold the test statistic with the highest power as well as any tests whose power is not significantly lower than the best-performing test. Statistical significance is assessed using a paired one-sided $t$-test at level~$5\%$.
\begin{table}[ht]
\vspace*{-0.25cm}
\centering
\footnotesize
\sisetup{
  table-format=2,
  detect-weight=true,
  detect-family=true
}
\renewcommand{\arraystretch}{1.0}

\begin{tabular}{l r| S S S S S S S S S S S}
\toprule
Distribution & $n$ & {$T_n(\hat{\lambda})$} & {$q_{n,\max}$} & {$T_n(\lambda_{\mathrm{CV}})$} & {$T_n(1)$} & {$T_n(4)$} & {$\mathrm{dKSD}$} & {$F_n$} & {$B_n$} & {$R_n$} & {$S_n$} & {PAD} \\
\midrule
$\mathrm{Unif}(\Spp)$ & 50 & 4.6 & 4.6 & 6.2 & 4.9 & 4.6 & 4.9 & 5.1 & 4.9 & 5.4 & 5.3 & 5.3 \\
 & 100 & 5.2 & 5.0 & 4.6 & 5.3 & 4.7 & 5.1 & 5.2 & 4.5 & 5.2 & 5.3 & 5.2 \\
\midrule
$\mathrm{vMF}(0.5)$ & 50 & 59.0 & 43.5 & 40.0 & 48.2 & 7.8 & 46.8 & 56.2 & 5.8 & \multicolumn{1}{S}{\bfseries 59.1} & 58.4 & 57.5 \\
 & 100 & \multicolumn{1}{S}{\bfseries 89.4} & 79.5 & 66.9 & 82.1 & 10.3 & 80.3 & 88.1 & 5.7 & \multicolumn{1}{S}{\bfseries 89.4} & 89.0 & 88.3 \\
\midrule
$\mathrm{Ca}(0.25)$ & 50 & 33.2 & 21.2 & 20.7 & 25.8 & 6.4 & 24.8 & 31.5 & 5.6 & \multicolumn{1}{S}{\bfseries 33.3} & \multicolumn{1}{S}{\bfseries 33.0} & 32.4 \\
 & 100 & \multicolumn{1}{S}{\bfseries 59.8} & 43.9 & 32.6 & 49.0 & 7.2 & 46.8 & 57.5 & 5.7 & 59.7 & 59.1 & 58.2 \\
\midrule
$\mathrm{W}(1)$ & 50 & 54.9 & 37.1 & 36.6 & 45.0 & 35.5 & 45.1 & 23.2 & \multicolumn{1}{S}{\bfseries 57.8} & 5.6 & 12.2 & 17.6 \\
 & 100 & 85.3 & 72.7 & 63.7 & 78.2 & 69.1 & 77.9 & 51.6 & \multicolumn{1}{S}{\bfseries 87.2} & 5.7 & 26.6 & 39.0 \\
\midrule
$\mathrm{SC}(0.5,0.5)$ & 50 & 50.3 & 36.9 & 34.6 & 46.0 & 10.5 & 45.0 & \multicolumn{1}{S}{\bfseries 50.6} & 12.1 & 49.2 & \multicolumn{1}{S}{\bfseries 50.8} & \multicolumn{1}{S}{\bfseries 50.8} \\
 & 100 & 83.2 & 73.6 & 60.7 & 82.1 & 18.6 & 80.3 & \multicolumn{1}{S}{\bfseries 84.7} & 19.7 & 81.4 & 84.0 & 84.2 \\
\midrule
$\mathrm{projNM}(5)$ & 50 & \multicolumn{1}{S}{\bfseries 10.9} & 7.0 & 8.6 & 4.8 & 5.6 & 4.9 & 5.2 & 4.8 & 5.1 & 5.2 & 5.4 \\
 & 100 & \multicolumn{1}{S}{\bfseries 19.1} & 9.8 & 11.2 & 5.1 & 6.6 & 5.0 & 5.6 & 4.6 & 5.3 & 5.3 & 5.9 \\
\midrule
$\mathrm{MvMF}_{2p}(30)$ & 50 & \multicolumn{1}{S}{\bfseries 100.0} & \multicolumn{1}{S}{\bfseries 100.0} & \multicolumn{1}{S}{\bfseries 100.0} & 7.7 & \multicolumn{1}{S}{\bfseries 100.0} & 29.9 & 53.4 & 6.6 & 5.5 & 6.7 & 77.3 \\
 & 100 & \multicolumn{1}{S}{\bfseries 100.0} & \multicolumn{1}{S}{\bfseries 100.0} & \multicolumn{1}{S}{\bfseries 100.0} & 8.7 & \multicolumn{1}{S}{\bfseries 100.0} & 88.7 & \multicolumn{1}{S}{\bfseries 100.0} & 6.7 & 5.2 & 7.3 & \multicolumn{1}{S}{\bfseries 100.0} \\
\midrule
$\mathrm{MvMF}_{2}(0.3)$ & 50 & \multicolumn{1}{S}{\bfseries 90.8} & 83.8 & 81.5 & \multicolumn{1}{S}{\bfseries 90.8} & 61.5 & 90.4 & 86.0 & 79.4 & 68.5 & 81.1 & 83.6 \\
 & 100 & \multicolumn{1}{S}{\bfseries 99.9} & 99.4 & 98.3 & \multicolumn{1}{S}{\bfseries 99.8} & 94.5 & \multicolumn{1}{S}{\bfseries 99.8} & 99.5 & 98.0 & 93.2 & 98.7 & 99.1 \\
\bottomrule
\end{tabular}
\caption{\small Empirical rejection percentages in dimension $p=2$ computed with $M=10,\!000$ Monte Carlo samples and at significance level $\alpha=5\%$. Bold entries indicate best-performing tests for each alternative.}
\label{table:powers_d2}
\vspace*{-0.25cm}
\end{table}

We make the following observations. In the case of unimodal alternatives in all dimensions, the default parameter $\lambda=1$ leads to similar results to those of the Rayleigh test. The optimal $\hat{\lambda}$ is as small as possible, achieving a rejection rate arbitrarily close to that of the limiting Rayleigh test, see Proposition~\ref{prop:limit}. Using $\lambda=4$ substantially reduces power, illustrating that larger values of $\lambda$ are suboptimal against weakly concentrated alternatives. The rates for the alternative $\mathrm{SC}(0.5,0.5)$ behave similarly to the unimodal ones, but the optimal $\hat\lambda$ does not approach zero.

In the axial alternatives, the Bingham test performs best against $\mathrm{W}(1)$. With optimal tuning, while $T_n(\hat{\lambda})$ does not reach the same power as the Bingham test, it is more sensitive than the other tests considered. The $\mathrm{MvMF}_{2}(0.3)$ alternative, with modes of different weights at opposite poles, reduces the advantage of the Bingham test, while improving detection rates for the other tests considered. In this mixed scenario, our test benefits from its flexibility. In $p=2$ the parameter $\lambda=1$ performs well, while in $p=5$, $\lambda=4$ produces higher rejection rates than the competitors. In the case of $\mathrm{SCM}(3)$, considerable mass is again concentrated near opposing poles, but there is further concentration around small circles. Thus, the tests presented are all sensitive to the alternative, but $T_n(\hat{\lambda})$ is the leading test among those considered in this setting.

For mixtures with multiple modes of high concentration, a larger tuning parameter is preferable. As observed in Figure \ref{fig:cor0}, an increasing parameter $\lambda$ has a localizing effect, improving the detection of high concentration modes and preventing cancellation between opposing modes. This is evident in $\mathrm{MvMF}_{2p}(30)$ and, in lower-dimensional settings, in $\mathrm{projNM}(5)$, where $T_n(\hat{\lambda})$ and $T_n(4)$ have the highest rejection rates, while most other tests show substantially lower power.

While both data-driven selection approaches lose significant power to the oracle test, they are robust across all alternatives. In our scenarios, especially under multimodal alternatives, these selection methods perform well, with the optimal $\lambda$ differing considerably from the preselected fixed parameters used in the comparison. Overall, the $q_{n,\max}$ test outperforms the $10$-fold cross-validation approach in the scenarios considered, except for $\mathrm{projNM}(5)$ in dimension $p=2$.

\begin{table}[ht]
\vspace*{-0.25cm}
\centering
\footnotesize
\sisetup{
  table-format=2,
  detect-weight=true,
  detect-family=true
}
\renewcommand{\arraystretch}{1.0}
\begin{tabular}{l r| S S S S S S S S S S S}
\toprule
Distribution & $n$ & {$T_n(\hat{\lambda})$} & {$q_{n,\max}$} & {$T_n(\lambda_{\mathrm{CV}})$} & {$T_n(1)$} & {$T_n(4)$} & {$\mathrm{dKSD}$} & {$F_n$} & {$B_n$} & {$R_n$} & {$S_n$} & {PAD} \\
\midrule
$\mathrm{Unif}(\Spp)$ & 50 & 5.2 & 5.1 & 5.5 & 4.9 & 5.0 & 4.9 & 5.0 & 4.9 & 5.2 & 5.1 & 5.0 \\
 & 100 & 4.5 & 4.5 & 4.0 & 4.7 & 4.6 & 4.8 & 4.7 & 4.5 & 5.0 & 4.8 & 4.7 \\
\midrule
$\mathrm{vMF}(0.5)$ & 50 & \multicolumn{1}{S}{\bfseries 36.3} & 27.0 & 25.1 & 34.1 & 8.7 & 32.9 & 35.2 & 5.1 & \multicolumn{1}{S}{\bfseries 36.3} & 35.8 & 35.4 \\
 & 100 & 66.3 & 53.7 & 39.1 & 63.6 & 12.1 & 61.8 & 64.8 & 5.2 & \multicolumn{1}{S}{\bfseries 66.5} & 65.3 & 64.9 \\
\midrule
$\mathrm{Ca}(0.25)$ & 50 & \multicolumn{1}{S}{\bfseries 40.8} & 31.3 & 29.2 & 38.6 & 10.5 & 37.5 & 39.7 & 6.5 & \multicolumn{1}{S}{\bfseries 40.7} & 40.4 & 40.0 \\
 & 100 & 71.8 & 59.6 & 44.6 & 69.2 & 14.9 & 67.8 & 70.3 & 6.9 & \multicolumn{1}{S}{\bfseries 71.9} & 70.7 & 70.5 \\
\midrule
$\mathrm{W}(1)$ & 50 & 32.4 & 22.0 & 20.3 & 13.3 & 26.9 & 17.6 & 10.5 & \multicolumn{1}{S}{\bfseries 37.4} & 5.5 & 8.9 & 9.2 \\
 & 100 & 60.6 & 43.2 & 34.3 & 27.7 & 51.1 & 36.0 & 19.3 & \multicolumn{1}{S}{\bfseries 67.4} & 5.6 & 14.7 & 15.2 \\
\midrule
$\mathrm{SC}(0.5,0.5)$ & 50 & \multicolumn{1}{S}{\bfseries 29.3} & 22.3 & 21.0 & 28.5 & 10.0 & 28.2 & 28.9 & 8.0 & 29.0 & \multicolumn{1}{S}{\bfseries 29.4} & \multicolumn{1}{S}{\bfseries 29.2} \\
 & 100 & \multicolumn{1}{S}{\bfseries 57.2} & 44.1 & 30.9 & 56.2 & 14.2 & 55.5 & 56.6 & 10.1 & 56.0 & 56.6 & 56.5 \\
\midrule
$\mathrm{projNM}(5)$ & 50 & \multicolumn{1}{S}{\bfseries 70.1} & 46.8 & 45.4 & 7.1 & 19.0 & 8.6 & 7.9 & 14.2 & 5.4 & 6.4 & 7.8 \\
 & 100 & \multicolumn{1}{S}{\bfseries 98.7} & 92.6 & 89.3 & 9.4 & 38.3 & 12.5 & 11.0 & 24.1 & 5.2 & 7.0 & 10.4 \\
\midrule
$\mathrm{MvMF}_{2p}(30)$ & 50 & \multicolumn{1}{S}{\bfseries 100.0} & \multicolumn{1}{S}{\bfseries 100.0} & \multicolumn{1}{S}{\bfseries 100.0} & 6.3 & \multicolumn{1}{S}{\bfseries 100.0} & 17.2 & 35.2 & 8.7 & 5.0 & 6.5 & 39.5 \\
 & 100 & \multicolumn{1}{S}{\bfseries 100.0} & \multicolumn{1}{S}{\bfseries 100.0} & \multicolumn{1}{S}{\bfseries 100.0} & 6.9 & \multicolumn{1}{S}{\bfseries 100.0} & 37.6 & 96.5 & 8.8 & 5.4 & 8.1 & 99.2 \\
\midrule
$\mathrm{MvMF}_{2}(0.3)$ & 50 & \multicolumn{1}{S}{\bfseries 79.6} & 70.1 & 68.2 & 72.6 & 57.4 & 75.4 & 69.1 & 62.7 & 56.7 & 67.0 & 66.8 \\
 & 100 & \multicolumn{1}{S}{\bfseries 98.2} & 96.3 & 92.8 & 97.1 & 91.4 & 97.6 & 95.8 & 92.7 & 86.6 & 94.7 & 94.8 \\
\midrule
$\mathrm{SCM}(3)$ & 50 & \multicolumn{1}{S}{\bfseries 66.5} & 58.9 & 57.4 & 47.5 & 61.9 & 54.8 & 43.4 & 56.0 & 24.6 & 38.0 & 39.9 \\
 & 100 & \multicolumn{1}{S}{\bfseries 95.9} & 93.2 & 90.1 & 86.6 & 94.2 & 91.6 & 84.0 & 88.5 & 48.5 & 76.7 & 80.4 \\
\bottomrule
\end{tabular}
\caption{\small Empirical rejection percentages in dimension $p=3$ computed with $M=10,\!000$ Monte Carlo samples and at significance level $\alpha=5\%$. Bold entries indicate best-performing tests for each alternative.}
\label{table:powers_d3}
\vspace*{-0.25cm}
\end{table}

Overall, while $T_n(\lambda)$ is omnibus-consistent for all $\lambda>0$, an appropriate choice of tuning parameter substantially improves rejection rates. The tuned test consistently leads the competitors or matches the best one across most settings, with $\mathrm{W}(1)$ as the only exception. The advantage is most pronounced for multimodal and mixture alternatives. Figure~\ref{fig:power_lambda} illustrates the sensitivity to $\lambda$ in comparison to the softmax test \citep{Fernandez-de-Marcos2023b}, as well as the $\mathrm{dKSD}$ test discussed in Section~\ref{sec:dKSD}, using the same tuning parameter $\lambda$ in each test statistic. The proposed test has a narrower range of near-optimal parameters, but achieves, with optimal $\lambda$, the highest rejection rates in the alternatives considered.

\section{Discussion}
\label{sec:disc}

We introduced an $L^2$-Stein test statistic in the sense of \citet[Section~5.2]{Anastasiou2023} for testing uniformity on the sphere. A key feature of our approach is the dual role of the Laplace--Beltrami operator, as both the Stein operator for uniformity and an operator whose eigenfunctions are the spherical harmonics, allowing for elegant, explicit series representations of both the statistic and its asymptotic null and non-null distributions.

We conclude the paper by pointing out some directions for further research. Within the spherical uniformity setting, one may generalize the procedure either by changing the set of test functions or by applying a different norm to the underlying process to further adapt the sensitivity profiles. Extending the setting beyond the sphere, the derived Stein operator extends to general smooth compact manifolds with empty boundary, so the construction \eqref{eq:Tn(lambda)} carries over. In that setting, Laplace--Beltrami eigenfunctions still yield an orthogonal basis but, in general, we lose the explicit harmonic decomposition of the test functions and the use of the Funk--Hecke theorem to derive the coefficients. The approach also allows for goodness-of-fit tests for distributions other than the uniform by applying the operator as stated in \eqref{eq:operator}. In this more general scenario, one must additionally estimate unknown model parameters, and the resulting operator no longer admits a spherical-harmonic eigenbasis, so the explicit harmonic decomposition available under uniformity is lost.

\begin{table}[h!]
\vspace*{-0.25cm}
\centering
\footnotesize
\sisetup{
  table-format=2,
  detect-weight=true,
  detect-family=true
}
\renewcommand{\arraystretch}{1.0}
\begin{tabular}{l r| S S S S S S S S S S S}
\toprule
Distribution & $n$ & {$T_n(\hat{\lambda})$} & {$q_{n,\max}$} & {$T_n(\lambda_{\mathrm{CV}})$} & {$T_n(1)$} & {$T_n(4)$} & {$\mathrm{dKSD}$} & {$F_n$} & {$B_n$} & {$R_n$} & {$S_n$} & {PAD} \\
\midrule
$\mathrm{Unif}(\Spp)$ & 50 & 4.8 & 4.9 & 4.4 & 5.5 & 5.0 & 5.2 & 5.4 & 4.8 & 5.4 & 5.3 & 5.3 \\
 & 100 & 4.7 & 5.4 & 4.7 & 5.6 & 5.0 & 5.7 & 5.6 & 4.5 & 5.7 & 5.5 & 5.6 \\
\midrule
$\mathrm{vMF}(0.5)$ & 50 & 18.9 & 13.2 & 12.9 & 18.7 & 8.7 & 18.0 & 18.4 & 5.3 & \multicolumn{1}{S}{\bfseries 19.0} & 18.4 & 18.4 \\
 & 100 & 38.6 & 27.1 & 22.8 & 38.0 & 14.0 & 37.3 & 37.9 & 5.7 & \multicolumn{1}{S}{\bfseries 38.7} & 37.8 & 38.1 \\
\midrule
$\mathrm{Ca}(0.25)$ & 50 & \multicolumn{1}{S}{\bfseries 54.0} & 42.1 & 41.0 & \multicolumn{1}{S}{\bfseries 54.1} & 21.6 & 53.5 & 53.8 & 6.5 & \multicolumn{1}{S}{\bfseries 53.8} & 53.8 & \multicolumn{1}{S}{\bfseries 53.9} \\
 & 100 & \multicolumn{1}{S}{\bfseries 87.9} & 79.1 & 71.9 & 87.7 & 45.0 & 86.9 & 87.4 & 8.6 & \multicolumn{1}{S}{\bfseries 88.0} & 87.5 & 87.6 \\
\midrule
$\mathrm{W}(1)$ & 50 & 12.8 & 8.5 & 8.8 & 6.5 & 12.5 & 7.4 & 6.8 & \multicolumn{1}{S}{\bfseries 14.5} & 5.6 & 6.6 & 6.3 \\
 & 100 & 22.5 & 13.6 & 13.6 & 8.1 & 21.5 & 10.7 & 8.7 & \multicolumn{1}{S}{\bfseries 27.4} & 6.1 & 8.4 & 8.0 \\
\midrule
$\mathrm{SC}(0.5,0.5)$ & 50 & \multicolumn{1}{S}{\bfseries 16.5} & 11.3 & 11.0 & \multicolumn{1}{S}{\bfseries 16.3} & 8.4 & 16.0 & 16.1 & 5.6 & \multicolumn{1}{S}{\bfseries 16.5} & 16.1 & \multicolumn{1}{S}{\bfseries 16.2} \\
 & 100 & 31.1 & 21.6 & 17.9 & \multicolumn{1}{S}{\bfseries 31.5} & 12.7 & 31.1 & 31.2 & 6.4 & \multicolumn{1}{S}{\bfseries 31.4} & 31.1 & \multicolumn{1}{S}{\bfseries 31.5} \\
\midrule
$\mathrm{projNM}(5)$ & 50 & \multicolumn{1}{S}{\bfseries 100.0} & \multicolumn{1}{S}{\bfseries 100.0} & \multicolumn{1}{S}{\bfseries 100.0} & 84.8 & \multicolumn{1}{S}{\bfseries 100.0} & \multicolumn{1}{S}{\bfseries 100.0} & 98.9 & \multicolumn{1}{S}{\bfseries 100.0} & 8.0 & 95.3 & 85.5 \\
 & 100 & \multicolumn{1}{S}{\bfseries 100.0} & \multicolumn{1}{S}{\bfseries 100.0} & \multicolumn{1}{S}{\bfseries 100.0} & \multicolumn{1}{S}{\bfseries 100.0} & \multicolumn{1}{S}{\bfseries 100.0} & \multicolumn{1}{S}{\bfseries 100.0} & \multicolumn{1}{S}{\bfseries 100.0} & \multicolumn{1}{S}{\bfseries 100.0} & 8.6 & \multicolumn{1}{S}{\bfseries 100.0} & \multicolumn{1}{S}{\bfseries 100.0} \\
\midrule
$\mathrm{MvMF}_{2p}(30)$ & 50 & \multicolumn{1}{S}{\bfseries 100.0} & \multicolumn{1}{S}{\bfseries 100.0} & 100.0 & 5.6 & 99.9 & 10.6 & 19.7 & 9.7 & 5.4 & 6.7 & 18.9 \\
 & 100 & \multicolumn{1}{S}{\bfseries 100.0} & \multicolumn{1}{S}{\bfseries 100.0} & \multicolumn{1}{S}{\bfseries 100.0} & 6.6 & \multicolumn{1}{S}{\bfseries 100.0} & 18.5 & 52.3 & 10.3 & 6.0 & 9.2 & 50.7 \\
\midrule
$\mathrm{MvMF}_{2}(0.3)$ & 50 & \multicolumn{1}{S}{\bfseries 49.2} & 39.0 & 39.7 & 43.0 & 36.9 & 45.8 & 43.5 & 29.2 & 38.8 & 43.2 & 42.3 \\
 & 100 & \multicolumn{1}{S}{\bfseries 82.3} & 74.1 & 71.4 & 75.8 & 71.8 & 79.5 & 77.0 & 59.3 & 69.2 & 76.4 & 75.5 \\
\midrule
$\mathrm{SCM}(3)$ & 50 & \multicolumn{1}{S}{\bfseries 99.3} & 92.6 & 90.9 & 84.2 & 97.3 & 93.9 & 88.4 & 86.0 & 64.9 & 86.6 & 84.4 \\
 & 100 & \multicolumn{1}{S}{\bfseries 100.0} & \multicolumn{1}{S}{\bfseries 100.0} & \multicolumn{1}{S}{\bfseries 100.0} & \multicolumn{1}{S}{\bfseries 100.0} & \multicolumn{1}{S}{\bfseries 100.0} & \multicolumn{1}{S}{\bfseries 100.0} & \multicolumn{1}{S}{\bfseries 100.0} & \multicolumn{1}{S}{\bfseries 100.0} & 99.1 & \multicolumn{1}{S}{\bfseries 100.0} & \multicolumn{1}{S}{\bfseries 100.0} \\
\bottomrule
\end{tabular}
\caption{\small Empirical rejection percentages in dimension $p=5$ computed with $M=10,\!000$ Monte Carlo samples and at significance level $\alpha=5\%$. Bold entries indicate best-performing tests for each alternative.}
\label{table:powers_d5}
\vspace*{-0.25cm}
\end{table}

\begin{figure}[h!]
    \vspace*{-0.25cm}
    \centering
    \begin{subfigure}[b]{0.32\linewidth}
    \centering
    \includegraphics[width=\linewidth]{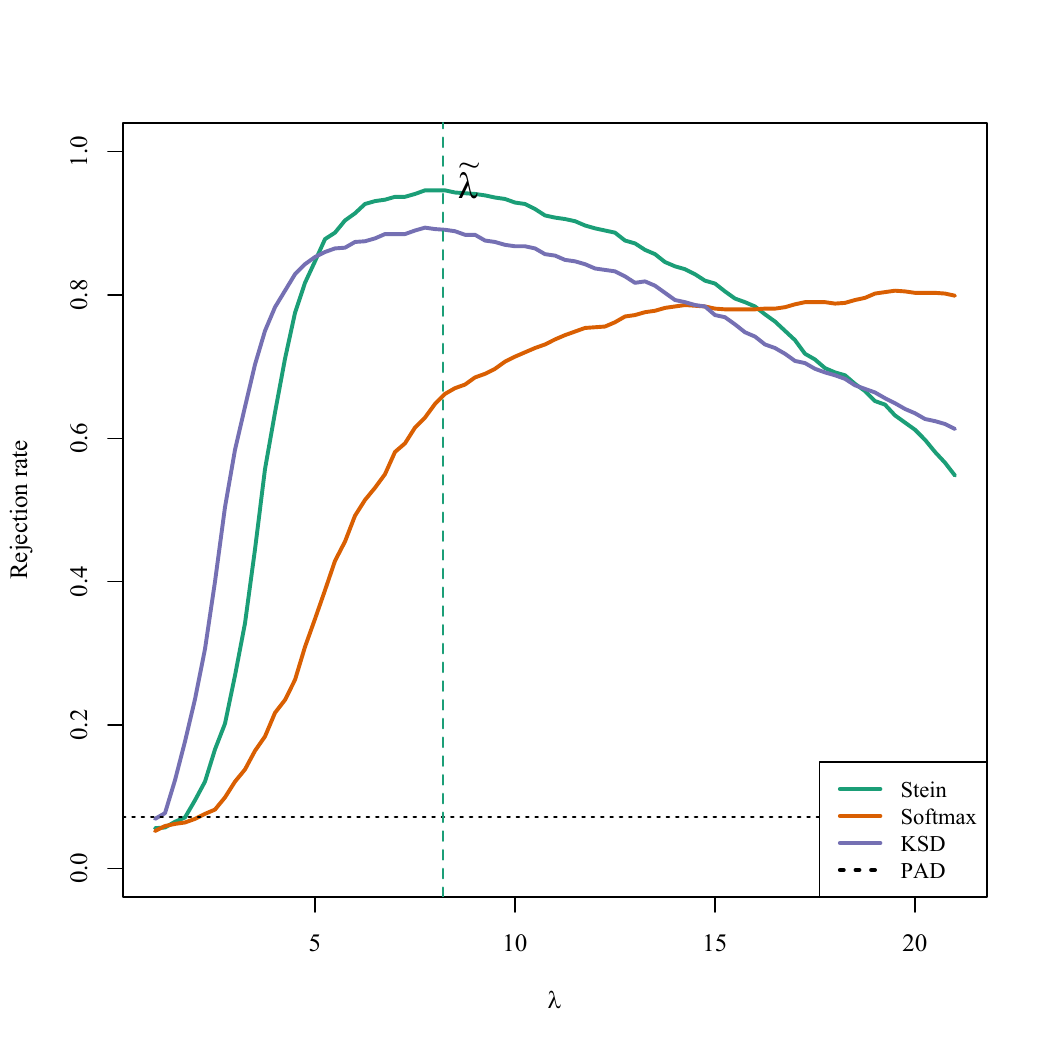}
    \end{subfigure}
    \begin{subfigure}[b]{0.32\linewidth}
    \centering
    \includegraphics[width=\linewidth]{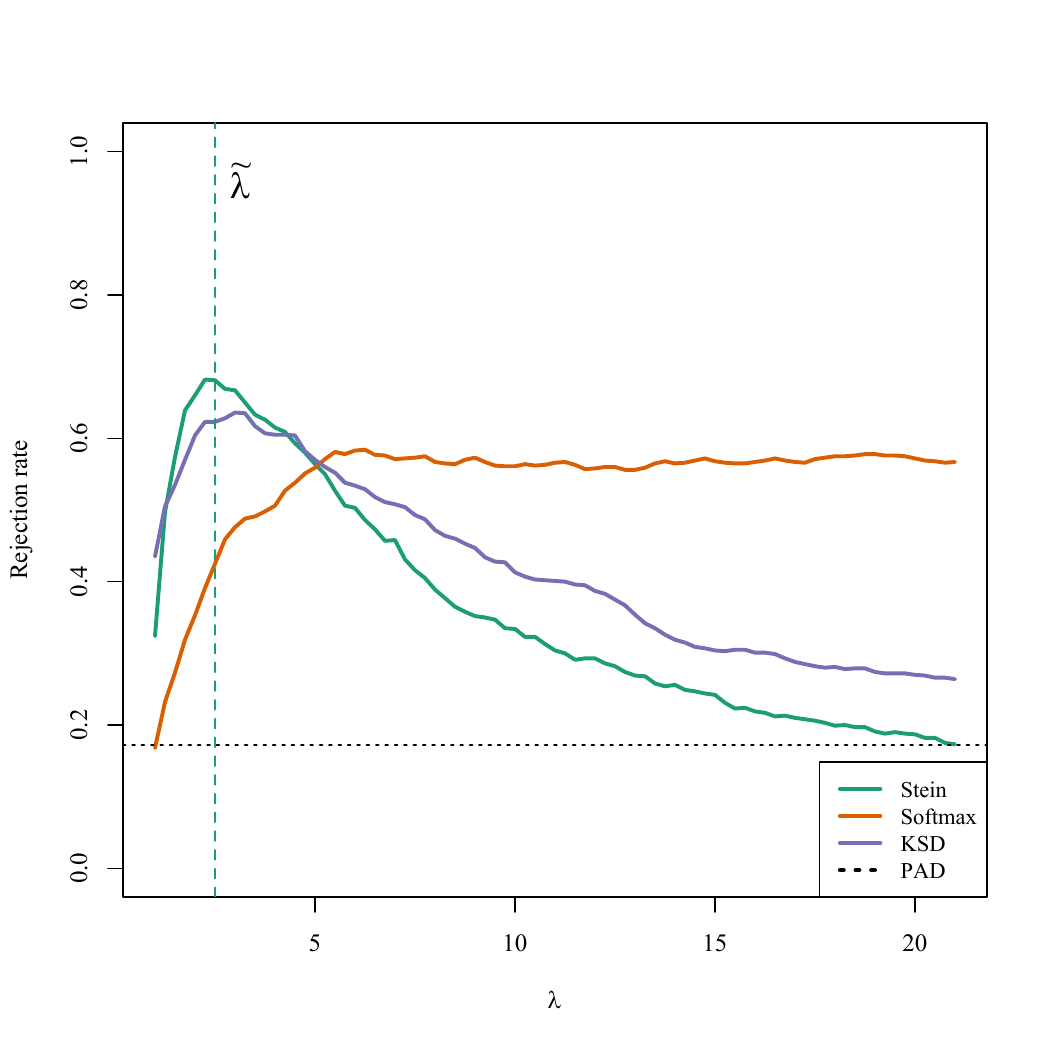}
    \end{subfigure}
    \begin{subfigure}[b]{0.32\linewidth}
    \centering
    \includegraphics[width=\linewidth]{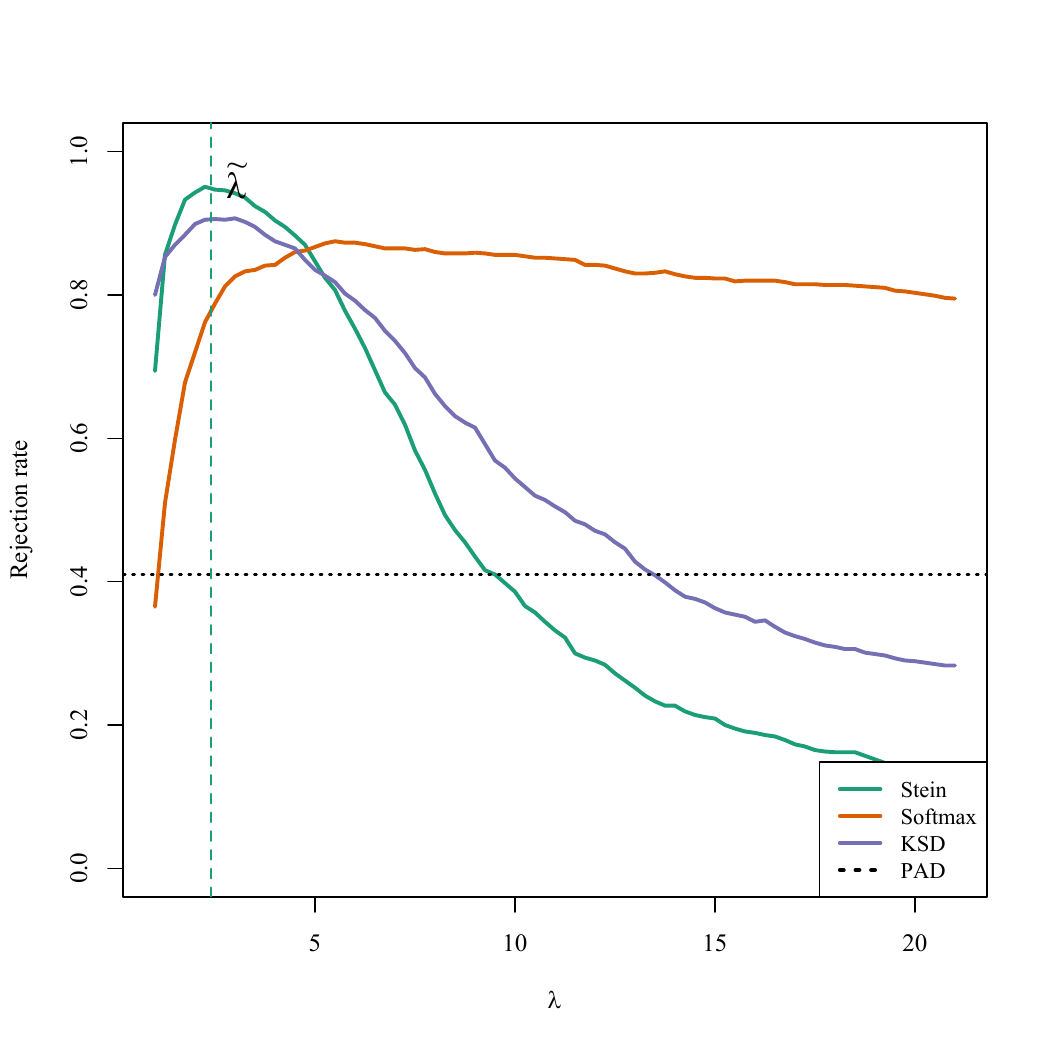}
    \end{subfigure}
    \begin{subfigure}[b]{0.32\linewidth}
    \centering
    \includegraphics[width=\linewidth]{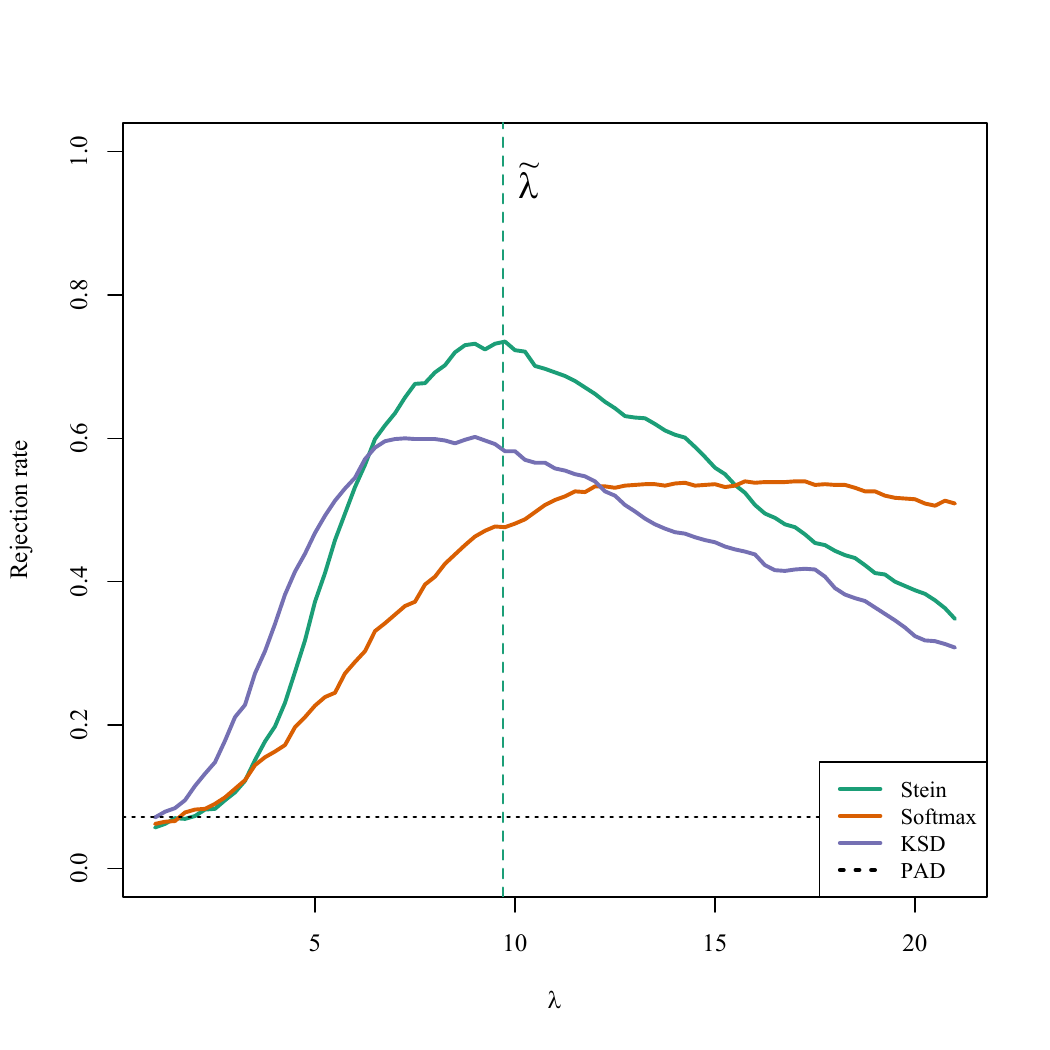}
    \end{subfigure}
    \begin{subfigure}[b]{0.32\linewidth}
    \centering
    \includegraphics[width=\linewidth]{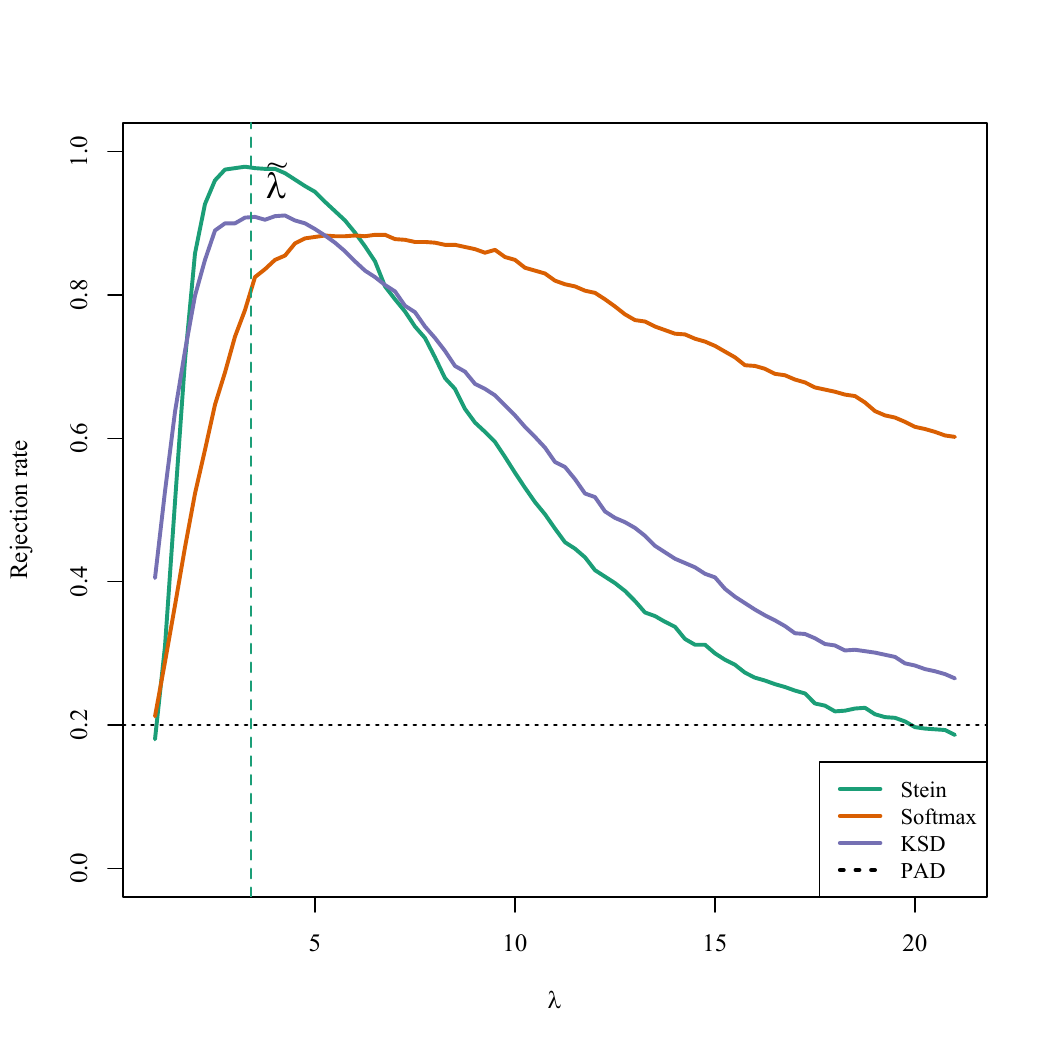}
    \end{subfigure}
    \begin{subfigure}[b]{0.32\linewidth}
    \centering
    \includegraphics[width=\linewidth]{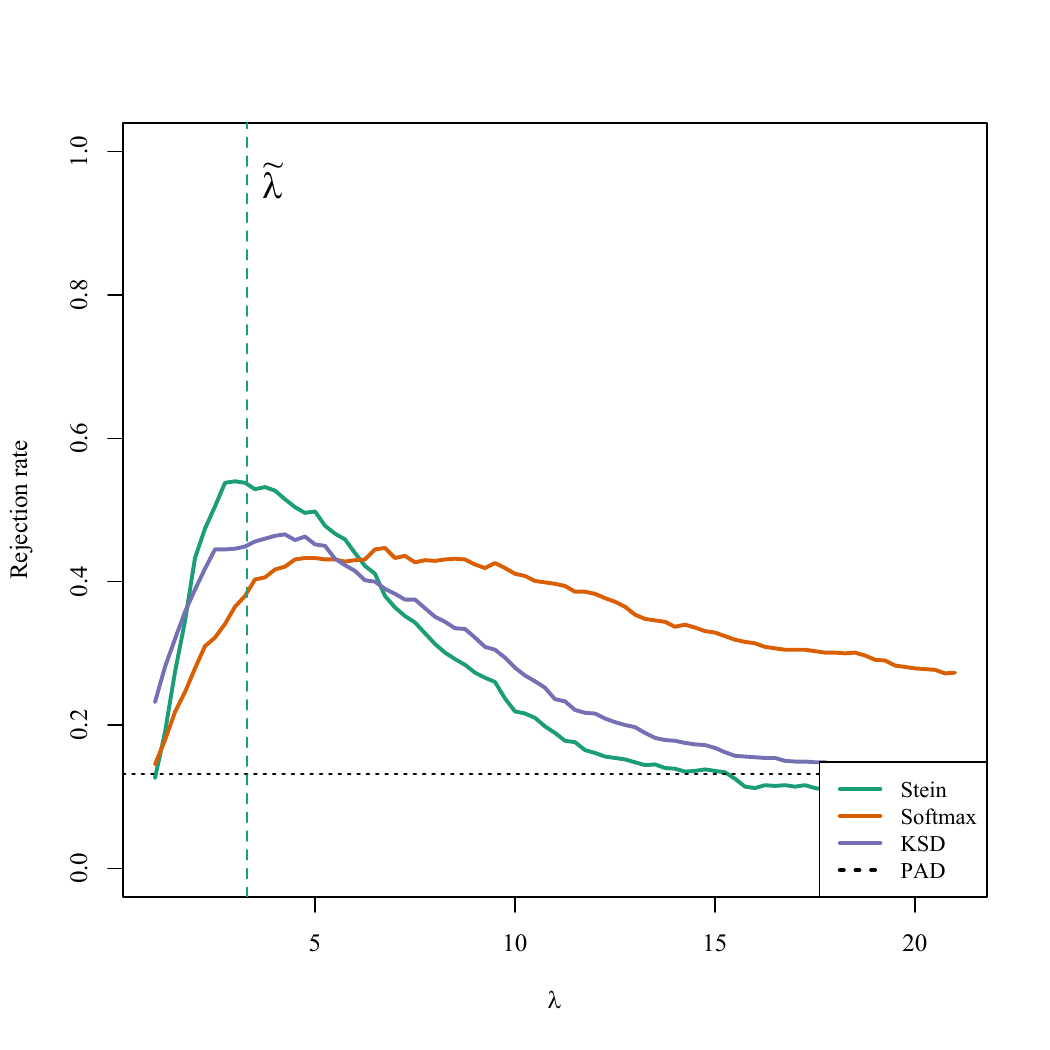}
    \end{subfigure}
    \vspace*{-0.25cm}
    \caption{\small Powers of the Stein, softmax, and $\mathrm{dKSD}$ tests, with concentration parameter $\lambda$ in each column under the alternative distributions $\mathrm{MvMF}_{2p}(10)$, $\mathrm{SCM}(3)$, and $\mathrm{W}(2)$. The top row corresponds to dimension $p=3$, and the bottom row to $p=5$. Here, we use significance level $\alpha=5\%$, sample size $n=50$, and $M=1,\!000$ samples.}\label{fig:power_lambda}
    \vspace*{-0.25cm}
\end{figure}

\section*{Acknowledgments}

The first two authors are funded by the Deutsche Forschungsgemeinschaft (DFG, German Research Foundation), grant 541565572. The third author acknowledges support from grant PCI2024-155058-2, funded by MICIU/AEI/10.13039/\-501100011033/UE.


\appendix
\vspace*{-0.25cm}
\section{Proofs}
\label{sec:proofs}

\begin{proof}{ of Proposition~\ref{prop:char}}
The first implication follows directly from the derivation of the Laplace--Beltrami operator as a Stein operator of the uniform law. In fact, for any smooth function $f$, the relation $\Es{\Delta_{\Spp}f(\bX)}{\mathcal{H}_0}=0$ holds. Choosing $f(\bx)=e^{\lambda \bt^\top \bx}$ yields $\Ebig{\Delta_{\Spp} e^{\lambda \bt^\top \bX}}=0$ for all $\bt\in\Spp$ and fixed $\lambda>0$.

Conversely, let $\bX$ be an $\Spp$-valued random vector such that $\Ebig{\Delta_{\Spp} e^{\lambda \bt^\top \bX}}=0$ for all $\bt\in\Spp$.
For $k\geq 0$ and $r\in \{1,\ldots,d_{k,p}\}$, define $\mu_{r,k}=\E{Y_{r,k}(\bX)}$ and by the addition formula $\E{C_k^{(p-2)/2}(\bt^\top \bX)}=\E{\sum_{r=1}^{d_{k,p}}\gamma_{k,p}Y_{r,k}(\bt)Y_{r,k}(\bX)}=\gamma_{k,p}\sum_{r=1}^{d_{k,p}}Y_{r,k}(\bt)\mu_{r,k}$. Using the harmonic expansion and its uniform convergence, we obtain
\begin{align*}
    \E{\Delta_{\Spp}e^{\lambda\bt^\top\bX}}=&\;\E{\sum_{k=0}^\infty (-k)(k+p-2)m_{k,p}(\lambda )C_k^{(p-2)/2}(\bt^\top \bX)}\\
    =&\;\sum_{k=0}^\infty (-k)(k+p-2)m_{k,p}(\lambda )\E{C_k^{(p-2)/2}(\bt^\top \bX)}\\
    =&\;\sum_{k=0}^\infty (-k)(k+p-2)m_{k,p}(\lambda )\gamma_{k,p}\sum_{r=1}^{d_{k,p}}Y_{r,k}(\bt)\mu_{r,k}.
\end{align*}
Since this function vanishes on $\Spp$, uniqueness of the expansion implies that $(-k)(k+p-2)m_{k,p}(\lambda )\gamma_{k,p}\mu_{r,k}=0$ for all $k\geq 1$ and $r\in \{1,\ldots,d_{k,p}\}$.
With $\gamma_{k,p}>0$ and $m_{k,p}(\lambda)>0$ for $k\geq 1$, we conclude $\mu_{r,k}=0$ for all $k\geq 1$, $r\in \{1,\ldots,d_{k,p}\}$. Hence, $\bX$ and the uniform distribution agree on the expectations of all spherical harmonics of positive degree, and clearly also on constants. Thus, they agree on all finite linear combinations of spherical harmonics, and hence on $C(\Spp)$ by density \citep[Section 2.2]{Dai2013}. Therefore, $\bX\sim \mathrm{Unif}(\Spp)$.
\end{proof}

\noindent \begin{proof}{ of Lemma~\ref{lem:cMdk}}
To derive the closed-form formulas of the coefficients, we use the series expansion~\eqref{eq:expansion_mkp}, the eigenfunction relation of the Gegenbauer polynomials, the uniform convergence of the series (see Remark~\ref{rem:unif.conv.}), and the Funk--Hecke formula \eqref{eq:Gegen_FH}, to see
\begin{align*}
    T_n(\lambda)=&\;n\int_{\Spp}\bigg(\frac{1}{n}\sum_{i=1}^n\Delta_{\Spp}e^{\lambda\bt^\top \bX_i}\bigg)^2\,\rd\nu_{p-1}(\bt)\\
    =&\;\frac{1}{n}\sum_{i,j=1}^n\int_{\Spp}\bigg(\sum_{k_1=0}^\infty m_{k_1,p}(\lambda)(-k_1)(k_1+p-2)C_{k_1}^{(p-2)/2}(\bt^\top \bX_i)\bigg) \\
    &\times\bigg(\sum_{k_2=0}^\infty m_{k_2,p}(\lambda)(-k_2)(k_2+p-2)C_{k_2}^{(p-2)/2}(\bt^\top \bX_j)\bigg)\,\rd\nu_{p-1}(\bt)\\
    =&\;\frac{1}{n}\sum_{i,j=1}^n\sum_{k_1,k_2=1}^\infty m_{k_1,p}(\lambda)(-k_1)(k_1+p-2)m_{k_2,p}(\lambda)(-k_2)(k_2+p-2)\\
    &\times \int_{\Spp} C_{k_1}^{(p-2)/2}(\bt^\top \bX_i)C_{k_2}^{(p-2)/2}(\bt^\top \bX_j)\,\rd\nu_{p-1}(\bt)\\
    =&\;\frac{1}{n}\sum_{i,j=1}^n\sum_{k=1}^\infty \big( m_{k,p}(\lambda)k(k+p-2)\big)^2 \gamma_{k,p}C_k^{(p-2)/2}(\bX_i^\top \bX_j).
\end{align*}
By plugging in expression~\eqref{eq:m_k}, the resulting coefficients for $p\ge 3$ are
\begin{align*}
    c_{k,p}(\lambda)=&\;\big( m_{k,p}(\lambda)k(k+p-2)\big)^2\gamma_{k,p}\\
    =&\;2^{p-3}\lambda^{2-p}(p-2)\left(k+\frac{p-2}{2}\right)\left(\Gamma\left(\frac{p-2}{2}\right) k(k+p-2)\Ical_{(p-2)/2+k}(\lambda)\right)^2,
\end{align*}
and for $p=2$
\begin{align*}
    c_{k,2}(\lambda)=&\;\left(m_{k,2}(\lambda)k^2\right)^2 \frac{1+1_{\{k=0\}}}{2}=\left(\frac{\Ical_k(\lambda)}{(1+1_{\{k=0\}})/2}k^2\right)^2 \frac{1+1_{\{k=0\}}}{2}
    =(2-1_{\{k=0\}})k^4\Ical_k(\lambda)^2.
\end{align*}
\end{proof}
\noindent \begin{proof}{ of Proposition~\ref{prop:truncation}}
    For sufficiently large $K\in\N$, with the asymptotic form of the modified Bessel function \citep[10.41.1]{NIST:DLMF}, and in the case $p>2$ \citep[18.14.4]{NIST:DLMF} for the bound on Gegenbauer polynomials, we obtain $c_{k,p}(\lambda)C_k^{(p-2)/2}(1)=O\left(k^{p+2}\left((e\lambda)/(2k+p-2)\right)^{2k+p-2}\right)$. So for constants $C,C'$ only depending on $p$ and $\lambda$, using the geometric series
\begin{align*}
    \abs{T_n(\lambda)-T_{n,K}(\lambda)}\le&\;\frac{1}{n} \sum_{k=K+1}^\infty\sum_{i,j=1}^n c_{k,p}(\lambda)\abs{C_k^{(p-2)/2}(\bX_i^\top \bX_j)} 
    \le n C\sum_{k=K+1}^\infty k^{p+2}\lrp{\frac{e\lambda}{2k+p-2}}^{2k+p-2}\\
    \le&\; n C'\sum_{k=K+1}^\infty  \lrp{\frac{e\lambda}{2K}}^{k}= n C'\frac{\lrp{e\lambda/(2K)}^{K+1}}{1-e\lambda/(2K)}\le nC'\lrp{\frac{e\lambda}{2K}}^K,
\end{align*}
assuming $K>e\lambda/2$. The case $p=2$ is analogous using \citep[18.3.1]{NIST:DLMF}.

Taking logarithms, we obtain $|T_{n,K_n}(\lambda) - T_n(\lambda)| \to 0$ for any sequence $(K_n)$ satisfying $K_n\log(K_n)-\log(n)\to \infty$ as $n\to \infty$. Thus, $K_n \ge c\log n$ for some constant $c > 0$, is a sufficient condition.
\end{proof}

\noindent \begin{proof}{ of Theorem~\ref{thm:1_mgf}}
This result follows from the central limit theorem in Hilbert spaces \citep[Theorem 17.29]{henze2024}, since for
\begin{align}
    \Psi(\bt,\bx)=\Delta_{\Spp}e^{\lambda\bt^\top \bx}=\sum_{k=0}^\infty m_{k,p}(\lambda)(-k)(k+p-2)C_k^{(p-2)/2}(\bt^\top \bx),\label{eq:Psi}
\end{align}
we have $W_n(\bt)=n^{-1/2}\sum_{i=1}^n\Psi(\bt,\bX_i)$. Here, the summands are iid and centered elements of $L^2(\Spp)$, i.e., $\E{\Psi(\cdot,\bX)}=0$, since $\E{\Psi(\bt,\bX)}=0 $ for all $\bt\in \Spp$, and have finite second moment
\begin{align}
     \E{\big\|\Psi(\cdot,\bX)\big\|_{L^2(\Spp)}^2} = &\;\sum_{k=0}^\infty \big( m_{k,p}(\lambda)(-k)(k+p-2)\big)^2\,\Ebigg{\int_{\Spp}\!\!\big(C_k^{(p-2)/2}(\bt^\top \bX)\big)^2\rd\nu_{p-1}(\bt)}\label{eq:fin_moment}\\
    =&\; \sum_{k=1}^\infty c_{k,p}(\lambda)C_k^{(p-2)/2}(1)<\infty.\nonumber
\end{align}
This allows direct application of \citet[Theorem 17.29]{henze2024}, implying $W_n\xrightarrow[]{d}\mathcal{W}$ for a centered Gaussian process $\mathcal{W}$ with the covariance kernel
\begin{align*}
K(\bs,\bt)=&\;\E{\Psi(\bs,\bX)\Psi(\bt,\bX)}\\
=&\;\sum_{k,\ell=0}^\infty m_{k,p}(\lambda)(-k)(k+p-2) m_{\ell,p}(\lambda)(-\ell)(\ell+p-2)\E{C_k^{(p-2)/2}(\bs^\top \bX)C_\ell^{(p-2)/2}(\bt^\top \bX)}\\
=&\;\sum_{k=0}^\infty \big( m_{k,p}(\lambda)k(k+p-2)\big)^2\gamma_{k,p}C_k^{(p-2)/2}(\bs^\top \bt) =\sum_{k=0}^\infty c_{k,p}(\lambda)C_k^{(p-2)/2}(\bs^\top \bt).
\end{align*}
\end{proof}

\noindent \begin{proof}{ of Theorem~\ref{thm:2_MGF}}
We prove this result using the Karhunen--Loève expansion \citep[Theorem 17.26]{henze2024} to the limiting Gaussian element $\mathcal{W}$ from Theorem~\ref{thm:1_mgf}. To this end, we determine the eigenfunctions and eigenvalues of the covariance operator $C$ defined by
\begin{align*}
    Cf(\bx)=\int_{\Spp}K(\bs,\bx)f(\bs)\,\rd\nu_{p-1}(\bs),\quad \bx\in\Spp,\quad f\in L^2(\Spp).
\end{align*}
This can be done for the basis $\{Y_{r,k}:r=1,\ldots,d_{k,p} \}$ of spherical harmonic functions of degree $k\in\N_0$, by using the Funk--Hecke formula \citep[Theorem 1.2.9]{Dai2013} (for $1\le r\le d_{k,p}$):
\begin{align}
    \int_{\Spp} C_k^{(p-2)/2}(\bt^\top \bx)Y_{r,k}(\bx)\,\rd\nu_{p-1}(\bx)=\frac{\omega_{p-2}h_{k,p}}{\omega_{p-1}C_k^{(p-2)/2}(1)}
\,Y_{r,k}(\bt)=\gamma_{k,p}Y_{r,k}(\bt),\label{eq:Gegen_SH}
\end{align}
$h_{k,p}=\|C_k^{(p-2)/2}\|^2_{L^{2,p}}$ and uniform convergence, to obtain
\begin{align*}
CY_{r,k}(\bx)=&\;\int_{\Spp}\sum_{m=0}^\infty c_{m,p}(\lambda)C_m^{(p-2)/2}(\bs^\top \bx)Y_{r,k}(\bs)\,\rd\nu_{p-1}(\bs)\\
=&\;\sum_{m=0}^\infty c_{m,p}(\lambda)\int_{\Spp}C_m^{(p-2)/2}(\bs^\top \bx)Y_{r,k}(\bs)\,\rd\nu_{p-1}(\bs)\\
=&\;\sum_{m=0}^\infty 1_{\{m=k\}}c_{m,p}(\lambda)\gamma_{k,p} Y_{r,k}(\bx)=c_{k,p}(\lambda)\gamma_{k,p} Y_{r,k}(\bx).
\end{align*}
Hence, $Y_{r,k}$ is an eigenfunction of $C$ with eigenvalue $c_{k,p}(\lambda)\gamma_{k,p}$. Since the eigenspace, for any degree $k$, has dimension $d_{k,p}$, the eigenvalue $c_{k,p}(\lambda)\gamma_{k,p}$ has multiplicity $d_{k,p}$. Theorem~\ref{thm:1_mgf} yields $T_n(\lambda)\xrightarrow{d}\|\mathcal{W}\|^2_{L^2(\Spp)}$, for the centered Gaussian element~$\mathcal{W}$, while $\Ebig{\langle\mathcal{W},Y_{r,k}\rangle_{L^2(\Spp)}^2}=c_{k,p}(\lambda)\gamma_{k,p}$, so $\langle\mathcal{W},Y_{r,k}\rangle_{L^2(\Spp)}\sim\mathcal{N}(0,c_{k,p}(\lambda)\gamma_{k,p})$. Therefore, applying \citet[Theorem 17.26]{henze2024} yields $\|\mathcal{W}\|^2_{L^2(\Spp)} =\sum_{k=0}^\infty \sum_{r=1}^{d_{k,p}}\langle\mathcal{W},Y_{r,k}\rangle^2_{L^2(\Spp)}=\sum_{k=1}^\infty \sum_{r=1}^{d_{k,p}}c_{k,p}(\lambda)\gamma_{k,p}N_{k,r}^2$, where all $N_{k,r}$ are independent standard normal random variables, so the stated result follows.
\end{proof}

\noindent \begin{proof}{ of Lemma~\ref{lem:mgf_alt}}
Let $\bt\in\R^p\setminus\{\zero\}$ and write $\bu:=\bt/\|\bt\|\in\Spp$. Using the harmonic decomposition \eqref{eq:betark}, its convergence in $L^2(\Spp)$ and the Funk--Hecke formula \eqref{eq:Gegen_FH}, we obtain
\begin{align*}
    M_{\bX}(\lambda\bt)=&\;\int_{\Spp}e^{\lambda\bt^\top \bx}q(\bx)\,\rd\nu_{p-1}(\bx)=\int_{\Spp}e^{\lambda\bt^\top \bx}\sum_{k=0}^\infty\sum_{r=1}^{d_{k,p}}\beta_{r,k}Y_{r,k}(\bx)\,\rd\nu_{p-1}(\bx)\\
    =&\;\sum_{k=0}^\infty\sum_{r=1}^{d_{k,p}}\beta_{r,k}\int_{\Spp}e^{\lambda\bt^\top \bx}Y_{r,k}(\bx)\,\rd\nu_{p-1}(\bx)
    =\sum_{k=0}^\infty\sum_{r=1}^{d_{k,p}}\beta_{r,k}\int_{\Spp}e^{\lambda\|\bt\|\bu^\top \bx}Y_{r,k}(\bx)\,\rd\nu_{p-1}(\bx)\\
    =&\;\sum_{k=0}^\infty\sum_{r=1}^{d_{k,p}}\beta_{r,k}m_{k,p}(\lambda\|\bt\|)\int_{\Spp} C_k^{(p-2)/2}\big(\bu^\top \bx\big) Y_{r,k}(\bx)\,\rd\nu_{p-1}(\bx)\\
    =&\;\sum_{k=0}^\infty\sum_{r=1}^{d_{k,p}}\beta_{r,k}m_{k,p}(\lambda\|\bt\|) \gamma_{k,p}
    Y_{r,k}\left(\bu\right)
\end{align*}
in $L^2(\Spp)$.

Fixing $\|\bt\|=1$, by restricting the function to $\Spp$, the only dependence of $M_{\bX}$ on $\bt$ is in the spherical harmonic $Y_{r,k}$. Here, we use \eqref{eq:Psi}, decomposition \eqref{eq:betark}, and uniform convergence to derive for $\bs\in\Spp$
\begin{align*}
    z(\bs)=&\;\Delta_{\Spp}M_{\bX}(\lambda\bs)=\int_{\Spp}\Delta_{\Spp,\bs}e^{\lambda\bs^\top \bx}q(\bx)\,\rd\nu_{p-1}(\bx)\\
    =&\;\int_{\Spp}\sum_{k=0}^\infty m_{k,p}(\lambda)(-k)(k+p-2)C_k^{(p-2)/2}(\bs^\top\bx)\sum_{\ell=0}^\infty\sum_{r=1}^{d_{\ell,p}}\beta_{r,\ell}Y_{r,\ell}(\bx)\,\rd\nu_{p-1}(\bx)\\
    =&\;\sum_{k=0}^\infty\sum_{\ell=0}^\infty\sum_{r=1}^{d_{\ell,p}}m_{k,p}(\lambda)(-k)(k+p-2)\beta_{r,\ell}\int_{\Spp} C_k^{(p-2)/2}(\bs^\top\bx)Y_{r,\ell}(\bx)\rd\nu_{p-1}(\bx)\\
    =&\;\sum_{k=0}^\infty\sum_{r=1}^{d_{k,p}}\beta_{r,k}m_{k,p}(\lambda)\gamma_{k,p}(-k)(k+p-2)Y_{r,k}(\bs).
\end{align*}
\end{proof}

\noindent \begin{proof}{ of Theorem~\ref{thm:tau:mgf}}
For the process $W_n(\bt)$, we use \eqref{eq:Psi} to write $W_n(\bt)=n^{-1/2}\sum_{i=1}^n\Psi(\bt,\bX_i)$. Since $\Ebig{\|\Psi(\bt,\bX)\|_{L^2(\Spp)}^2} =\sum_{k=0}^\infty c_{k,p}(\lambda)C_k^{(p-2)/2}(1)<\infty$; see \eqref{eq:fin_moment}, by the strong law of large numbers in Hilbert spaces \citep[Theorem~17.15]{henze2024} we obtain
\begin{align*}
    \frac{T_n(\lambda)}{n}=\frac{\|W_n\|_{L^2(\Spp)}^2}{n}=\bigg\|\frac{1}{n}\sum_{i=1}^n \Psi(\cdot,\bX_i)\bigg\|_{L^2(\Spp)}^2\xrightarrow[]{a.s.}\|z\|_{L^2(\Spp)}^2,\quad n\to\infty.
\end{align*}
In Lemma~\ref{lem:mgf_alt}, we represent $z$ in terms of spherical harmonics in $L^2(\Spp)$. In particular, $\langle z,Y_{r,k}\rangle_{L^2(\Spp)}=\left(\beta_{r,k}m_{k,p}(\lambda)\gamma_{k,p}(-k)(k+p-2)\right)$. Since the spherical harmonics form an orthonormal basis of $L^2(\Spp)$, Parseval's identity yields
\begin{align}
    \tau=&\;\|z\|_{L^2(\Spp)}^2 =\sum_{k=0}^\infty\sum_{r=1}^{d_{k,p}}\langle z,Y_{r,k}\rangle_{L^2(\Spp)}^2=\sum_{k=0}^\infty\sum_{r=1}^{d_{k,p}}\left(\beta_{r,k}m_{k,p}(\lambda)\gamma_{k,p}(-k)(k+p-2)\right)^2.\label{eq:tau_parseval}
\end{align}
By definition \eqref{eq:Tn_gegen}, this proves the claim.
\end{proof}

\noindent \begin{proof}{ of Theorem~\ref{thm:3mgf}}
We write the centered process by definition as
\begin{align*}
\left(W_n-\sqrt{n}z\right)=\frac{1}{\sqrt{n}}\sum_{i=1}^n\big(\Delta_{\Spp}e^{\lambda\bt^\top \bX_i}-\Delta_{\Spp}M_{\bX}(\lambda\bt)\big),
\end{align*}
where the summands $\Delta_{\Spp}e^{\lambda\bt^\top \bX_i}-\Delta_{\Spp}M_{\bX}(\lambda\bt)$ are iid, centered elements of $L^2(\Spp)$ with finite second moment by \eqref{eq:fin_moment}. Therefore, the central limit theorem in Hilbert spaces \citep[Theorem~17.29]{henze2024} yields a centered Gaussian limit in $L^2(\Spp)$, with covariance kernel given by
\begin{align*}
    K'(\bs,\bt)=&\;\E{\left(\Delta_{\Spp}e^{\lambda\bs^\top \bX}-\Delta_{\Spp}M_{\bX}(\lambda\bs)\right)\left(\Delta_{\Spp}e^{\lambda\bt^\top \bX}-\Delta_{\Spp}M_{\bX}(\lambda\bt)\right)}\\
    =&\;\E{\Delta_{\Spp}e^{\lambda\bs^\top \bX}\Delta_{\Spp}e^{\lambda\bt^\top \bX}}-\Delta_{\Spp}M_{\bX}(\lambda\bs)\Delta_{\Spp}M_{\bX}(\lambda\bt).
\end{align*}

By symmetry of the zonal kernel $e^{\lambda\bs^\top\bx}$, we have $\Delta_{\Spp,\bx}e^{\lambda\bs^\top\bx}=\Delta_{\Spp,\bs}e^{\lambda\bs^\top\bx}$, and analogously $\Delta_{\Spp,\bx}e^{\lambda\bt^\top\bx}=\Delta_{\Spp,\bt}e^{\lambda\bt^\top\bx}$. Since the operators $\Delta_{\Spp,\bs}$ and $\Delta_{\Spp,\bt}$ act on different variables, this implies
$\Ebig{\Delta_{\Spp}e^{\lambda\bs^\top \bX}\Delta_{\Spp}e^{\lambda\bt^\top \bX}}=\Delta_{\Spp,\bs}\Delta_{\Spp,\bt}M_{\bX}\big(\lambda(\bs+\bt)\big)$.

For the second representation, we expand the initial representation in Gegenbauer polynomials
\begin{align*}
    \E{\Delta_{\Spp}e^{\lambda\bs^\top \bX}\Delta_{\Spp}e^{\lambda\bt^\top \bX}}
    =&\;\sum_{k_1=1}^\infty\sum_{k_2=1}^\infty \big(m_{k_1,p}(\lambda)(-k_1)(k_1+p-2)\big)
     \big(m_{k_2,p}(\lambda)(-{k_2})({k_2}+p-2)\big)\\
    &\times\E{C_{k_1}^{(p-2)/2}(\bs^\top \bX)C_{k_2}^{(p-2)/2}(\bt^\top \bX)}.
\end{align*}
This yields the second representation, with $\xi_{k_1,k_2}(\bs,\bt)=\Ebig{C_{k_1}^{(p-2)/2}(\bs^\top \bX)C_{k_2}^{(p-2)/2}(\bt^\top \bX)}$ .
\end{proof}

\noindent \begin{proof}{ of Theorem~\ref{thm:var:mgf}}
We proceed as in \cite{BEH:2017} to obtain the distribution of the centered test statistic,
\begin{align*}
    \sqrt{n}\bigg(\frac{T_n(\lambda)}{n}-\tau\bigg)=&\;\sqrt{n}\bigg(\bigg\|\frac{W_n}{\sqrt{n}}\bigg\|^2_{L^2(\Spp)}-\|z\|_{L^2(\Spp)}^2\bigg)=\sqrt{n}\bigg\langle\frac{W_n}{\sqrt{n}}-z,\frac{W_n}{\sqrt{n}}+z\bigg\rangle_{L^2(\Spp)}\\
    =&\;\sqrt{n}\bigg\langle\frac{W_n}{\sqrt{n}}-z,2z+\frac{W_n}{\sqrt{n}}-z\bigg\rangle_{L^2(\Spp)}\\
    =&\;2\bigg\langle\sqrt{n}\bigg(\frac{W_n}{\sqrt{n}}-z\bigg),z\bigg\rangle_{L^2(\Spp)}+\frac{1}{\sqrt{n}}\bigg\|\sqrt{n}\bigg(\frac{W_n}{\sqrt{n}}-z\bigg)\bigg\|^2_{L^2(\Spp)}.
\end{align*}
In Theorem~\ref{thm:3mgf}, we saw the convergence of $\sqrt{n}\left(W_n/\sqrt{n}-z\right)$ to a centered Gaussian element $\mathcal{W}'$ in $L^2(\Spp)$. By Slutsky's lemma and the continuous mapping theorem, $\sqrt{n}\left(T_n(\lambda)/n-\tau\right)\xrightarrow[]{d}2\left\langle\mathcal{W}',z\right\rangle$. Here, $2\left\langle\mathcal{W}',z\right\rangle$ is centered Gaussian with variance $\Ebig{4\left\langle\mathcal{W}',z\right\rangle^2}$. With Fubini's theorem and applying $K'$ from Theorem~\ref{thm:3mgf}, we derive
\begin{align*}
    \sigma^2=&\; 4\int_{\Spp}\int_{\Spp}\E{\mathcal{W}'(\bs)\mathcal{W}'(\bt)}
    {z(\bs)z(\bt)}\,\rd\nu_{p-1}(\bs)\,\rd\nu_{p-1}(\bt)\\
    =&\;4\int_{\Spp}\int_{\Spp}K'(\bs,\bt){z(\bs)z(\bt)}\,\rd\nu_{p-1}(\bs)\,\rd\nu_{p-1}(\bt)\\
    =&\;4\int_{\Spp}\int_{\Spp}\bigg(\sum_{k_1=0}^\infty\sum_{k_2=0}^\infty \big(m_{k_1,p}(\lambda)(-k_1)(k_1+p-2)\big)
     \big(m_{k_2,p}(\lambda)(-{k_2})({k_2}+p-2)\big)\xi_{k_1,k_2}(\bs,\bt)\bigg)\\
     &\times{z(\bs)z(\bt)}\,\rd\nu_{p-1}(\bs)\,\rd\nu_{p-1}(\bt)-4\int_{\Spp}\int_{\Spp}z(\bs)z(\bt){z(\bs)z(\bt)}\,\rd\nu_{p-1}(\bs)\,\rd\nu_{p-1}(\bt).
\end{align*}

We start by considering the first term in $\sigma^2$ and simplify the double integral by applying Fubini's theorem:
\begin{align*}
    \int_{\Spp}\int_{\Spp}&\xi_{k_1,k_2}(\bs,\bt)z(\bs)z(\bt)\,\rd\nu_{p-1}(\bs)\,\rd\nu_{p-1}(\bt)\\
    =&\;\E{\int_{\Spp}C_{k_1}^{(p-2)/2}(\bs^\top \bX)z(\bs)\,\rd\nu_{p-1}(\bs)\int_{\Spp}C_{k_2}^{(p-2)/2}(\bt^\top \bX)z(\bt)\,\rd\nu_{p-1}(\bt)}.
\end{align*}
Let $\bY$ denote an independent copy of $\bX$, appearing in the definition of $z$, to separate the expectations. By applying Lemma~\ref{lem:mgf_alt} and changing the order of integration, we derive
\begin{align*}
    \int_{\Spp}&C_{k_1}^{(p-2)/2}(\bt^\top \bx)z(\bt)\,\rd\nu_{p-1}(\bt)
    =\;\mathbb{E}_{\bY}\left[\int_{\Spp}C_{k_1}^{(p-2)/2}(\bt^\top \bx)\Delta_{\Spp}e^{\lambda \bt^\top \bY}\,\rd\nu_{p-1}(\bt)\right]\\
    =&\;\mathbb{E}_{\bY}\left[\int_{\Spp}C_{k_1}^{(p-2)/2}(\bt^\top \bx)\sum_{{k_2}=0}^\infty \big( m_{{k_2},p}(\lambda)(-{k_2})({k_2}+p-2)\big)C_{k_2}^{(p-2)/2}(\bt^\top \bY)\,\rd\nu_{p-1}(\bt)\right]\\
    =&\;\big(m_{k_1,p}(\lambda)(-{k_1})({k_1}+p-2)\big)\gamma_{k_1,p}\int_{\Spp}C_{k_1}^{(p-2)/2}(\by^\top \bx)q(\by)\,\rd\nu_{p-1}(\by).
\end{align*}
Here, by \eqref{eq:betark} and \eqref{eq:Gegen_SH}, $\int_{\Spp}C_{k}^{(p-2)/2}(\by^\top \bx)q(\by)\,\rd\nu_{p-1}(\by)  =\gamma_{k,p}\sum_{r=1}^{d_{k,p}}\beta_{r,k}Y_{r,k}(\bx)$.
Combining these results, we obtain
\begin{align*}
    &\int_{\Spp}\int_{\Spp}\xi_{k_1,k_2}(\bs,\bt)z(\bs)z(\bt)\,\rd\nu_{p-1}(\bs)\,\rd\nu_{p-1}(\bt)\\
    =&\;\mathbb{E}\bigg[\big(m_{k_1,p}(\lambda)(-{k_1})({k_1}+p-2)\big)\gamma_{k_1,p}^2
\Big(\sum_{r_1=1}^{d_{k_1,p}}Y_{r_1,k_1}(\bX)\beta_{r_1,k_1} \Big) \\
&\times\big(m_{k_2,p}(\lambda)(-{k_2})({k_2}+p-2)\big)
    \gamma_{k_2,p}^2
\Big(\sum_{r_2=1}^{d_{k_2,p}}Y_{r_2,k_2}(\bX)\beta_{r_2,k_2}\Big)\bigg]\\
=&\;\sum_{r_1=1}^{d_{k_1,p}}\sum_{r_2=1}^{d_{k_2,p}}\beta_{r_1,k_1}\left(m_{k_1,p}(\lambda)(-{k_1})({k_1}+p-2)\right)\gamma_{k_1,p}^2\beta_{r_2,k_2}\big(m_{k_2,p}(\lambda)(-{k_2})({k_2}+p-2)\big)\\
&\gamma_{k_2,p}^2\E{Y_{r_1,k_1}(\bX)Y_{r_2,k_2}(\bX)}.
\end{align*}
Taking the sum over all $k_1,k_2$, the first term leads to
\begin{align*}
    \sum_{k_1=0}^\infty\sum_{k_2=0}^\infty\sum_{r_1=1}^{d_{k_1,p}}\sum_{r_2=1}^{d_{k_2,p}} &\big(\gamma_{k_1,p}m_{k_1,p}(\lambda)(-k_1)(k_1+p-2)\big)^2\big(\gamma_{k_2,p}m_{k_2,p}(\lambda)(-{k_2})({k_2}+p-2)\big)^2\\
        &
     \times \beta_{r_1,k_1}
    \beta_{r_2,k_2}
\E{Y_{r_1,k_1}(\bX)Y_{r_2,k_2}(\bX)}\\
=&\;\sum_{k_1=0}^\infty\sum_{k_2=0}^\infty\sum_{r_1=1}^{d_{k_1,p}}\sum_{r_2=1}^{d_{k_2,p}} \gamma_{k_1,p}c_{k_1,p}\gamma_{k_2,p}c_{k_2,p}
     \beta_{r_1,k_1}
    \beta_{r_2,k_2}
\E{Y_{r_1,k_1}(\bX)Y_{r_2,k_2}(\bX)}.
\end{align*}

In the case of rotationally symmetric alternatives, we use the linearization formula \citet[Equation 18.18.22]{NIST:DLMF} to get the expression in Remark \ref{rem:var_rot}, since
\begin{align*}
    &\E{C_{k_1}^{(p-2)/2}(\bmu^\top\bX)C_{k_2}^{(p-2)/2}(\bmu^\top\bX)}\\
    =&\;\sum_{k_3=0}^\infty \beta_{k_3}\int_{\Spp}\sum_{\ell=0}^{\min(k_1,k_2)}L_{k_1,k_2}^{(p)}(\ell)C_{k_1+k_2-2\ell}^{(p-2)/2}(\bmu^\top \bx)C_{k_3}^{(p-2)/2}(\bmu^\top \bx)\,\rd\nu_{p-1}(\bx)\\
    =&\;\sum_{k_3=0}^\infty \beta_{k_3}\sum_{\ell=0}^{\min(k_1,k_2)}L_{k_1,k_2}^{(p)}(\ell)\int_{\Spp}C_{k_1+k_2-2\ell}^{(p-2)/2}(\bmu^\top \bx)C_{k_3}^{(p-2)/2}(\bmu^\top \bx)\,\rd\nu_{p-1}(\bx)\\
    =&\;\sum_{\ell=0}^{\min(k_1,k_2)} \beta_{k_1+k_2-2\ell}L_{k_1,k_2}^{(p)}(\ell)
    \gamma_{k_1+k_2-2\ell,p}C_{k_1+k_2-2\ell}^{(p-2)/2}(1).
\end{align*}

For the second term in $\sigma^2$, we exploit orthogonality of the spherical harmonics via Parseval's identity, which yields the representation in \eqref{eq:tau_parseval}. Hence,
\begin{align*}
    \int_{\Spp}\int_{\Spp}z(\bs)z(\bt)z(\bs)z(\bt)\,\rd\nu_{p-1}(\bs)\,\rd\nu_{p-1}(\bt)=\|z\|_{L^2(\Spp)}^4=\bigg(\sum_{k=0}^\infty\sum_{r=1}^{d_{k,p}}\beta_{r,k}^2\gamma_{k,p}c_{k,p}(\lambda)\bigg)^2,
\end{align*}
which completes the derivation of the result.
\end{proof}

\begin{remark}
For the product of two Gegenbauer or Chebyshev polynomials, evaluated at the same argument $y\in[-1,1]$, we get
\begin{align}
    C_{k_1}^{(p-2)/2}(\bs^\top \bx)C_{k_2}^{(p-2)/2}(\bs^\top \bx)
    =\sum_{\ell=0}^{\min(k_1,k_2)}L_{k_1,k_2}^{(p)}(\ell)C_{k_1+k_2-2\ell}^{(p-2)/2}(\bs^\top \bx)\quad \textit{for all } \bs,\bx\in\Spp,\label{eq:linearization}
\end{align}
where for $p\geq 3$ we apply the linearization formula \citep[18.18.22]{NIST:DLMF} with
\begin{align*}
    L_{k_1,k_2}^{(p)}(\ell)=&\;\frac{(k_1+k_2+(p-2)/2-2\ell)(k_1+k_2-2\ell)!}{(k_1+k_2+(p-2)/2-\ell)\ell!(k_1-\ell)!(k_2-\ell)!}\\
    &\times\frac{((p-2)/2)_{\ell}((p-2)/2)_{k_1-\ell}((p-2)/2)_{k_2-\ell}(p-2)_{k_1+k_2-\ell}}{((p-2)/2)_{k_1+k_2-\ell}(p-2)_{k_1+k_2-2\ell}}
\end{align*}
and for $p=2$ \citep[18.18.21]{NIST:DLMF} $L_{k_1,k_2}^{(2)}(\ell)=\frac{1}{2}(1_{\{\ell=0\}}+1_{\{\ell=\min(k_1,k_2)\}})$.
\end{remark}

\noindent \begin{proof}{ of Proposition~\ref{prop:H0_func}}
For any fixed compact interval $[a,b]\subset(0,\infty)$, weak convergence in the continuous functions $C([a,b])$ with the supremum norm follows from finite-dimensional convergence and tightness \citep[Theorem 14.25]{henze2024}. 
First, let $M\in\N$, $\lambda_1,\ldots,\lambda_M\in(0,\infty)$, and any $a_1,\ldots,a_M\in \R$. By the absolute convergence
\begin{align*}
    \sum_{m=1}^M a_m T_n(\lambda_m)=\sum_{k=1}^{\infty}\sum_{m=1}^M a_m c_{k,p}(\lambda_m)A_k\xrightarrow[]{d} \sum_{m=1}^M a_m T_{\infty}(\lambda_m),
\end{align*}
which can be seen as a Sobolev statistic with signed absolutely summable coefficients $\sum_{m=1}^M a_m c_{k,p}(\lambda_m)$, the Cramér--Wold device \citep[Theorem 6.18]{henze2024} yields finite-dimensional convergence. To prove tightness in $C([a,b])$, we use the modulus of continuity criterion
\begin{align*}
    \lim_{\delta\to 0}\limsup_{n\to \infty} \Prob{w(T_n,\delta)\ge \epsilon}=0 \quad \text{for all }\epsilon>0,\quad \text{ where }\quad w(f,\delta):=\sup_{|\lambda-\lambda'|\le \delta}|f(\lambda)-f(\lambda')|.
\end{align*}
Using \citet[10.29.2]{NIST:DLMF}, for $a\le \lambda\le b$, yields
\[
    |c'_{k,p}(\lambda)|\le C'\lrp{k^5 \Ical_{k+(p-2)/2}(b)\Ical_{k+p/2}(b)
    +k^6 \Ical_{k+(p-2)/2}(b)^2}.
\]
Hence, by \citet[10.41.1]{NIST:DLMF}, 
\[
    \sup_{a\le \lambda\le b}|c'_{k,p}(\lambda)|
    \le C_{a,b,p}\,k^5\left(eb/(2k+p-2)\right)^{2k+p-2}.
\]
Furthermore, since $\mathbb{E}|A_k|=\mathbb{E}[A_k]=C_k^{(p-2)/2}(1)$, and since $C_k^{(p-2)/2}(1)$ grows polynomially in $k$, we obtain $\sum_{k=1}^{\infty}\sup_{a\le\lambda\le b}|c'_{k,p}(\lambda)|\mathbb{E}|A_k|<\infty$.
By the mean value theorem, it follows that
\begin{align*}
    w(T_{n},\delta)\le \sup_{|s-t|\le \delta}\sum_{k=1}^{\infty}|c_{k,p}(s)-c_{k,p}(t)||A_k|\le \delta\sum_{k=1}^{\infty} \sup_{a \le \lambda \le b}|c'_{k,p}(\lambda)||A_k|.
\end{align*}
Therefore, by Markov's inequality,
\begin{align*}
    \Prob{w(T_n,\delta)\ge \epsilon}\le \frac{\delta}{\epsilon}\sum_{k=1}^{\infty}\sup_{a\le\lambda\le b}|c'_{k,p}(\lambda)|\mathbb{E}|A_k|\to 0,\qquad \delta\to 0,
\end{align*}
so that $(T_n|_{[a,b]})_n$ is tight in $C([a,b])$ and $T_n|_{[a,b]} \rightarrow T_\infty|_{[a,b]}$ in $(C([a,b]),\|\cdot\|_\infty)$. 

Finally, since $(0,\infty)=\cup_{i\in\N}[1/i,i]$, \citet[Theorem 1.6.1]{VanDerVaart2023} implies that $T_n \rightarrow T_\infty$ in $C(0,\infty)$ equipped with the metric $d(z_1,z_2)=\sum_{i=1}^\infty2^{-i}\lrp{\sup_{\lambda\in[1/i,i]}|z_1(\lambda)-z_2(\lambda)|}$.
\end{proof}

\noindent \begin{proof}{ of Proposition~\ref{prop:limit}}
We first consider the case $\lambda\to 0$.
By the small argument approximation $\Ical_k(a)\approx (a/2)^k/\Gamma(k+1)$ for $a\to 0$ in \citet[10.30.1]{NIST:DLMF} and the harmonic decomposition \eqref{eq:cdkreal}, there exist constants $C_{k,p},C_{k,p}^*>0$ depending only on $p\geq 2$, $k\geq 1$ so that
$\lambda^{-2}c_{k,p}(\lambda)=C_{k,p}\lambda^{-p}\Ical_{(p-2)/2+k}(\lambda)^2\approx C_{k,p}^*\lambda^{2k-2}, \lambda\to 0$. Now, clearly $
    C_{k,p}^*\lambda^{2k-2}\to C_{k,p}^*1_{\{k=1\}},
$ for $\lambda\to 0$
implies that the limit is equivalent to the \cite{Rayleigh1919} test, since
\begin{align*}
    \lim_{\lambda\to 0}\frac{T_n(\lambda)}{\lambda^2}=\frac{1}{n}\sum_{i,j=1}^nC_{1,p}^*C_1^{(p-2)/2}(\bX_i^\top \bX_j)\propto \frac{1}{n}\sum_{i,j=1}^n\bX_i^\top \bX_j,
\end{align*}
which is the Rayleigh test up to a multiplicative constant.

For the case $\lambda\to\infty$, evaluating the integral in \eqref{eq:Tn(lambda)} in terms of the coefficient $ m_{0,p}(\lambda\|\bX_i+\bX_j\|)$ from \eqref{eq:m_k}, yields
\begin{align*}
    T_n(\lambda)=&\;\frac{1}{n}\bigg\|\sum_{i=1}^n\Delta_{\Spp}e^{\lambda\bt^\top\bX_i}\bigg\|^2_{L^2(\Spp)}
    =\frac{1}{n}\sum_{i,j=1}^n\Delta_{\Spp,\bX_i} \Delta_{\Spp,\bX_j}
    \int _{\Spp}e^{\lambda\bt^\top(\bX_i+\bX_j)}\, \rd\nu_{p-1}(\bt)\\
=&\;\frac{1}{n}\sum_{i,j=1}^n\Delta_{\Spp,\bX_i} \Delta_{\Spp,\bX_j}
m_{0,p}(\lambda\|\bX_i+\bX_j\|).
\end{align*}
Here, we use the zonal structure resulting in $\Delta_{\Spp,\bt}e^{\lambda\bt^\top\bX_i}=\Delta_{\Spp,\bX_i}e^{\lambda\bt^\top\bX_i}$. For large $\lambda$, we use the Bessel function approximation $\Ical_k(a)\approx e^a/\sqrt{2\pi a}$ as $a\to\infty$ from \citet[10.30.4]{NIST:DLMF} to get for $p\ge 3$ that
\begin{align*}
    \Delta_{\Spp,\bX_i} &\Delta_{\Spp,\bX_j}
m_{0,p}(\lambda\|\bX_i+\bX_j\|)\\
=&\;\Delta_{\Spp,\bX_i} \Delta_{\Spp,\bX_j}
\bigg(\frac{2}{\lambda\|\bX_i+\bX_j\|}\bigg)^{(p-2)/2}\Gamma\left(\frac{p-2}{2}\right) \left(\frac{p-2}{2}\right)\Ical_{(p-2)/2}(\lambda\|\bX_i+\bX_j\|) \\
\approx &\;p_\lambda(\|\bX_i+\bX_j\|)e^{\lambda\|\bX_i+\bX_j\|}.
\end{align*}
In the case $p=2$, the expression simplifies to
\begin{align*}
    \Delta_{\mathcal{S}^{1},\bX_i} \Delta_{\mathcal{S}^{1},\bX_j}
m_{0,2}(\lambda\|\bX_i+\bX_j\|)=\Delta_{\mathcal{S}^{1},\bX_i} \Delta_{\mathcal{S}^{1},\bX_j} \Ical_{0}(\lambda\|\bX_i+\bX_j\|)\approx p_\lambda(\|\bX_i+\bX_j\|)e^{\lambda\|\bX_i+\bX_j\|}.
\end{align*}
For each fixed $u>0$, $p_{\lambda}(u)$ is at most polynomial in $\lambda$, hence $\log(p_\lambda)/\lambda\to 0$ for $\lambda\to \infty$. Therefore,
\begin{align*}
    \frac{1}{\lambda}\log \big(T_n(\lambda)-D_n(\lambda)\big)&=
    \frac{1}{\lambda}\log\bigg( 2\sum_{i<j}
    \exp\big( \log\big(p_\lambda(\|\bX_i+\bX_j\|)\big) + {\lambda\|\bX_i+\bX_j\|}\big)\bigg)\\
    &=\frac{1}{\lambda}\log\bigg(\sum_{i<j}
    \exp\bigg(\lambda \bigg(\frac{\log\big(p_\lambda(\|\bX_i+\bX_j\|)\big)}{\lambda} + {\|\bX_i+\bX_j\|}\bigg)\bigg)\bigg)+\frac{\log 2}{\lambda}\\
    &\to \max_{i<j}\|\bX_i+\bX_j\|, \quad \lambda \to \infty.
\end{align*}
In the last step, we use the convergence of the LogSumExp to the maximum, 
\[
\lim_{a\to\infty}a^{-1}\log\big(\sum_{i=1}^n\exp(ax_i)\big)=\max_{1\leq i\leq n} x_i.
\]
Due to the monotone nature of the transformations $t\mapsto \lambda^{-1}\log\big(t-D_n(\lambda)\big)$ and $t\mapsto \sqrt{2+2t}$, the resulting rejection rule is equivalent to that based on $\max_{i<j}\bX_i^\top\bX_j$, as in \cite{Cai2013}.
\end{proof}
\end{document}